\documentclass[10pt,a4paper]{article}

\usepackage[utf8]{inputenc}
\usepackage[english]{babel}

\usepackage{amsmath,amssymb,mathrsfs,latexsym,xcolor}
\usepackage{authblk}
\usepackage{comment}
\usepackage{stmaryrd}
\usepackage[shortlabels]{enumitem}
\usepackage{hyperref}
\usepackage[nameinlink,capitalize]{cleveref}

\usepackage{amsmath,amssymb,amsthm,mathtools}
\usepackage{bm}
\usepackage{stmaryrd}

\usepackage{xcolor}
\usepackage{comment}

\usepackage{enumitem}
\usepackage{fancyhdr}
\usepackage[onehalfspacing]{setspace}
\usepackage{semantic}
\usepackage{float}
\usepackage{cite}
\usepackage{url}
\usepackage[ruled,vlined]{algorithm2e}
\usepackage[justification=centering]{caption}
\usepackage{graphicx}
\usepackage{rotating}

\usepackage[T1]{fontenc}
 \usepackage{dsfont}

\usepackage{hyperref}

\usepackage[textwidth=2cm]{todonotes}
\setuptodonotes{size=\tiny}

\usepackage{authblk}

\usepackage[
  margin=1in
]{geometry}

\newcommand{\normmm}[1]{{\left\vert\kern-0.25ex\left\vert\kern-0.25ex\left\vert #1 
   \right\vert\kern-0.25ex\right\vert\kern-0.25ex\right\vert}}

\newcommand{\R}{\mathbb{R}}
\newcommand{\E}{\mathbb{E}}

\newcommand{\be}{\begin{equation}}
\newcommand{\ee}{\end{equation}}

\newcommand{\MF}{{\mathcal{F}}}
\newcommand{\MA}{{\mathcal{A}}}
\newcommand{\MZ}{{\mathcal{Z}}}
\newcommand{\MS}{{\mathcal{S}}}
\newcommand{\MH}{{\mathcal{H}}}

\newcommand{\bL}{\bar{L}}
\newcommand{\bC}{\bar{C}}
\newcommand{\bR}{\bar{R}}

\newcommand{\hC}{\hat{C}}

\newcommand{\vep}{{\varepsilon}}

\newcommand{\Balpha}{\boldsymbol{\alpha}}

\newcommand{\hBalpha}{\hat{\Balpha}}

\newcommand{\Bx}{{\mathbf{x}}}

\newcommand{\halpha}{{\hat{\alpha}}}

\theoremstyle{plain}
\newtheorem{Theorem}{Theorem}[section]
\newtheorem{Lemma}[Theorem]{Lemma}
\newtheorem{Corollary}[Theorem]{Corollary}
\newtheorem{Remark}[Theorem]{Remark}
\newtheorem{example}[Theorem]{Example}
\newtheorem{Definition}[Theorem]{Definition}
\newtheorem{Assumption}[Theorem]{Assumption}

\theoremstyle{definition}

\numberwithin{equation}{section}

\newcommand{\lps}[2]{{}^{#1}\!#2}

\newcommand{\trinm}[1]{\vert\!\vert\!\vert #1 \vert\!\vert\!\vert}

\begin{document}

\title{Multidimensional quadratic BSDEs with weak interactions and their applications in mean-field games of controls}

 \author[1]{Ulrich Horst\thanks{Email: ulrich.horst@hu-berlin.de}}
 \author[3]{Emil Schmidek\thanks{Email: emil.johannes.cyrill.schmidek.1@hu-berlin.de}}
\author[2,3]{Huilin Zhang\thanks{Email: huilinzhang@sdu.edu.cn}}

\affil[1]{\small Department of Mathematics, and School of Business and Economics, Humboldt University Berlin}
\affil[2]{\small Research Center for Mathematics and Interdisciplinary Sciences, Frontiers Science Center for Nonlinear Expectations (Ministry of Education), Shandong University}
\affil[3]{\small Department of Mathematics, Humboldt University Berlin}





\maketitle

\begin{abstract}



The well-posedness of multidimensional quadratic backward stochastic differential equations (qBSDEs) remains one of the central open problems in BSDE theory.
Motivated by a mean-field utility maximization model with price impact, we introduce a new class of multidimensional qBSDEs that lies beyond the scope of existing well-posedness results
(see \eqref{equ:Def_f} for the generator). In order to study the limit as the number of players tends to infinity, we establish existence and uniqueness for a large class of such qBSDEs under suitable smallness conditions imposed on each individual dynamics. A key feature of our approach is that the smallness condition is independent of the dimension of the BSDE system. In particular, the system itself is not confined to a small neighborhood, which allows us to analyze the mean-field limit of the underlying utility maximization problem. Such a condition is also natural in view of the well-known fact that general multidimensional qBSDEs may fail to admit solutions in the absence of suitable smallness assumptions.
In addition, we derive a stability result for this class of equations based on the application of Picard iterations. Finally, using this stability result, we establish quantitative convergence rates toward the corresponding mean-field equilibria in two settings: Nash and Radner equilibria.\\

\end{abstract}


\textbf{MSC2020 subject classifications: 93E20, 60H30, 91A16}


\section{Introduction}
\label{sec1}

In this paper, we study a novel class of multi-dimension quadratic backward stochastic differential equations (qBSDEs) arising in continuous-time equilibrium (Nash and Radner) models of incomplete financial markets. We start with a multiplayer equilibrium problem with stochastic parameters and price impact, where each player's action affects the drift of all underlying assets through the aggregate behavior of the entire population. Price impact models are used in various settings in the mathematical finance literature and arise naturally in problems such as optimal execution, liquidation, order book modeling, asset pricing, and utility maximization; see, for example, \cite{AC2001,Car2007,HXZ-2020, FHX1}. We consider an incomplete market model where price impact is described through the arithmetic mean of the players' controls, reflecting the fact that each individual player has only a negligible influence on asset prices, while the collective actions of all market participants determine the overall market impact. By formally applying the classical martingale approach to utility maximization problems (see, e.g. \cite{HIM05}), we show that the corresponding Nash equilibrium can be characterized by a multidimensional qBSDE. 
It turns out that the same class of multidimensional qBSDEs also appears in a Radner-type equilibrium model where asset prices are determined by a market-clearing condition. We study the well-posedness of the resulting multidimensional qBSDE system in the finite-player setting under a weak interaction condition and establish the convergence of finite-player equilibria to the corresponding mean-field limit. 

\subsection{Multidimensional qBSDEs: state of the art and new challenges}


BSDEs were first introduced by Bismut \cite{Bis73} in the linear setting and later generalized to the nonlinear framework in the seminal work of Pardoux and Peng \cite{PP90}. Since then, they have become a fundamental tool in stochastic control, mathematical finance, stochastic differential games, and, more recently, machine learning. Despite their broad applicability, the well-posedness theory for multidimensional qBSDEs remains largely open; see Peng \cite{Peng1999}, Espinosa and Touzi \cite{ET2015}, and Jackson \cite{Jack23}. A major breakthrough in the one-dimensional setting was achieved by Kobylanski \cite{koby00}, who established existence and uniqueness under bounded terminal conditions. Subsequent works extended the theory to more general generators and unbounded terminal conditions; see, among others, Briand and Hu \cite{BriandHu2006PTRF,BriandHu2008PTRF}, Briand and Elie \cite{BriandElie2013SPA}, Delbaen et al. \cite{DelbaenHuRichou2011AIHPPS,DelbaenHuRichou2015DCD}, Barrieu and El Karoui \cite{BarrieuElKaroui2013AoP}, and Tevzadze \cite{Tev08}. 

The multidimensional setting is substantially more delicate. Cheridito and Nam \cite{cheridito2015Stochastics} studied several special classes of multidimensional qBSDEs. In the Markovian framework, Xing and \v{Z}itkovi\'{c} \cite{XZ18} established a general solvability theory based on the so-called (BF) and (AB) structural conditions, which restrict the quadratic growth of the generator along prescribed directions. Their results were later extended to the non-Markovian setting by Jackson and \v{Z}itkovi\'{c} \cite{JZ22} and Jackson \cite{Jack23}. Another important line of research originates from the work of Hu and Tang \cite{HT16}, who established well-posedness under the so-called diagonally quadratic condition, where the \(i\)-th component \(f^i\) of the generator depends quadratically only on the \(i\)-th row \(z^i\) of the matrix variable \(z\). This structural assumption was later generalized by Luo \cite{luo20} to the triangularly quadratic case, in which \(f^i\) may depend quadratically on \(z^j\) for \(j \le i\). These ideas further motivated subsequent developments such as \cite{JZ22}, \cite{FanHuTang2023JDE}, and the recent work \cite{FHT25}.

Although the above works are quite general within their respective frameworks and include many interesting examples from stochastic control and game theory, they do not apply to the mean-field type qBSDEs arising in mean-field games. First, because the interaction term depends on the empirical average of all players’ controls, the generator \(f^i\) defined in \eqref{equ:Def_f} depends quadratically on all components \(z^j\), \(j=1,\dots,N\). Consequently, the generator is neither triangularly quadratic nor quadratic-linear in the sense of \cite{Jack23}. Second, since our objective is to study the limit as the number of players \(N\) tends to infinity, the quadratic structure cannot be restricted to finitely many prescribed directions and therefore the framework of the (AB) condition introduced in \cite{XZ18} is not applicable. Third, because the dimension of the system itself grows with \(N\), the standard estimates available for multidimensional qBSDEs are typically not robust with respect to the dimension. For instance, the smallness condition appearing in \cite{Tev08} degenerates as \(N \to \infty\), leading to exploding estimates and preventing one from passing to the mean-field limit. The well-posedness of such multidimensional qBSDEs therefore remains largely open.


Since general multidimensional qBSDEs may fail to admit solutions (see Frei and Dos Reis \cite{FreiDosReis2011MFE}), inspired by \cite{Tev08}, we impose suitable smallness assumptions either on the terminal condition or on the quadratic part of the generator, which we refer to as qBSDEs with weak interaction. Furthermore, in order to obtain estimates that remain stable as \(N \to \infty\), we work with the maximum norm across dimensions rather than the classical Frobenius norm. Our key observation is that, under different scaling assumptions on the pure quadratic terms $|z^i|^2$ and the cross-quadratic terms ($|z^i \cdot z^j|, i \neq j$) (see Assumption \ref{assu:well-posedBSDE}), one can establish a contraction mapping argument on a suitable ball, yielding well-posedness under a dimension-free smallness condition. The latter is essential for passing to the mean-field limit. A typical example covered by our theorem is an $N$-dimensional BSDE driven by a one-dimensional Brownian motion and the driver
$
f:\mathbb{R}^N \to \mathbb{R}^N,
\ 
f(z^{(1)}, z^{(2)}, ..., z^{(N)})=\big(\vep (z^{(N)})^2,0, ..., 0\big),
$
for some $\vep$ small enough (depending on the terminal of the qBSDE). It is straightforward to check that the above driver is neither triangularly quadratic nor quadratic-linear, and is not covered by the (AB) condition. Moreover, the well-posedness result obtained in this paper remains valid uniformly in \(N\), and is therefore robust with respect to the mean-field limit as \(N\to\infty\).

\subsection{Equilibrium dynamics in incomplete markets}

Our existence and uniqueness of solutions results are applied to establish existence of equilibrium results in two classes of dynamically incomplete market models: Nash equilibria in mean-field portfolio games with price impact and continuous-time Radner equalibria in incomplete markets.  

\subsubsection{Mean-field portfolio games}

The foundational ideas of mean-field game theory were introduced independently in the pioneering works of Lasry and Lions \cite{LasryLions2007} and Huang, Malham\'e, and Caines \cite{HMC2006}. Mean-field games arise naturally in systems with a large number of interacting agents and have found important applications in economics and mathematical finance, including systemic risk, market impact, portfolio optimization, and trading under relative performance concerns.
Research on mean-field games mainly focuses on two types of questions. The first concerns the characterization of the limiting equilibrium through measure-dependent equations, while the second concerns the convergence of finite-player equilibria toward the mean-field equilibrium as the number of players tends to infinity.
For the first class of problems, two main approaches are available. The first is the PDE approach, where one typically works with Markovian controls and characterizes the equilibrium through a second-order PDE on the Wasserstein space, known as the {\it master equation}; see, among many others, Cardaliaguet et al. \cite{cardaliaguet2019master} and Gangbo et al. \cite{GangboMeszarosMouZhang2022}. The second is the probabilistic approach, where the equilibrium is characterized through McKean--Vlasov forward-backward stochastic differential equations (FBSDEs) via a stochastic maximum principle; see, for example, Bensoussan, Yam, and Zhang \cite{Ben-Yam-Zhang15}, Carmona and Delarue \cite{carmona-delarue13}, Buckdahn et al. \cite{Buck-Dje-Li-Peng09}, and Lauri\`ere and Tangpi \cite{lauriere-tanpi22}.

Concerning the convergence problem, important breakthroughs for large classes of mean-field games were obtained by Lacker \cite{LackerAAP} and Fischer \cite{Fischer17}. Using the master equation approach, Cardaliaguet et al. \cite{cardaliaguet2019master} established convergence results for closed-loop controls; see also Delarue, Lacker, and Ramanan \cite{Del-Lac-Ram-19,Del-Lac-Ram_Concent} for this direction. On the probabilistic side, convergence results based on propagation of chaos techniques for FBSDEs were established by Lauri\`ere and Tangpi \cite{lauriere-tanpi22} and Possama\"i and Tangpi \cite{PossamaiTangpi2025}, under a strong and weak formulation, respectively. In our setting, the interaction is through the empirical average of controls, and such models are now referred to as mean-field games of controls, or extended mean-field games, in the literature. The deterministic setting was first studied in Gomes, Patrizi, and Voskanyan \cite{RefN7} and Gomes and Voskanyan \cite{RefN8}, while more general stochastic formulations were investigated in Carmona and Lacker \cite{RefN2} and further developed in Possama\"i and Tangpi \cite{PossamaiTangpi2025}. We refer to the monographs of Carmona and Delarue \cite{CarmonaDelarue2018Vol1,CarmonaDelarue2018Vol2} and Cardaliaguet and Porretta \cite{Ref7} for comprehensive introductions to this rapidly developing theory.
 
Several works related to price impact mean-field games have recently appeared. Carmona and Lacker \cite{RefN2} studied a weak formulation of such models and established existence and uniqueness of equilibria under compact-valued admissible controls, in contrast to our strong formulation and the absence of compactness restrictions on the control set. In their framework, the equilibrium is characterized by BSDEs with Lipschitz generators, rather than the quadratic structure considered in the present work. A strong formulation was investigated by Cardaliaguet and Lehalle \cite{RefN1} in the case of deterministic coefficients. This was later generalized by Djete \cite{RefN3}, allowing random coefficients depending on both the states and controls of the players, at the cost of additional boundedness assumptions and tightness arguments, which are not satisfied in our setting. In B\"auerle and G\"oll \cite{Ref1}, a related model with relative performance concerns was considered. Since their market model is complete and involves only one risky asset, the equilibrium can be characterized through classical Hamilton--Jacobi--Bellman equations, which reduce to backward linear ODEs under deterministic coefficients. Moreover, the focus there is on finite-player Nash equilibria rather than the mean-field limit. Fu and co-authors \cite{FGHP-2018,FHH1, FHH2} considered mean-field games of optimal portfolio liquidation under market impact. In these models, equilibria can be described in terms of (F)BSDEs with singular terminal condition.    

\subsubsection{Radner equilibria in dynamically incomplete markets}

{Our second benchmark model concerns equilibrium pricing in a mean-field setting, where the equilibrium consists jointly of an asset price process and the optimal trading strategies of all agents. Such equilibria are commonly referred to as \emph{Radner equilibria}, following the seminal works of Radner \cite{Rad68,Rad72}, and can be viewed as a dynamic extension of the classical Arrow--Debreu framework; see, for example, Cheridito et al. \cite{CHKP16}. In continuous-time models, equilibrium pricing has been studied by Horst and M\"uller \cite{HM2007Spanning}, Escauriaza, Schwarz, and Xing \cite{EscauriazaSchwarzXing2022AAP} and Kardaras, Xing, and \v{Z}itkovi\'{c} \cite{Kardaras2022} among others. In contrast to our setting, the sources of uncertainty in these works are driven by a fixed finite-dimensional Brownian motion. In our model, each agent possesses an idiosyncratic source of risk, so that the dimension of the underlying Brownian motion grows with the number of agents and becomes unbounded in the mean-field limit. More recently, Fujii and Sekine \cite{RefN10,RefN9} investigated Radner-type equilibria in the mean-field limit and established asymptotic market-clearing conditions. However, due to the absence of a well-posedness theory for the underlying multidimensional qBSDE systems, the corresponding models were treated essentially as mean-field control problems rather than genuine mean-field games.

}

\subsection{ Contributions and organization of the paper}

We achieve the above objectives in several steps. First, we establish a well-posedness result for the new class of multidimensional qBSDEs under a smallness condition independent of the dimension. Under the tailor-made Assumption \ref{assu:well-posedBSDE}, we show that our framework applies to the mean-field price impact models introduced in Section \ref{sec2}; see Corollary \ref{Cor3.4}. Second, we derive a stability result for this class of qBSDEs (Theorem \ref{Thm3.9}),  which is of independent interest, as stability may in some sense be viewed as closely related to well-posedness. This result plays a central role in establishing backward propagation of chaos, namely the convergence of the multidimensional qBSDE system towards its mean-field limit. A major technical difficulty is that backward propagation of chaos under the BMO norm --- arguably the most natural norm for qBSDEs --- appears to be out of reach for equations of our type. Furthermore, the standard approach for deriving stability estimates, namely taking the difference of two BSDEs and treating the resulting equation as a linear BSDE with BMO coefficients (see, e.g. \cite{Jack23}), is not applicable in the mean-field setting, in the sense that assumptions would be needed that are not dimension-free. Indeed, the BMO norm of the resulting matrix-valued coefficient typically grows with the dimension of the system, which in turn causes the corresponding estimates for the solutions to blow up as the dimension tends to infinity. We  establish convergence in the \(\MH^2\)-norm and combine this estimate with a Picard iteration argument, leading to a convergence rate of order \((\log N)^{-\alpha}\); see Theorem \ref{Thm5.3}. Finally, we introduce the limiting mean-field qBSDE and prove its well-posedness under the same smallness assumptions. We then apply the stability result to prove convergence of the finite-player equilibria toward the mean-field equilibria, both in the Nash and Radner equilibrium framework.

In summary, our contributions are fourfold.
First, to the best of our knowledge, we are the first to establish a well-posedness theory for this new class of multidimensional qBSDEs arising from mean-field games of controls. Our framework is sufficiently flexible to cover the models considered in this paper and can be extended to a broad range of mean-field interaction models.
Second, we develop a stability theory for these multidimensional qBSDEs that is specifically tailored to backward propagation of chaos arguments. This result is of independent interest and serves as a key ingredient in our convergence analysis.
Third, we provide a rigorous equilibrium theory for utility maximization models with price impact through multidimensional qBSDE systems. In particular, we establish the convergence of finite-player equilibria toward their mean-field counterparts and derive quantitative convergence rates in both the Nash and Radner equilibrium frameworks. Fourth, to the best of our knowledge, we are the first to formulate a multi-dimensional Radner equilibrium model in the mean-field setting and to show the convergence of finite-player Radner equilibria toward the corresponding mean-field limit.

The paper is organized as follows. Section \ref{sec2} introduces the multiplayer utility maximization model with price impact and derives the associated multidimensional qBSDE system. In Section \ref{sec3}, we present our main well-posedness and stability results and then apply the well-posedness result to the benchmark model. Section \ref{sec4.1} is devoted to the limiting mean-field qBSDE and its well-posedness. We then prove the backward propagation of chaos and convergence of equilibria in the Nash equilibrium framework. Finally, in Section \ref{sec:4.2}, we briefly discuss the analogous results in the Radner equilibrium setting.\\

{\bf Acknowledgement.}
The authors gratefully acknowledge financial support from DFG CRC/TRR 190 - Project ID 280092119 and DFG CRC/TRR 388 “Rough Analysis, Stochastic Dynamics and Related Fields” - Project ID 516748464.
HZ is partially supported by the Fundamental Research Funds for the Central Universities, NSF of Shandong (ZR2023MA026).

\subsection{Notations and preliminaries}
\label{sec1.1}

Let $(\Omega,\mathcal{F}, \mathbb{P})$ be a complete probability space supporting an $(N+1)$-dimensional Brownian motion
$W = (W^{0},W^{1},...,W^{N})$ with $W^{0}$ as the common noise and $W^{i}$ as the idiosyncratic noise. Fix a maturity time $T>0$. Denote by $(\mathcal{F}'_{t} )_{t \in [0,T]}$ the augmented filtration generated by $(W^{1},...,W^{N})$, $(\mathcal{F}_{t}^{i})_{t \in [0,T]}$ the one augmented from the raw filtration generated by $W^{i}$ for $i \in \{0,...,N\}$, respectively, and $\mathbb{F}=(\MF_t)_{t\in[0,T]}$ the completion of $(\mathcal{F}'_{t}  \otimes \mathcal{F}_{t}^{0})_{t \in [0,T]}$. For an $\Bx \in \mathbb{R}^{N}$ we write $\Bx^{-i} := (x^{1},...,x^{i-1},x^{i+1},...,x^{N}) \in \mathbb{R}^{N-1}$ and $(\Bx^{-i},y) := (x^{1},...,x^{i-1},y,x^{i+1},...,x^{N})$ for $y \in \mathbb{R}$.
We write $|\cdot|$ for the Euclidean norm, and in particular, for $z=(z_{ij}) \in \mathbb{R}^{n \times m}$ we have $|z|^{2} = \sum_{i,j}z_{ij}^{2}$. 

Define $\mathcal{T}$ as the set of all stopping times taking values in $[0,T]$ with respect to the filtration $\mathbb{F}$. For (multidimensional) $(\MF_t)$-progressively measurable processes $Y$ and $Z$, let 
\[
\|Z\|_{\mathcal{H}^{2}}^{2} := \mathbb{E}\Bigl[\int_0^T |Z_r|^2\,dr\Bigr],
\|Z\|_{\mathcal{Z}^{2}_{BMO}}^{2} := \operatorname*{ess\,sup}_{(\omega,\tau)\in\Omega\times\mathcal{T}} \mathbb{E}_{\tau}\Bigl[\int_{\tau}^{T}|Z_r|^2\,dr\Bigr],
\|Y\|_{\mathcal{S}^{\infty}} := \operatorname*{ess\,sup}_{(\omega,t)\in\Omega\times[0,T]} |Y_t(\omega)|,
\]
where we use the short notation $\mathbb{E}_{t}[Y_{u}] := \mathbb{E}[Y_{u}|\mathcal{F}_{t}]$ for the conditional expectation. These norms allow us to define the spaces
\begin{align*}
\mathcal{H}^{2} &:= \bigl\{ Z: [0,T]\times \Omega \rightarrow \mathbb{R}^{n \times m} | \text{ $Z$ is predictable and } ||Z||_{\mathcal{H}^{2}} < \infty\bigr\}, \\ 
\mathcal{Z}^{2}_{BMO} &:= \bigl\{ Z: [0,T]\times \Omega \rightarrow \mathbb{R}^{n \times m} | \text{ $Z$ is predictable and } ||Z||_{\mathcal{Z}^{2}_{BMO}} < \infty\bigr\},
\end{align*}
and
\[
\mathcal{S}^{\infty} := \bigl\{ Y: [0,T]\times \Omega \rightarrow \mathbb{R}^{n} | \text{ $Y$ is progressively measurable, continuous and } ||Y||_{\mathcal{S}^{\infty}} < \infty \bigr\},
\]
where we abuse notation in the dimension of the processes. 

 
The above defined BMO-space $\mathcal{Z}_{BMO}^{2}$ is very important for the whole paper. Therefore, we introduce the BMO-theory that we need later. For details and the full BMO-theory, see \cite{Ref3}. For now, let $\mathcal{Z}_{BMO}^{2}$ denote the space of scalar-valued processes, and $B$ the one-dimensional Brownian motion that generates the filtration.
Then we have the following well-known energy inequalities (see \cite[Lemma 9.6.5]{RefN11}): 
for any $Z \in \mathcal{Z}_{BMO}^{2}$, $n \in \mathbb{N}$ and stopping times $\tau$ in $[0,T]$, it holds
\begin{equation}
\label{equ:BMO.1}
\mathbb{E}_{\tau}\Bigl[\Bigl(\int_{\tau}^{T}Z_{r}^{2}dr\Bigr)^{n}\Bigr] \leq n!||Z||_{\mathcal{Z}_{BMO}^{2}}^{2n}.
\end{equation}
Now, we get to the important inequalities that we use later. The following Lemma is a standard result that is stated for the convenience of the reader; for the proof, see e.g. \cite[Theorem 2.2]{Ref3}.

\begin{Lemma}
\label{Lemma2.3}
Let $Z \in \mathcal{Z}_{BMO}^{2}$ with $||Z||_{\mathcal{Z}_{BMO}^{2}} < 1$. Then for all stopping times $\tau$ in $[0,T]$ holds
\begin{equation}
\label{equ:BMO.2}
\mathbb{E}_{\tau}\Bigl[\exp\Bigl(\int_{\tau}^{T}Z_{r}^{2}dr\Bigr)\Bigr] \leq \big(1-||Z||_{\mathcal{Z}_{BMO}^{2}}^{2}\big)^{-1}.
\end{equation}
\end{Lemma}

\begin{Corollary}
\label{Cor2.4}
For $Z \in \mathcal{Z}_{BMO}^{2}$ with $||Z||_{\mathcal{Z}_{BMO}^{2}}^{2} < \tfrac{1}{2}$ holds
\begin{equation}
\label{equ:BMO.3}
\mathbb{E}\Bigl[\exp\Bigl(\int_{0}^{T}Z_{r}dB_{r}\Bigr)\Bigr] \leq \big(1-2||Z||_{\mathcal{Z}_{BMO}^{2}}^{2}\big)^{-\tfrac{1}{2}}.
\end{equation}
\end{Corollary}

\begin{proof}

We can compute
\begin{align*}
\mathbb{E}\Bigl[\exp\Bigl(\int_{0}^{T}Z_{r}dB_{r}\Bigr)\Bigr] &= \mathbb{E}\Bigl[\exp\Bigl(\int_{0}^{T}Z_{r}dB_{r} - \int_{0}^{T}Z_{r}^{2}dr\Bigr) \cdot \exp\Bigl(\int_{0}^{T}Z_{r}^{2}dr\Bigr)\Bigr]\\
&\leq \biggl(\mathbb{E}\Bigl[\Bigl(\exp\Bigl(\int_{0}^{T}Z_{r}dB_{r} - \int_{0}^{T}Z_{r}^{2}dr\Bigr)\Bigr)^{2}\Bigr]\biggr)^{\tfrac{1}{2}} \cdot \biggl(\mathbb{E}\Bigl[\Bigl(\exp\Bigl(\int_{0}^{T}Z_{r}^{2}dr\Bigr)\Bigr)^{2}\Bigr]\biggr)^{\tfrac{1}{2}}\\
&= \biggl(\mathbb{E}\Bigl[\exp\Bigl(\int_{0}^{T}2Z_{r}dB_{r} - \tfrac{1}{2}\int_{0}^{T}(2Z_{r})^{2}dr\Bigr)\Bigr]\biggr)^{\tfrac{1}{2}} \cdot \biggl(\mathbb{E}\Bigl[\exp\Bigl(\int_{0}^{T}(\sqrt{2}Z_{r})^{2}dr\Bigr)\Bigr]\biggr)^{\tfrac{1}{2}},
\end{align*}
and for the last term we get that the left expectation is equal to 1 because it is the expectation of the stochastic exponential of $\int_{0}^{T}2Z_{r}dB_{r}$, which is a martingale (this follows from the fact that Novikov's condition holds by Lemma \ref{Lemma2.3}), and we can estimate the right expectation with Lemma \ref{Lemma2.3} to get
\[
\mathbb{E}\Bigl[\exp\Bigl(\int_{0}^{T}Z_{r}dB_{r}\Bigr)\Bigr] \leq \big(1-||\sqrt{2}Z||_{\mathcal{Z}_{BMO}^{2}}^{2}\big)^{-\tfrac{1}{2}} = \big(1-2||Z||_{\mathcal{Z}_{BMO}^{2}}^{2}\big)^{-\tfrac{1}{2}}.
\] 
\end{proof}



\section{Utility maximization with price impact and quadratic BSDEs}
\label{sec2}

In this section, we focus on the price impact model and derive the multidimensional qBSDEs that form the main object of study in the next section. The second benchmark model, namely the equilibrium pricing model, is deferred to Section \ref{sec:4.2}.

\subsection{Multiplayer model with price impact of averaged actions}
\label{sec2.1}

Suppose the financial market has $N$ companies (players) with their stock prices $S^i$ depending on all the other companies' actions (controls) $\Balpha^{-i}:=(\alpha^j)_{j=1,j\neq i}^N \in \MA^{N-1}$, with $\MA$ as the set of admissible controls given by
\begin{equation}
\label{equ:2.12}
\mathcal{A} := \bigl\{\alpha \ \vert \ \mathbb{R}\text{-valued and progressively measurable, such that: } ||\alpha||_{\mathcal{Z}_{BMO}^{2}} \leq M \bigr\},
\end{equation}
for some constant $M>0.$ Let the stock price $S^i$ be defined as
\be
\label{equ:2.4}
dS_{t}^{i,\Balpha} = S_{t}^{i,\Balpha}((\mu_{t}^{N,\Balpha}+\tilde{\mu}_{t}^{i})dt+\sigma_{t}^{i}dW_{t}^{i}+\sigma_{t}^{0}dW_{t}^{0}),
\ee
where $\tilde{\mu}^{i}$, $\sigma^{i}$, and $\sigma^{0}$ are progressively measurable processes with respect to $\mathbb{F}$, and $\mu_{t}^{N,\Balpha}=\mu^N_t(\Balpha^{-i})$ is the price impact, a typical example of which is $\mu_{t}^{N,\Balpha}=\mu^{N,\Balpha^{-i}}_t := \frac{1}{N-1} \sum_{j \neq i} \alpha_{t}^{j} $.
For any initial wealth $\Bx=(x^1, ..., x^N) \in \R^N,$ the wealth process of player $i$ is given by
\begin{equation}
\label{equ:2.5}
X_{t}^{x, \Balpha, i} = x^i + \int_{0}^{t}\alpha_{r}^{i}\tfrac{dS_{r}^{i,\Balpha}}{S_{r}^{i,\Balpha}} = x^i + \int_{0}^{t}{\alpha_{r}^{i}(\mu_{r}^{N,\Balpha}+\tilde{\mu}_{r}^{i})dr} + \int_{0}^{t}{\alpha_{r}^{i}\sigma_{r}^{i}dW_{r}^{i}} + \int_{0}^{t}\alpha_{r}^{i}\sigma_{r}^{0}dW_{r}^{0} \ , \ t \in [0,T].
\end{equation}
Suppose each player has an exponential utility, yielding that every player is expected to maximize
\begin{equation}
\label{equ:2.6}
I_{i}^{N}(\Balpha) := \mathbb{E}[U(X_{T}^{x,\Balpha,i}-\xi_{i})] \text{   with   } U(x) := -\exp(-\eta x) \ , \ \eta \in \mathbb{R}_{>0},
\end{equation}
where the $\xi_{i}$ are $\mathcal{F}_{T}$-measurable and represent the payoffs of the respective players at time $T$. 
The goal is to find a Nash equilibrium for the above game. 

\begin{Definition}[Nash Equilibrium]
\label{DefEpsNE}
A control $\hat{\Balpha} \in \mathcal{A}^{N}$ is called a Nash equilibrium for the above $N$-player game if for any $i \in \{1,...,N\}$ and $\alpha \in \mathcal{A}$ it holds that
$
I_{i}^{N}(\hat{\Balpha})  \geq I_{i}^{N}(\hat{\Balpha}^{-i},\alpha).
$ If for an $\varepsilon>0$ it holds that 
\begin{equation}
\label{equ:2.7}
I_{i}^{N}(\hat{\Balpha}) + \varepsilon \geq \sup_{\alpha \in \MA} I_{i}^{N}(\hat{\Balpha}^{-i},\alpha),
\end{equation}
we call $\hat{\Balpha}$ above an $\varepsilon$-Nash Equilibrium.
\end{Definition}


\subsection{The martingale optimization approach and related quadratic BSDEs}
\label{sec2.2}
To investigate equilibria of the $N$-player game, we would like to construct, according to the optimal martingale approach (see, for example, \cite[Section 6.6.1]{Ref4}), a family of processes $\{ (J_{t}^{\Balpha,i})_{0 \leq t \leq T} \}_{\Balpha \in \mathcal{A}^{N}}$, and some $\hBalpha \in \MA^N,$ that satisfy the following conditions for every $i \in \{1,...,N\}$ and $\alpha \in \MA$: 
\begin{align}
&J_{T}^{(\hBalpha^{-i},\alpha),i} = U(X_{T}^{x,(\hBalpha^{-i},\alpha),i}-\xi_{i}), \label{(i)} \tag{\textbf{i}} \\
&J_{0}^{(\hBalpha^{-i}, \alpha),i} \ \text{is a constant independent of} \ \alpha  \in \mathcal{A}, \label{(ii)} \tag{\textbf{ii}} \\
&J^{(\hBalpha^{-i}, \alpha),i} \ \text{is a supermartingale and }  J^{\hBalpha,i} \ \text{is a martingale.} \label{(iii)} \tag{\textbf{iii}}
\end{align}
If such a family exists, we get for any $\alpha \in \mathcal{A} $ and $\hBalpha= (\hat{\alpha}^{i})_{i=1}^N$ from (\ref{(iii)}),
\begin{equation}
\label{equ:martopt}
\begin{split}
\mathbb{E}[ U(X_{T}^{x,(\hat{\Balpha}^{-i} , \alpha ),i}-\xi_{i})] & = \mathbb{E}[J_{T}^{(\hat{\Balpha}^{-i} , \alpha ),i}]\\
& \leq J_{0}^{(\hat{\Balpha}^{-i} , \alpha ),i} = J_{0}^{(\hBalpha^{-i},\hat{\alpha}^{i}),i} = \mathbb{E}[J_{T}^{(\hBalpha^{-i},\hat{\alpha}^{i}),i}] = \mathbb{E}[U(X_{T}^{x, \hBalpha,i}-\xi_{i})],
\end{split}
\end{equation}
where for simplicity, we take the initial wealth $x^i =x$ for any $i=1,...,N.$ Hence, $\hBalpha$ is a Nash equilibrium.
To achieve this, we want $J^{\Balpha,i}$ to be of the form
\[
J_{t}^{\Balpha,i} = U(X_{t}^{x,\Balpha,i}-Y_{t}^{i}) \ , \ 0 \leq t \leq T \ , \ \Balpha \in \mathcal{A}^{N} \ , \ i \in \{1,...,N\},
\]
where $Y_{t}^{i}$ is defined to be a solution of the BSDE system
\begin{equation}
\label{equ:qBSDE}
Y_{t}^{i} = \xi_{i} + \int_{t}^{T}f_{r}^{i}(Z_{r})dr - \int_{t}^{T}\sum_{j=0}^{N}Z_{r}^{ij}dW_{r}^{j} \ , \ i \in \{1,...,N\},
\end{equation}
where $Z = (Z^{ij})_{ij}$ is $\mathbb{R}^{N \times (N+1)}$-valued and $f^{i}$ is to be determined. Our goal is to represent $J^{\Balpha,i}$ as a product of the form
\be\label{eq:martop}
J_{t}^{\Balpha,i} = M_{t}^{\Balpha,i}C_{t}^{\Balpha,i},
\ee
where, for every $\Balpha \in \mathcal{A}^{N}$, $M^{\Balpha,i}$ is a martingale and $C^{\Balpha,i}$ a non-increasing process such that $C^{\hBalpha,i}$ is constant for some $\hat{\Balpha}  \in \mathcal{A}^N$. We can compute
\begin{align*}
X_{t}^{x,\Balpha,i} - Y_{t}^{i}
&= x - Y_{0}^{i} + \int_{0}^{t}(\alpha_{r}^{i}(\mu_{r}^{N,\Balpha}+\tilde{\mu}_{r}^{i})+f_{r}^{i}(Z_{r}))dr + \int_{0}^{t}(\alpha_{r}^{i}\sigma_{r}^{i}-Z_{r}^{ii})dW_{r}^{i}  - \int_{0}^{t}\sum_{j \neq i,0}Z_{r}^{ij}dW_{r}^{j}\\
&\ \ \ + \int_{0}^{t}(\alpha_{r}^{i}\sigma_{r}^{0}-Z_{r}^{i0})dW_{r}^{0}.
\end{align*}
Taking the above identity to \eqref{eq:martop}, we have 
\begin{align}
M_{t}^{\Balpha,i} := & e^{-\eta(x-Y_{0}^{i})}\exp\biggl(-\eta\Bigl(\int_{0}^{t}(\alpha_{r}^{i}\sigma_{r}^{i}-Z_{r}^{ii})dW_{r}^{i} + \int_{0}^{t}(\alpha_{r}^{i}\sigma_{r}^{0}-Z_{r}^{i0})dW_{r}^{0} - \int_{0}^{t}\sum_{j \neq i,0}Z_{r}^{ij}dW_{r}^{j}\Bigr) \nonumber \\
&- \tfrac{1}{2}\eta^{2}\Bigl(\int_{0}^{t}(\alpha_{r}^{i}\sigma_{r}^{i}-Z_{r}^{ii})^{2}dr + \int_{0}^{t}(\alpha_{r}^{i}\sigma_{r}^{0}-Z_{r}^{i0})^{2}dr + \int_{0}^{t}\sum_{j \neq i,0}(Z_{r}^{ij})^{2}dr\Bigr)\biggr) \label{equ:DefM} \\
C_{t}^{\Balpha,i} :=& -\exp\biggl(-\eta\int_{0}^{t}(\alpha_{r}^{i}(\mu_{r}^{N,\Balpha}+\tilde{\mu}_{r}^{i})+f_{r}^{i}(Z_{r}))dr \nonumber \\
&+ \tfrac{1}{2}\eta^{2}\Bigl(\int_{0}^{t}(\alpha_{r}^{i}\sigma_{r}^{i}-Z_{r}^{ii})^{2}dr + \int_{0}^{t}(\alpha_{r}^{i}\sigma_{r}^{0}-Z_{r}^{i0})^{2}dr + \int_{0}^{t}\sum_{j \neq i,0}(Z_{r}^{ij})^{2}dr\Bigr)\biggr) \nonumber \\
 = & -\exp\Bigl(\int_{0}^{t}\rho_{r}^{\Balpha, i}(\alpha_{r}^{i},Z_{r})dr\Bigr), \label{equ:DefC}
\end{align}
where we define
$
\rho_{t}^{\Balpha, i}(a,z) := \tfrac{1}{2}\eta^{2}(a\sigma_{t}^{i}-z^{ii})^{2} + \tfrac{1}{2}\eta^{2}(a\sigma_{t}^{0}-z^{i0})^{2} + \tfrac{1}{2}\eta^{2}\sum_{j \neq i,0}(z^{ij})^{2} - \eta(a(\mu_{t}^{N,\Balpha}+\tilde{\mu}_{t}^{i})+f_{t}^{i}(z)).
$
Now, to achieve (\ref{(iii)}), we only need
\begin{equation}
\label{equ:new2}
\rho_{t}^{\Balpha, i}(\alpha^{i}_{t},Z_{t}) \geq 0  \ \ \forall \Balpha  \in \mathcal{A}^N \ , \ 0 \leq t \leq T \ \ \ \text{and} \ \ \ \rho_{t}^{\hBalpha, i}(\hat{\alpha}^{i}_{t},Z_{t}) = 0 \ , \ 0 \leq t \leq T.
\end{equation}
Set $(\bar{\sigma}_{t}^{i})^{2} := (\sigma_{t}^{i})^{2} + (\sigma_{t}^{0})^{2}$. We can compute
\begin{align*}
\tfrac{1}{\eta}\rho_{t}^{\Balpha, i}(a,z) = & \tfrac{\eta (\bar{\sigma}_{t}^{i})^{2}}{2}\Bigl(a - (\bar{\sigma}_{t}^{i})^{-2}(\sigma_{t}^{i}z^{ii} + \sigma_{t}^{0}z^{i0} + \tfrac{\mu_{t}^{N,\Balpha}+\tilde{\mu}_{t}^{i}}{\eta})\Bigr)^{2} - \tfrac{\eta}{2(\bar{\sigma}_{t}^{i})^{2}}\bigl(\sigma_{t}^{i}z^{ii} + \sigma_{t}^{0}z^{i0} + \tfrac{\mu_{t}^{N,\Balpha}+\tilde{\mu}_{t}^{i}}{\eta}\bigr)^{2}\\
&+ \tfrac{\eta}{2}\sum_{j=0}^{N}(z^{ij})^{2} - f_{t}^{i}(z),    
\end{align*}
which then implies that (\ref{equ:new2}) is satisfied if we choose
\be\label{eq:fop}
f_{t}^{i}(z) = \tfrac{\eta}{2}\sum_{j=0}^{N}(z^{ij})^{2} - \tfrac{\eta}{2(\bar{\sigma}_{t}^{i})^{2}}\bigl(\sigma_{t}^{i}z^{ii} + \sigma_{t}^{0}z^{i0} + \tfrac{\mu_{t}^{N,\hBalpha}+\tilde{\mu}_{t}^{i}}{\eta}\bigr)^{2}, 
 \  \text{ and } \ \ 
 \hat{\alpha}_{t}^{i} = \tfrac{\sigma_{t}^{i}}{(\bar{\sigma}_{t}^{i})^{2}}Z_{t}^{ii} + \tfrac{\sigma_{t}^{0}}{(\bar{\sigma}_{t}^{i})^{2}}Z_{t}^{i0} + \tfrac{\mu_{t}^{N,\hBalpha}+\tilde{\mu}_{t}^{i}}{\eta(\bar{\sigma}_{t}^{i})^{2}}.
\ee
Note that there is $\mu^{N,\hBalpha}=\mu^N(\hBalpha^{-i})$ in the expression of $f^i$. It remains to represent $\mu^{N,\hBalpha}$ by $Z$ of BSDE \eqref{equ:qBSDE} by the second identity above. For simplicity of calculation and expression, now we take (otherwise, see the following Remark \ref{rem:NvsN-1})
\begin{equation}
\label{equ:DefMu}
\mu_{t}^{N,\Balpha}:= \tfrac{1}{N}\sum_{j=1}^{N}\alpha_{t}^{j}.
\end{equation}
Then, by the expression of $\hBalpha$ in \eqref{eq:fop}, 
\[
\mu_{t}^{N,\hat{\Balpha}}  = \Bigl(\tfrac{1}{N}\sum_{j=1}^{N}\tfrac{\sigma_{t}^{j}}{(\bar{\sigma}_{t}^{j})^{2}}Z_{t}^{jj}\Bigr) + \Bigl(\tfrac{1}{N}\sum_{j=1}^{N}\tfrac{\sigma_{t}^{0}}{(\bar{\sigma}_{t}^{j})^{2}}Z_{t}^{j0}\Bigr) + \tfrac{1}{\eta}\Bigl(\tfrac{1}{N}\sum_{j=1}^{N}(\bar{\sigma}_{t}^{j})^{-2}\Bigr)\mu_{t}^{N,\hat{\Balpha}} + \tfrac{1}{\eta}\Bigl(\tfrac{1}{N}\sum_{j=1}^{N}\tfrac{\tilde{\mu}_{t}^{j}}{(\bar{\sigma}_{t}^{j})^{2}}\Bigr),
\]
which is equivalent to
\begin{equation}
\label{equ:RepmuN}
\mu_{t}^{N,\hat{\Balpha}} = c_{t}\Bigl(\tfrac{1}{N}\sum_{j=1}^{N}\tfrac{\sigma_{t}^{j}}{(\bar{\sigma}_{t}^{j})^{2}}Z_{t}^{jj}\Bigr) + c_{t}\Bigl(\tfrac{1}{N}\sum_{j=1}^{N}\tfrac{\sigma_{t}^{0}}{(\bar{\sigma}_{t}^{j})^{2}}Z_{t}^{j0}\Bigr) + \tfrac{c_{t}}{\eta}\Bigl(\tfrac{1}{N}\sum_{j=1}^{N}\tfrac{\tilde{\mu}_{t}^{j}}{(\bar{\sigma}_{t}^{j})^{2}}\Bigr),
\end{equation}
where  
$
c_{t}:=\Bigl(1-\tfrac{1}{\eta}\Bigl(\tfrac{1}{N}\sum_{j=1}^{N}(\bar{\sigma}_{t}^{j})^{-2}\Bigr)\Bigr)^{-1}.
$
Taking the above $\mu_{t}^{N,\hat{\Balpha}}$ back to \eqref{eq:fop} yields that the driver of BSDE \eqref{equ:qBSDE} is given by
\begin{equation}
\label{equ:Def_f}
f_{t}^{i}(Z_{t}) := \phi_{t}^{i}(Z_{t}) + \psi_{t}^{i}(Z_{t}) + l_{t}^{i}(Z_{t}) \ , \ i \in \{1,...,N\} \ , \ t \in [0,T],
\end{equation}
with
\begin{align}
\phi_{t}^{i}(Z_{t}) &= (\tfrac{\eta}{2}+\nu_{t,1}^{i})(Z_{t}^{ii})^{2} + (\tfrac{\eta}{2}+\nu_{t,2}^{i})(Z_{t}^{i0})^{2} + \nu_{t,3}^{i}Z_{t}^{ii}Z_{t}^{i0} + \nu_{t,4}^{i}Z_{t}^{ii}(\tilde{Z}_{t} + \tilde{Z}_{t}^{0}) \nonumber \\
&\ \ \ + \nu_{t,5}^{i}Z_{t}^{i0}(\tilde{Z}_{t} + \tilde{Z}_{t}^{0}) + \nu_{t,6}^{i}(\tilde{Z}_{t} + \tilde{Z}_{t}^{0})^{2}, \label{equ:Def_g} \\
\psi_{t}^{i}(Z_{t}) &= \tfrac{\eta}{2}\sum_{j \neq i,0}(Z_{t}^{ij})^{2}, \label{equ:Def_h} \\
l_{t}^{i}(Z_{t}) &= \nu_{t,7}^{i}Z_{t}^{ii} + \nu_{t,8}^{i}Z_{t}^{i0} + \nu_{t,9}^{i}(\tilde{Z}_{t} + \tilde{Z}_{t}^{0}) + \nu_{t,10}^{i}, \label{equ:Def_l}
\end{align}
where 
\begin{equation}
\label{equ:DefZtilde}
\tilde{Z}_{t} := \tfrac{1}{N}\sum_{j=1}^{N}\tfrac{\sigma_{t}^{j}}{(\bar{\sigma}_{t}^{j})^{2}}Z_{t}^{jj} \ \ \text{and} \ \ \tilde{Z}_{t}^{0} := \tfrac{1}{N}\sum_{j=1}^{N}\tfrac{\sigma_{t}^{0}}{(\bar{\sigma}_{t}^{j})^{2}}Z_{t}^{j0},
\end{equation}
and $(\nu^i_{t,k})_{k=1}^{10}$ are parameters given by
\begin{align}
& \nu_{t,1}^{i} := -\tfrac{\eta (\sigma_{t}^{i})^{2}}{2(\bar{\sigma}_{t}^{i})^{2}}, \
\nu_{t,2}^{i} := -\tfrac{\eta (\sigma_{t}^{0})^{2}}{2(\bar{\sigma}_{t}^{i})^{2}},   \
\nu_{t,3}^{i} := -\tfrac{\eta\sigma_{t}^{i}\sigma_{t}^{0}}{(\bar{\sigma}_{t}^{i})^{2}},   \
\nu_{t,4}^{i} := -\tfrac{c_{t}\sigma_{t}^{i}}{(\bar{\sigma}_{t}^{i})^{2}}  ,  \
\nu_{t,5}^{i} := -\tfrac{c_{t}\sigma_{t}^{0}}{(\bar{\sigma}_{t}^{i})^{2}}  ,  
\label{equ:nu1-5} \\
&\nu_{t,6}^{i} := -\tfrac{c_{t}^{2}}{2\eta(\bar{\sigma}_{t}^{i})^{2}},  \ \nu_{t,7}^{i} := -\tfrac{c_{t}\sigma_{t}^{i}}{\eta(\bar{\sigma}_{t}^{i})^{2}}\Bigl(\tfrac{1}{N}\sum_{j=1}^{N}\tfrac{\tilde{\mu}_{t}^{j}}{(\bar{\sigma}_{t}^{j})^{2}}\Bigr) - \tfrac{\sigma_{t}^{i}\tilde{\mu}_{t}^{i}}{(\bar{\sigma}_{t}^{i})^{2}}, \nu_{t,8}^{i} := -\tfrac{c_{t}\sigma_{t}^{0}}{\eta(\bar{\sigma}_{t}^{i})^{2}}\Bigl(\tfrac{1}{N}\sum_{j=1}^{N}\tfrac{\tilde{\mu}_{t}^{j}}{(\bar{\sigma}_{t}^{j})^{2}}\Bigr) - \tfrac{\sigma_{t}^{0}\tilde{\mu}_{t}^{i}}{(\bar{\sigma}_{t}^{i})^{2}} \label{equ:nu678}, \\
&\nu_{t,9}^{i} := -\tfrac{c_{t}^{2}}{\eta^{2}(\bar{\sigma}_{t}^{i})^{2}}\Bigl(\tfrac{1}{N}\sum_{j=1}^{N}\tfrac{\tilde{\mu}_{t}^{j}}{(\bar{\sigma}_{t}^{j})^{2}}\Bigr) - \tfrac{c_{t}\tilde{\mu}_{t}^{i}}{\eta(\bar{\sigma}_{t}^{i})^{2}} \label{equ:nu9}, \\
&\nu_{t,10}^{i} := -\tfrac{c_{t}^{2}}{2\eta^{3}(\bar{\sigma}_{t}^{i})^{2}}\Bigl(\tfrac{1}{N}\sum_{j=1}^{N}\tfrac{\tilde{\mu}_{t}^{j}}{(\bar{\sigma}_{t}^{j})^{2}}\Bigr)^{2} - \tfrac{c_{t}\tilde{\mu}_{t}^{i}}{\eta^{2}(\bar{\sigma}_{t}^{i})^{2}}\Bigl(\tfrac{1}{N}\sum_{j=1}^{N}\tfrac{\tilde{\mu}_{t}^{j}}{(\bar{\sigma}_{t}^{j})^{2}}\Bigr) - \tfrac{(\tilde{\mu}_{t}^{i})^{2}}{2\eta(\bar{\sigma}_{t}^{i})^{2}}\label{equ:nu10}.
\end{align}
Suppose that the quadratic BSDE system \eqref{equ:qBSDE} driven by $f$ above can be solved by a solution $(Y,Z)$. Then the equilibrium is given by \eqref{eq:fop} and \eqref{equ:RepmuN}.

\begin{Remark}\label{rem:NvsN-1}

We can work with $\mu_{t}^{N,\Balpha}:=\mu^{N,\Balpha^{-i}}_t:= \tfrac{1}{N-1} \sum_{j\neq i} \alpha_{t}^{j}$ and then identity \eqref{equ:RepmuN} becomes
\[
\begin{split}
\mu^{N,\Balpha^{-i}}_t & =   \Bigl(\tfrac{1}{N-1}\sum_{j\neq i} \tfrac{\sigma_{t}^{j}}{(\bar{\sigma}_{t}^{j})^{2}}Z_{t}^{jj}\Bigr) + \Bigl(\tfrac{1}{N-1}\sum_{j \neq i} \tfrac{\sigma_{t}^{0}}{(\bar{\sigma}_{t}^{j})^{2}}Z_{t}^{j0}\Bigr) + \tfrac{1}{\eta}\Bigl(\tfrac{1}{N-1}\sum_{j \neq i }(\bar{\sigma}_{t}^{j})^{-2}\mu_{t}^{N, {\hat{\Balpha}}^{-j} } \Bigr)+ \tfrac{1}{\eta}\Bigl(\tfrac{1}{N-1}\sum_{j \neq i} \tfrac{\tilde{\mu}_{t}^{j}}{(\bar{\sigma}_{t}^{j})^{2}}\Bigr)\\
& = :  a^i_t (Z) + \frac{1}{N-1} \sum_{j \neq i} b^j_t \mu^{N,\Balpha^{-j}}_t,
\end{split}
\]
where $a^i(Z)$ is linear in $Z$ and $b^i$ is independent of $Z.$
It is not difficult to solve the above equation and obtain that 
\begin{equation} \label{equ:defrealmu}
\mu^{N,\Balpha^{-i}}_t
=
\frac{(N-1)a_t^i}{N-1+b_t^i}
+
\frac{N-1}{N-1+b_t^i}
\cdot
\frac{\displaystyle \sum_{k=1}^N \frac{b^k_t a^k_t}{N-1+b^k_t}}
{\displaystyle 1-\sum_{k=1}^N \frac{b^k_t}{N-1+b^k_t}}=:\mu^i_t(Z).
\end{equation}
Plugging this into \eqref{eq:fop} would yield a driver that is of the same form as \eqref{equ:Def_f}, and we obtain the candidate for the true equilibrium $\lps{*}\Balpha=(\lps{*}{\alpha}^{i})_{i=1}^N$,
\be\label{eq:trueeq}
\lps{*}{\alpha}_{t}^{i} = \tfrac{\sigma_{t}^{i}}{(\bar{\sigma}_{t}^{i})^{2}}Z_{t}^{ii} + \tfrac{\sigma_{t}^{0}}{(\bar{\sigma}_{t}^{i})^{2}}Z_{t}^{i0} + \tfrac{\mu^i_t(Z) +\tilde{\mu}_{t}^{i}}{\eta(\bar{\sigma}_{t}^{i})^{2}}.
\ee
Indeed, by replacing $\mu^{N,\Balpha^{-i}}$ with \eqref{equ:DefMu} and solving the BSDE system driven by \eqref{equ:Def_f}, we obtain an $\varepsilon_N$-Nash equilibrium for the original game. Moreover, in view of Section \ref{sec4.1}, both the true equilibrium 
and this $\varepsilon_N$-equilibrium converge to the same mean-field limit as $N$ goes to infinity.
    
\end{Remark}

\section{Multidimensional quadratic BSDEs with weak interactions}
\label{sec3}

In this section, we show well-posedness and stability results for general multidimensional quadratic BSDEs, a typical example of which is our equation \eqref{equ:qBSDE} with driver \eqref{equ:Def_f}. Let $(\Omega,\mathcal{F}, \mathbb{P})$ be a complete probability space supporting a $d$-dimensional Brownian motion $(W^{1},...,W^{d})$ and $(\MF^W_t)_t$ the augmented filtration generated by $(W^i)_{i=1}^d$.

\subsection{Well-posedness of quadratic BSDE systems with weak interaction}
\label{sec3.1}

To apply our quadratic BSDE result on mean-field games, we need to build the theory and estimates independent of the dimension of the solution $(Y,Z).$ More precisely, we keep track of the dimension of $(Y,Z)$ and denote it by $M:=(N,N\times d)$. 

Consider the following BSDE driven by $f^i_M$ and $d$-dimensional Brownian motion $(W^{1},...,W^{d}):$
\begin{equation}
\label{equ:3.genBSDE}
Y_{t}^{M,i} = \xi_{i}^{M} + \int_{t}^{T}f_{M}^{i}(r,Y_{r}^{M},Z_{r}^{M})dr - \int_{t}^{T}\sum_{j=1}^{d}Z_{r}^{M,ij}dW_{r}^{j}, \ \ i \in \{1,...,N\}.
\end{equation}
To build the well-posedness of the above equation, we assume the following.

\begin{Assumption}\label{assu:well-posedBSDE}

\begin{itemize}

\item[\textnormal{(i)}] $\xi_{i}^{M} \in L^{\infty}(\MF^W_T)$ for every $i \in \{1,...,N\}$.

\item[\textnormal{(ii)}] $f_{M},g_{M},h_{M},l_{M}: \Omega \times [0,T] \times \mathbb{R}^{N} \times \mathbb{R}^{N \times d} \rightarrow \mathbb{R}^{N}$ such that  $f_{M} = g_{M} + h_{M} + l_{M}$, and for any $(y,z) \in \R^N \times \R^{N \times d},$ $\phi \in \{ f_{M},g_{M},h_{M},l_{M}\},$ $\phi(\cdot, y,z)$ is predictable. Moreover, for any $i \in \{1,...,N\}$, $t \in [0,T]$ and $(y,z),(\bar{y},\bar{z}) \in \mathbb{R}^{N} \times \mathbb{R}^{N \times d}$, 
\begin{align}\label{equ:3.estcomponent}
&|g_{M}^{i}(t,y,z) - g_{M}^{i}(t,\bar{y},\bar{z})|  \\
&\leq \sum_{j=1}^{N}C_{M,i}^{Y,j}|\Delta y^{j}|(|y^{j}|+|\bar{y}^{j}|) + \sum_{j=1}^{N}\sum_{k=1}^{d}C_{M,i}^{Z,jk}|\Delta z^{jk}|(|z^{jk}|+|\bar{z}^{jk}|), \nonumber \\ \label{equ:3.estcross} 
&|h_{M}^{i}(t,y,z) - h_{M}^{i}(t,\bar{y},\bar{z})|  \\ 
& \leq \Bigl(\sum_{j=1}^{N}\tilde{C}_{M,i}^{Y,j}|\Delta y^{j}| + \sum_{j=1}^{N}\sum_{k=1}^{d}\tilde{C}_{M,i}^{Z,jk}|\Delta z^{jk}|\Bigr)\Bigl(\sum_{j=1}^{N}\tilde{C}_{M,i}^{Y,j}(|y^{j}|+|\bar{y}^{j}|) + \sum_{j=1}^{N}\sum_{k=1}^{d}\tilde{C}_{M,i}^{Z,jk}(|z^{jk}|+|\bar{z}^{jk}|)\Bigr), \nonumber \\ \label{equ:3.estlinear}
&|l_{M}^{i}(t,y,z)-l_{M}^{i}(t,\bar{y},\bar{z})| \leq \sum_{j=1}^{N}L_{M,i}^{Y,j}|\Delta y^{j}| + \sum_{j=1}^{N}\sum_{k=1}^{d}L_{M,i}^{Z,jk}|\Delta z^{jk}|, 
\end{align}
where we write $\Delta y^{j} := y^{j}-\bar{y}^{j}$, $\Delta z^{jk} := z^{jk}-\bar{z}^{jk}$, and $C_{M,i}^{Y,j}$, $\tilde{C}_{M,i}^{Y,j}$, $C_{M,i}^{Z,jk}$, $\tilde{C}_{M,i}^{Z,jk}$, $L_{M,i}^{Y,j}$, $L_{M,i}^{Z,jk}$ are non-negative constants.


\item[\textnormal{(iii)}] There exist constants $C,L>0$ independent of $M$ such that 
\begin{align*}
&\max \Biggl\{T\sum_{j=1}^{N}C_{M,i}^{Y,j} \ , \ T\sum_{j=1}^{N}\tilde{C}_{M,i}^{Y,j} \ , \ \sum_{j=1}^{N}\max_{k \in \{1,...,d\}}C_{M,i}^{Z,jk} \ , \ \sum_{j=1}^{N}\sum_{k=1}^{d}\tilde{C}_{M,i}^{Z,jk}\Biggr\} \leq C\\
& \max\biggl\{\sum_{j=1}^{N}L_{M,i}^{Y,j},\sum_{j=1}^{N}\sum_{k=1}^{d}L_{M,i}^{Z,jk}\biggr\} \leq L.
\end{align*}

\item[\textnormal{(iv)}] With $\beta := 32L^{2}(1+T)$, $R^{2}:= 4\max_{i \in \{1,...,N\}}||\xi_{i}^{M}||_{L^{\infty}}^{2} + \tfrac{8}{\beta}\max_{i \in \{1,...,N\}}\Bigl|\Bigl|\int_{0}^{T}(f_{M}^{i}(r,0,0))^{2}dr\Bigr|\Bigr|_{L^{\infty}}$ it holds
\[
256C^{2}(1+2C^{2}+\tfrac{2C^{2}}{T^{2}})e^{\beta T}R^{2} < 1.
\]

\end{itemize}

\end{Assumption}

\begin{Remark}
According to \textnormal{(ii)}, the driver $f_{M}$ is composed of a ``pure quadratic part'' $g_{M}$, a ``cross quadratic part'' $h_{M}$ and a ``linear part'' $l_{M}$, a typical example of which is our model \eqref{equ:Def_f}.

Condition \textnormal{(iii)} highlights that the pure quadratic terms and the cross quadratic terms in the $z$-component must be controlled at different scales in order for the above bounds to hold. In particular, the coefficients associated with the cross terms are required to satisfy a stronger (summability-type) condition than those corresponding to the purely quadratic terms.

Moreover, condition \textnormal{(iv)} imposes a smallness requirement ensuring well-posedness. More precisely, it requires that either the terminal condition in each component or the quadratic growth (as captured by the constants in \textnormal{(iii)}) is sufficiently small. Importantly, this requirement is \emph{dimension-free}: it is imposed componentwise and does not deteriorate as the dimension $N$ increases, even though the norm of the solution, viewed as a vector in $\mathbb{R}^N$, may become large. This smallness condition is crucial for the well-posedness of multidimensional quadratic BSDEs; see \cite[Theorem 2.1]{FR11} for a counterexample demonstrating that a two-dimensional quadratic BSDE may fail to admit a solution.

\end{Remark}

\begin{Remark}[Necessity of Assumption \ref{assu:well-posedBSDE} (ii)--(iii)]
The detailed structure imposed in Assumption \ref{assu:well-posedBSDE} \textnormal{(ii)} and \textnormal{(iii)} is crucial for studying systems whose dimension \(N\) tends to infinity. To illustrate this point, let us compare our framework with a standard quadratic growth assumption of the form
\[
|f(z)-f(\bar{z})|
\leq
\hat{C} \bigl(|z|+|\bar{ z}|\bigr)|z-\bar{ z}|,
\]
where, for simplicity, we assume that the driver depends only on \(Z\) and that \(d=N\). Here \( |\cdot| \) denotes the Frobenius norm
$
|z|_F
:=
\Bigl(
\sum_{i,j=1}^N |z^{ij}|^2
\Bigr)^{1/2}.
$
Under the usual uniform boundedness assumptions for qBSDEs, one naturally expects the radius of the fixed-point ball to satisfy
$
\hat R := |\xi|_F \sim \sqrt{N}
\text{ as } N\to\infty.
$
Therefore, in order to obtain estimates independent of the dimension, condition \textnormal{(iv)} above (or equivalently the intuition from the one-dimensional case) suggests that one would need
$
\hat C \sim N^{-1/2}.
$
However, such a scaling excludes natural terms arising in mean-field models, such as
$
f^i(z)=\sum_{j \neq i} (z^{ij})^2,
$
whose quadratic coefficient does not decay with \(N\). This illustrates why a more refined structural assumption, such as Assumption \ref{assu:well-posedBSDE} \textnormal{(ii)} and \textnormal{(iii)}, is necessary. 
\end{Remark}

\begin{example}
Consider the dimension $M := (N,N\times d)$ and define a driver by $f := g + h + l$ such that for $i \in \{1,...,N\}$ and $t \in [0,T]$
\begin{align*}
g_{t}^{i}(y,z) &:= \sum_{j=1}^{N}\gamma_{i,t}^{Y,j}(y^{j})^{2} + \sum_{j=1}^{N}\sum_{k=1}^{d}\gamma_{i,t}^{Z,jk}(z^{jk})^{2}, \\
h_{t}^{i}(y,z) &:= \Bigl(\sum_{j=1}^{N}\alpha_{i,t}^{Y,j}y^{j} + \sum_{j=1}^{N}\sum_{k=1}^{d}\alpha_{i,t}^{Z,jk}z^{jk}\Bigr)\Bigl(\sum_{j=1}^{N}\beta_{i,t}^{Y,j}y^{j} + \sum_{j=1}^{N}\sum_{k=1}^{d}\beta_{i,t}^{Z,jk}z^{jk}\Bigr), \\
l_{t}^{i}(y,z) &:= \sum_{j=1}^{N}\delta_{i,t}^{Y,j}y^{j} + \sum_{j=1}^{N}\sum_{k=1}^{d}\delta_{i,t}^{Z,jk}z^{jk},
\end{align*}
where, for every $(i,j,k) \in \{1,...,N\}\times \{1,...,N\} \times \{1,...,d\}$ we have that $\alpha_{i,t}^{Y,j}$, $\alpha_{i,t}^{Z,jk}$, $\beta_{i,t}^{Y,j}$, $\beta_{i,t}^{Z,jk}$, $\gamma_{i,t}^{Y,j}$, $\gamma_{i,t}^{Z,jk}$, $\delta_{i,t}^{Y,j}$, $\delta_{i,t}^{Z,jk}$ are bounded predictable processes. Then, we can see that $f$ is of the form from Assumption \ref{assu:well-posedBSDE} \textnormal{(ii)} if there exist non-negative constants $C_{M,i}^{Y,j}$, $\tilde{C}_{M,i}^{Y,j}$, $C_{M,i}^{Z,jk}$, $\tilde{C}_{M,i}^{Z,jk}$, $L_{M,i}^{Y,j}$, $L_{M,i}^{Z,jk}$ satisfying \textnormal{(iii)}, such that
\begin{align*}
&|\gamma_{i,t}^{Y,j}| \leq C_{i}^{Y,j} \ , \ |\gamma_{i,t}^{Z,jk}| \leq C_{i}^{Z,jk} \ , \ \max\bigl\{|\alpha_{i,t}^{Y,j}|,|\beta_{i,t}^{Y,j}|\bigr\} \leq \tilde{C}_{i}^{Y,j}, \\
&|\delta_{i,t}^{Y,j}| \leq L_{i}^{Y,j} \ \ , \ |\delta_{i,t}^{Z,jk}| \leq L_{i}^{Z,jk} \ , \ \max\bigl\{|\alpha_{i,t}^{Z,jk}|,|\beta_{i,t}^{Z,jk}|\bigr\} \leq \tilde{C}_{i}^{Z,jk}.
\end{align*}
\end{example}


To build the solution in dependence of the dimension $N \times d$, we define the norm $\trinm{ \cdot }_\alpha$ with constant $\alpha>0$ on $\mathcal{S}^{\infty} \times \mathcal{Z}_{BMO}^{2}$ by
\[
\trinm{(Y,Z)}_\alpha^{2} := \max_{i \in \{1,...,N\}}\bigl\{||e^{\tfrac{1}{2}\alpha\cdot}Y^{i}||_{\mathcal{S}^{\infty}}^{2},||e^{\tfrac{1}{2}\alpha\cdot}Z^{i}||_{\mathcal{Z}_{BMO}^{2}}^{2}\bigr\},
\]
where we denote by $Z^{i}$ the $i$-th row of the matrix $Z$ and $e^{\tfrac{1}{2}\alpha\cdot}$ denotes the process $e^{\tfrac{1}{2}\alpha t}$ for $t \in [0,T]$.

\begin{Theorem}
\label{Thm3.1}
Suppose Assumption \ref{assu:well-posedBSDE} holds and $\beta$ is given by \textnormal{(iv)}. Then the BSDE system \eqref{equ:3.genBSDE} has a unique solution $(Y^{M},Z^{M})$ inside $\mathcal{B}_{R}$, where  
\[
\mathcal{B}_{R} := \bigl\{(Y,Z) \in \mathcal{S}^{\infty} \times \mathcal{Z}_{BMO}^{2} : \trinm{(Y,Z)}_\beta^{2} \leq e^{\beta T}R^{2} \bigr\}.
\]
\end{Theorem}

\begin{Remark}
In \cite[Example 2.3]{RefN12} it was proven that there exists a two-dimensional quadratic BSDE with terminal condition $\xi = 0$ and driver satisfying $f(0)=0$ that has a non-trivial solution $(Y,Z)$. Therefore, in general, there can exist a second solution outside of $\mathcal{B}_{R}$ for multi-dimensional qBSDEs.
\end{Remark}


\begin{proof}

For simplicity of notation, we will omit the index $M$. We start by defining the function $\Phi:\mathcal{B}_{R} \rightarrow \mathcal{B}_{R}$, where for $(y,z) \in \mathcal{B}_{R}$ we define $\Phi(y,z)$ as the solution to the BSDE system
\be\label{eq:solmap}
Y_{t}^{i} = \xi_{i} + \int_{t}^{T}f^{i}(r,y_{r},z_{r})dr -\int_{t}^{T}\sum_{j=1}^{d}Z_{r}^{ij}dW_{r}^{j} \ \ , \ i \in \{1,...,N\}.
\ee
We first want to show that $\Phi$ is well-defined. We can apply the Itô-formula to $e^{\beta t}(Y_{t}^{i})^{2}$ and take the conditional expectation $\mathbb{E}_{t}[...]$ on both sides to get
\[
e^{\beta t}(Y_{t}^{i})^{2} + \mathbb{E}_{t}\Bigl[\int_{t}^{T}\beta e^{\beta r}(Y_{r}^{i})^{2}dr\Bigr] + \mathbb{E}_{t}\Bigl[\int_{t}^{T}e^{\beta r}\sum_{j=1}^{d}(Z_{r}^{ij})^{2}dr\Bigr] = \mathbb{E}_{t}[e^{\beta T}\xi_{i}^{2}] + \mathbb{E}_{t}\Bigl[\int_{t}^{T}2e^{\beta r}Y_{r}^{i}f^{i}(r,y_{r},z_{r})dr\Bigr],
\]
where the right side is bounded by
\begin{align}
&e^{\beta T}||\xi_{i}||_{L^{\infty}}^{2} + \mathbb{E}_{t}\Bigl[\int_{t}^{T}2e^{\beta r}|Y_{r}^{i}| \cdot |f^{i}(r,y_{r},z_{r})-f^{i}(r,0,0)|dr\Bigr] + \mathbb{E}_{t}\Bigl[\int_{t}^{T}2e^{\beta r}|Y_{r}^{i}f^{i}(r,0,0)|dr\Bigr] \nonumber \\
&\leq e^{\beta T}||\xi_{i}||_{L^{\infty}}^{2} + \mathbb{E}_{t}\Bigl[\int_{t}^{T}2e^{\beta r}|Y_{r}^{i}|\bigl(|g^{i}(r,y_{r},z_{r})-g^{i}(r,0,0)|+|h^{i}(r,y_{r},z_{r})-h^{i}(r,0,0)|\bigr)dr\Bigr] \nonumber \\
&\ \ \ + \mathbb{E}_{t}\Bigl[\int_{t}^{T}2e^{\beta r}|Y_{r}^{i}| \cdot |l^{i}(r,y_{r},z_{r})-l^{i}(r,0,0)|dr\Bigr] + \mathbb{E}_{t}\Bigl[\int_{t}^{T}2e^{\beta r}|Y_{r}^{i}f^{i}(r,0,0)|dr\Bigr] \nonumber \\
&\leq e^{\beta T}||\xi_{i}||_{L^{\infty}}^{2} + \tfrac{1}{2}||Y^{i}||_{\mathcal{S}^{\infty}}^{2} + 2\biggr(\mathbb{E}_{t}\Bigl[\int_{t}^{T}e^{\beta r}\bigl(|g^{i}(r,y_{r},z_{r})-g^{i}(r,0,0)|+|h^{i}(r,y_{r},z_{r})-h^{i}(r,0,0)|\bigr)dr\Bigr]\biggr)^{2} \nonumber \\
& \ \ \ + \mathbb{E}_{t}\Bigl[\int_{t}^{T}2e^{\beta r}|Y_{r}^{i}| \cdot |l^{i}(r,y_{r},z_{r})-l^{i}(r,0,0)|dr\Bigr] + \mathbb{E}_{t}\Bigl[\int_{t}^{T}2e^{\beta r}|Y_{r}^{i}|\cdot |f^{i}(r,0,0)|dr\Bigr] \label{equ:misc1}.
\end{align}
Furthermore, for the third term on the right hand side of the above inequality, by Assumption \ref{assu:well-posedBSDE} (ii), (iii), and Young's inequality, we get
\begin{align}
&\biggr(\mathbb{E}_{t}\Bigl[\int_{t}^{T}e^{\beta r}|g^{i}(r,y_{r},z_{r})-g^{i}(r,0,0)|dr\Bigr]\biggr)^{2} \nonumber \\
&\leq 2\biggr(\mathbb{E}_{t}\Bigl[\int_{t}^{T}e^{\beta r}\sum_{j=1}^{N}C_{i}^{Y,j}(y_{r}^{j})^{2}dr\Bigr]\biggr)^{2} + 2\biggr(\mathbb{E}_{t}\Bigl[\int_{t}^{T}e^{\beta r}\sum_{j=1}^{N}\sum_{k=1}^{d}C_{i}^{Z,jk}(z_{r}^{jk})^{2}dr\Bigr]\biggr)^{2} \nonumber \\
&\leq 2T^{2}\Bigl(\sum_{j=1}^{N}C_{i}^{Y,j}||e^{\tfrac{1}{2}\beta\cdot}y^{j}||_{\mathcal{S}^{\infty}}^{2}\Bigr)^{2} + 2\Bigl(\sum_{j=1}^{N}\max_{k \in \{1,...,d\}}C_{i}^{Z,jk}||e^{\tfrac{1}{2}\beta \cdot}z^{j}||_{\mathcal{Z}_{BMO}^{2}}^{2}\Bigr)^{2}, \label{equ:misc2}
\end{align}
and similarly
\begin{align}
&\biggr(\mathbb{E}_{t}\Bigl[\int_{t}^{T}e^{\beta r}|h^{i}(r,y_{r},z_{r})-h^{i}(r,0,0)|dr\Bigr]\biggr)^{2} \nonumber \\
&\leq 4\biggr(\mathbb{E}_{t}\Bigl[\int_{t}^{T}e^{\beta r}\Bigl(\sum_{j=1}^{N}\tilde{C}_{i}^{Y,j}\Bigr)\Bigl(\sum_{j=1}^{N}\tilde{C}_{i}^{Y,j}(y_{r}^{j})^{2}\Bigr)dr\Bigr]\biggr)^{2} \nonumber \\
&\ \ \ +4\biggr(\mathbb{E}_{t}\Bigl[\int_{t}^{T}e^{\beta r}\Bigl(\sum_{j=1}^{N}\sum_{k=1}^{d}\tilde{C}_{i}^{Z,jk}\Bigr)\Bigl(\sum_{j=1}^{N}\sum_{k=1}^{d}\tilde{C}_{i}^{Z,jk}(z_{r}^{jk})^{2}\Bigr)dr\Bigr]\biggr)^{2} \nonumber \\
&\leq 4C^{2}\Bigl(\sum_{j=1}^{N}\tilde{C}_{i}^{Y,j}||e^{\tfrac{1}{2}\beta\cdot}y^{j}||_{\mathcal{S}^{\infty}}^{2}\Bigr)^{2} + 4C^{2}\Bigl(\sum_{j=1}^{N}\max_{k \in \{1,...,d\}}\tilde{C}_{i}^{Z,jk}||e^{\tfrac{1}{2}\beta \cdot}z^{j}||_{\mathcal{Z}_{BMO}^{2}}^{2}\Bigr)^{2}. \label{equ:misc3}
\end{align}
Furthermore, for the last two terms on the right-hand side of \eqref{equ:misc1}, we have
\begin{align}
&\mathbb{E}_{t}\Bigl[\int_{t}^{T}2e^{\beta r}|Y_{r}^{i}|\cdot|l^{i}(r,y_{r},z_{r})-l^{i}(r,0,0)|dr\Bigr] + \mathbb{E}_{t}\Bigl[\int_{t}^{T}2e^{\beta r}|Y_{r}^{i}|\cdot |f^{i}(r,0,0)|dr\Bigr] \nonumber \\
&\leq \mathbb{E}_{t}\Bigl[\int_{t}^{T}\beta e^{\beta r}(Y_{r}^{i})^{2}dr\Bigr] + \tfrac{2}{\beta}\mathbb{E}_{t}\Bigl[\int_{t}^{T}e^{\beta r}(l^{i}(r,y_{r},z_{r})-l^{i}(r,0,0))^{2}dr\Bigr] + \tfrac{2}{\beta}\mathbb{E}_{t}\Bigl[\int_{t}^{T}e^{\beta r}(f^{i}(r,0,0))^{2}dr\Bigr], \nonumber
\end{align}
which is bounded by
\begin{align}
&\mathbb{E}_{t}\Bigl[\int_{t}^{T}\beta e^{\beta r}(Y_{r}^{i})^{2}dr\Bigr] + \tfrac{4LT}{\beta}\Bigl(\sum_{j=1}^{N}L_{i}^{Y,j}||e^{\tfrac{1}{2}\beta\cdot}y^{j}||_{\mathcal{S}^{\infty}}^{2}\Bigr) + \tfrac{4L}{\beta}\Bigl(\sum_{j=1}^{N}\max_{k \in \{1,...,d\}}L_{i}^{Z,jk}||e^{\tfrac{1}{2}\beta \cdot}z^{j}||_{\mathcal{Z}_{BMO}^{2}}^{2}\Bigr) \nonumber \\
&+\tfrac{2}{\beta}\mathbb{E}_{t}\Bigl[\int_{t}^{T}e^{\beta r}(f^{i}(r,0,0))^{2}dr\Bigr]. \label{equ:miscL}
\end{align}
Now, putting (\ref{equ:misc2}), (\ref{equ:misc3}), and (\ref{equ:miscL}) into (\ref{equ:misc1}) and applying Assumption \ref{assu:well-posedBSDE} (iv), it follows that
\begin{align*}
&e^{\beta t}(Y_{t}^{i})^{2} + \mathbb{E}_{t}\Bigl[\int_{t}^{T}\beta e^{\beta r}(Y_{r}^{i})^{2}dr\Bigr] + \mathbb{E}_{t}\Bigl[\int_{t}^{T}e^{\beta r}\sum_{j=1}^{d}(Z_{r}^{ij})^{2}dr\Bigr] \\
&\leq \tfrac{1}{4}e^{\beta T}R^{2} + \tfrac{1}{2}||Y^{i}||_{\mathcal{S}^{\infty}}^{2} + 8T^{2}\Bigl(\sum_{j=1}^{N}C_{i}^{Y,j}||e^{\tfrac{1}{2}\beta\cdot}y^{j}||_{\mathcal{S}^{\infty}}^{2}\Bigr)^{2} + 8\Bigl(\sum_{j=1}^{N}\max_{k \in \{1,...,d\}}C_{i}^{Z,jk}||e^{\tfrac{1}{2}\beta\cdot}z^{j}||_{\mathcal{Z}_{BMO}^{2}}^{2}\Bigr)^{2} \\
&\ \ \ +16C^{2}\Bigl(\sum_{j=1}^{N}\tilde{C}_{i}^{Y,j}||e^{\tfrac{1}{2}\beta\cdot}y^{j}||_{\mathcal{S}^{\infty}}^{2}\Bigr)^{2} + 16C^{2}\Bigl(\sum_{j=1}^{N}\max_{k \in \{1,...,d\}}\tilde{C}_{i}^{Z,jk}||e^{\tfrac{1}{2}\beta\cdot}z^{j}||_{\mathcal{Z}_{BMO}^{2}}^{2}\Bigr)^{2} \\
&\ \ \ + \tfrac{4LT}{\beta}\Bigl(\sum_{j=1}^{N}L_{i}^{Y,j}||e^{\tfrac{1}{2}\beta\cdot}y^{j}||_{\mathcal{S}^{\infty}}^{2}\Bigr) + \tfrac{4L}{\beta}\Bigl(\sum_{j=1}^{N}\max_{k \in \{1,...,d\}}L_{i}^{Z,jk}||e^{\tfrac{1}{2}\beta\cdot}z^{j}||_{\mathcal{Z}_{BMO}^{2}}^{2}\Bigr) + \mathbb{E}_{t}\Bigl[\int_{t}^{T}\beta e^{\beta r}(Y_{r}^{i})^{2}dr\Bigr],
\end{align*}
which, if we assume that $||Y^{i}||_{\mathcal{S}^{\infty}}<\infty$ and use the fact that $||Y^{i}||_{\mathcal{S}^{\infty}}^{2} \leq ||e^{\tfrac{1}{2}\beta\cdot}Y^{i}||_{\mathcal{S}^{\infty}}^{2}$, implies
\begin{align}
&\max\Bigl\{||e^{\tfrac{1}{2}\beta\cdot}Y^{i}||_{\mathcal{S}^{\infty}}^{2} \ , \ ||e^{\tfrac{1}{2}\beta\cdot}Z^{i}||_{\mathcal{Z}_{BMO}^{2}}^{2}\Bigr\} \nonumber \\
&\leq \tfrac{1}{2}e^{\beta T}R^{2} + 16T^{2}\Bigl(\sum_{j=1}^{N}C_{i}^{Y,j}||e^{\tfrac{1}{2}\beta\cdot}y^{j}||_{\mathcal{S}^{\infty}}^{2}\Bigr)^{2} + 16\Bigl(\sum_{j=1}^{N}\max_{k \in \{1,...,d\}}C_{i}^{Z,jk}||e^{\tfrac{1}{2}\beta\cdot}z^{j}||_{\mathcal{Z}_{BMO}^{2}}^{2}\Bigr)^{2} \nonumber \\
&\ \ \ + 32C^{2}\Bigl(\sum_{j=1}^{N}\tilde{C}_{i}^{Y,j}||e^{\tfrac{1}{2}\beta\cdot}y^{j}||_{\mathcal{S}^{\infty}}^{2}\Bigr)^{2}  + 32C^{2}\Bigl(\sum_{j=1}^{N}\max_{k \in \{1,...,d\}}\tilde{C}_{i}^{Z,jk}||e^{\tfrac{1}{2}\beta\cdot}z^{j}||_{\mathcal{Z}_{BMO}^{2}}^{2}\Bigr)^{2} \nonumber \\
& \ \ \ + \tfrac{8LT}{\beta}\Bigl(\sum_{j=1}^{N}L_{i}^{Y,j}||e^{\tfrac{1}{2}\beta\cdot}y^{j}||_{\mathcal{S}^{\infty}}^{2}\Bigr) + \tfrac{8L}{\beta}\Bigl(\sum_{j=1}^{N}\max_{k \in \{1,...,d\}}L_{i}^{Z,jk}||e^{\tfrac{1}{2}\beta\cdot}z^{j}||_{\mathcal{Z}_{BMO}^{2}}^{2}\Bigr).\label{equ:misc4}
\end{align}
Now, because $(y,z) \in \mathcal{B}_{R}$, 
by (\ref{equ:misc4}) and Assumption \ref{assu:well-posedBSDE} (iii) and (iv), we get
\[
\max\Bigl\{||e^{\tfrac{1}{2}\beta\cdot}Y^{i}||_{\mathcal{S}^{\infty}}^{2} \ , \ ||e^{\tfrac{1}{2}\beta\cdot}Z^{i}||_{\mathcal{Z}_{BMO}^{2}}^{2}\Bigr\} \leq (\tfrac{1}{2}+\tfrac{8L^{2}(1+T)}{\beta})e^{\beta T}R^{2} + 32C^{2}(1+C^{2} + \tfrac{C^{2}}{T^{2}})e^{2\beta T}R^{4} \leq e^{\beta T}R^{2},
\]
which proves that $(Y,Z) \in \mathcal{B}_{R}$ and therefore that $\Phi$ is well-defined. Thus, it is only left to show that $||Y^{i}||_{\mathcal{S}^{\infty}}<\infty$. We get this immediately by taking the conditional expectation $\mathbb{E}_{t}[...]$ on both sides of the BSDE system \eqref{eq:solmap}.

 
Now we want to show that $\Phi$ is a contraction. For any  $(y^{1},z^{1}),(y^{2},z^{2}) \in \mathcal{B}_{R}$, set $(Y^{i},Z^{i}) = \Phi(y^{i},z^{i})$ (for $i \in \{1,2\}$), $\Delta Y := Y^{1}-Y^{2}$, and analogously $\Delta Z$, $\Delta y$, and $\Delta z$. Then, by taking the difference of the two respective BSDE systems, we get for every $i \in \{1,...,N\}$
\[
\Delta Y_{t}^{i} = \int_{t}^{T}\bigl(f^{i}(r,y_{r}^{1},z_{r}^{1})-f^{i}(r,y_{r}^{2},z_{r}^{2})\bigr)dr -\int_{t}^{T}\sum_{j=1}^{d}\Delta Z_{r}^{ij}dW_{r}^{j}.
\]
Then, applying the Itô-formula to $e^{\beta t}(\Delta Y_{t}^{i})^{2}$ and taking the conditional expectation $\mathbb{E}_{t}$ on both sides yields
\begin{align}
&e^{\beta t}(\Delta Y_{t}^{i})^{2} + \mathbb{E}_{t}\Bigl[\int_{t}^{T}\beta e^{\beta r}(\Delta Y_{r}^{i})^{2}dr\Bigr] + \mathbb{E}_{t}\Bigl[\int_{t}^{T}e^{\beta r}\sum_{j=1}^{d}(\Delta Z_{r}^{ij})^{2}dr\Bigr] \nonumber \\
&\leq \mathbb{E}_{t}\biggl[\int_{t}^{T}2e^{\beta r}|\Delta Y_{r}^{i}|\Bigl(|g^{i}(r,y_{r}^{1},z_{r}^{1})-g^{i}(r,y_{r}^{2},z_{r}^{2})\bigr| + |h^{i}(r,y_{r}^{1},z_{r}^{1})-h^{i}(r,y_{r}^{2},z_{r}^{2})\bigr|\Bigr)dr\biggr] \nonumber \\
&\ \ \ + \mathbb{E}_{t}\Bigl[\int_{t}^{T}2e^{\beta r}|\Delta Y_{r}^{i}|\cdot|l^{i}(r,y_{r}^{1},z_{r}^{1})-l^{i}(r,y_{r}^{2},z_{r}^{2})\bigr|dr\Bigr] \nonumber \\
&\leq \tfrac{1}{2}||\Delta Y^{i}||_{\mathcal{S}^{\infty}}^{2} + 4\biggl(\mathbb{E}_{t}\Bigl[\int_{t}^{T}e^{\beta r}\bigl|g^{i}(r,y_{r}^{1},z_{r}^{1})-g^{i}(r,y_{r}^{2},z_{r}^{2})\bigr|dr\Bigr]\biggr)^{2}  \label{equ:misc7} \\
& \ \ \ + 4\biggl(\mathbb{E}_{t}\Bigl[\int_{t}^{T}e^{\beta r}\bigl|h^{i}(r,y_{r}^{1},z_{r}^{1})-h^{i}(r,y_{r}^{2},z_{r}^{2})\bigr|dr\Bigr]\biggr)^{2} 
+ \mathbb{E}_{t}\Bigl[\int_{t}^{T}2e^{\beta r}|\Delta Y_{r}^{i}|\cdot|l^{i}(r,y_{r}^{1},z_{r}^{1})-l^{i}(r,y_{r}^{2},z_{r}^{2})\bigr|dr\Bigr]. \nonumber 
\end{align}
For the second term on the right-hand side of \eqref{equ:misc7}, by H\"older's inequality, we have
\begin{align*}\nonumber
&\biggl(\mathbb{E}_{t}\Bigl[\int_{t}^{T}e^{\beta r}\bigl|g^{i}(r,y_{r}^{1},z_{r}^{1})-g^{i}(r,y_{r}^{2},z_{r}^{2})\bigr|dr\Bigr]\biggr)^{2} \nonumber \\
&\leq 2\biggl(\mathbb{E}_{t}\Bigl[\int_{t}^{T}e^{\beta r}\sum_{j=1}^{N}C_{i}^{Y,j}|\Delta y_{r}^{j}|\bigl(|y_{r}^{1,j}| + |y_{r}^{2,j}|\bigr)dr\Bigr]\biggr)^{2} \\
&\ \ \ + 2\biggl(\mathbb{E}_{t}\Bigl[\int_{t}^{T}e^{\beta r}\sum_{j=1}^{N}\sum_{k=1}^{d}C_{i}^{Z,jk}|\Delta z_{r}^{jk}|\bigl(|z_{r}^{1,jk}| + |z_{r}^{2,jk}|\bigr)dr\Bigr]\biggr)^{2} \nonumber \\
&\leq 2\mathbb{E}_{t}\Bigl[\int_{t}^{T}e^{\beta r}\sum_{j=1}^{N}C_{i}^{Y,j}(\Delta y_{r}^{j})^{2}dr\Bigr]\mathbb{E}_{t}\Bigl[\int_{t}^{T}e^{\beta r}\sum_{j=1}^{N}C_{i}^{Y,j}\bigl(|y_{r}^{1,j}| + |y_{r}^{2,j}|\bigr)^{2}dr\Bigr] \nonumber \\
&\ \ \ + 2\mathbb{E}_{t}\Bigl[\int_{t}^{T}e^{\beta r}\sum_{j=1}^{N}\sum_{k=1}^{d}C_{i}^{Z,jk}(\Delta z_{r}^{jk})^{2}dr\Bigr]\mathbb{E}_{t}\Bigl[\int_{t}^{T}e^{\beta r}\sum_{j=1}^{N}\sum_{k=1}^{d}C_{i}^{Z,jk}\bigl(|z_{r}^{1,jk}| + |z_{r}^{2,jk}|\bigr)^{2}dr\Bigr]=:\Delta G_t. \nonumber
\end{align*}
Furthermore, by Assumption \ref{assu:well-posedBSDE} (ii) and (iii), for the above term, we have
\begin{align}
\Delta G_t&\leq 4\Bigl(T\sum_{j=1}^{N}C_{i}^{Y,j}||e^{\tfrac{1}{2}\beta\cdot}\Delta y^{j}||_{\mathcal{S}^{\infty}}^{2}\Bigr)\Bigr(T\sum_{j=1}^{N}C_{i}^{Y,j}\bigl(||e^{\tfrac{1}{2}\beta\cdot}y^{1,j}||_{\mathcal{S}^{\infty}}^{2} + ||e^{\tfrac{1}{2}\beta\cdot}y^{2,j}||_{\mathcal{S}^{\infty}}^{2}\bigr)\Bigr) \nonumber \\
&+ 4\Bigl(\sum_{j=1}^{N}\max_{k \in \{1,...,d\}}C_{i}^{Z,jk}||e^{\tfrac{1}{2}\beta\cdot}\Delta z^{j}||_{\mathcal{Z}_{BMO}^{2}}^{2}\Bigr)\Bigl(\sum_{j=1}^{N}\max_{k \in \{1,...,d\}}C_{i}^{Z,jk}\bigl(||e^{\tfrac{1}{2}\beta\cdot}z^{1,j}||_{\mathcal{Z}_{BMO}^{2}}^{2} + ||e^{\tfrac{1}{2}\beta\cdot}z^{2,j}||_{\mathcal{Z}_{BMO}^{2}}^{2}\bigr)\Bigr) \nonumber \\
&\leq 16C^{2}e^{\beta T}R^{2}\max_{i \in \{1,...,N\}}\Bigl\{||e^{\tfrac{1}{2}\beta\cdot}\Delta y^{i}||_{\mathcal{S}^{\infty}}^{2} \ , \ ||e^{\tfrac{1}{2}\beta\cdot}\Delta z^{i}||_{\mathcal{Z}_{BMO}^{2}}^{2}\Bigr\}. \label{equ:misc8}
\end{align}
For the third term on the right hand side of \eqref{equ:misc7}, similarly, we have
\begin{align*}
&\biggl(\mathbb{E}_{t}\Bigl[\int_{t}^{T}e^{\beta r}\bigl|h^{i}(r,y_{r}^{1},z_{r}^{1})-h^{i}(r,y_{r}^{2},z_{r}^{2})\bigr|dr\Bigr]\biggr)^{2} \nonumber \\
&\leq \biggl(2\mathbb{E}_{t}\Bigl[\int_{t}^{T}e^{\beta r}\Bigl(\sum_{j=1}^{N}\tilde{C}_{i}^{Y,j}|\Delta y_{r}^{j}|\Bigr)^{2}dr\Bigr] + 2\mathbb{E}_{t}\Bigl[\int_{t}^{T}e^{\beta r}\Bigl(\sum_{j=1}^{N}\sum_{k=1}^{d}\tilde{C}_{i}^{Z,jk}|\Delta z_{r}^{jk}|\Bigr)^{2}dr\Bigr]\biggr)  \nonumber \\
&\ \ \ \cdot \biggl(2\mathbb{E}_{t}\Bigl[\int_{t}^{T}e^{\beta r}\Bigl(\sum_{j=1}^{N}\tilde{C}_{i}^{Y,j}(|y_{r}^{1,j}|+|y_{r}^{2,j}|)\Bigr)^{2}dr\Bigr] + 2\mathbb{E}_{t}\Bigl[\int_{t}^{T}e^{\beta r}\Bigl(\sum_{j=1}^{N}\sum_{k=1}^{d}\tilde{C}_{i}^{Z,jk}(|z_{r}^{1,jk}|+|z_{r}^{2,jk}|)\Bigr)^{2}dr\Bigr]\biggr) \\
&=:\Delta H_t. \nonumber
\end{align*}
By applying Assumption \ref{assu:well-posedBSDE} (iii) and the fact that $(\sum_k a_k b_k)^2 \le \sum_k |a_k| \sum_k |a_k| | b_k|^2$, we get
\begin{align}
\Delta H_t &\leq \biggl(2C\Bigl(\sum_{j=1}^{N}\tilde{C}_{i}^{Y,j}||e^{\tfrac{1}{2}\beta\cdot}\Delta y^{j}||_{\mathcal{S}^{\infty}}^{2}\Bigr) + 2C\Bigl(\sum_{j=1}^{N}\max_{k \in \{1,...,d\}}\tilde{C}_{i}^{Z,jk}||e^{\tfrac{1}{2}\beta\cdot}\Delta z^{j}||_{\mathcal{Z}_{BMO}^{2}}^{2}\Bigr)\biggr) \nonumber \\
&\ \ \ \cdot \biggl(4C\Bigl(\sum_{j=1}^{N}\tilde{C}_{i}^{Y,j}(||e^{\tfrac{1}{2}\beta\cdot}y^{1,j}||_{\mathcal{S}^{\infty}}^{2} + ||e^{\tfrac{1}{2}\beta\cdot}y^{2,j}||_{\mathcal{S}^{\infty}}^{2})\Bigr) \nonumber \\
&\ \ \  + 4C\Bigl(\sum_{j=1}^{N}\max_{k \in \{1,...,d\}}\tilde{C}_{i}^{Z,jk}(||e^{\tfrac{1}{2}\beta\cdot}z^{1,j}||_{\mathcal{Z}_{BMO}^{2}}^{2} + ||e^{\tfrac{1}{2}\beta\cdot}z^{2,j}||_{\mathcal{Z}_{BMO}^{2}}^{2})\Bigr)\biggr) \nonumber \\
&\leq \Bigl(\tfrac{2C^{2}}{T}\max_{i \in \{1,...,N\}}||e^{\tfrac{1}{2}\beta\cdot}\Delta y^{i}||_{\mathcal{S}^{\infty}}^{2} + 2C^{2}\max_{i \in \{1,...,N\}}||e^{\tfrac{1}{2}\beta\cdot}\Delta z^{i}||_{\mathcal{Z}_{BMO}^{2}}^{2}\Bigr)\bigl(\tfrac{8C^{2}}{T} + 8C^{2}\bigr)e^{\beta T}R^{2} \nonumber \\
&\leq 16C^{4}(1+\tfrac{1}{T})^{2}e^{\beta T}R^{2}\max_{i \in \{1,...,N\}}\Bigl\{||e^{\tfrac{1}{2}\beta\cdot}\Delta y^{i}||_{\mathcal{S}^{\infty}}^{2} \ , \ ||e^{\tfrac{1}{2}\beta\cdot}\Delta z^{i}||_{\mathcal{Z}_{BMO}^{2}}^{2}\Bigr\}. \label{equ:misc9}
\end{align}
For the last term on the right hand side of \eqref{equ:misc7}, we have
\begin{align}
&\mathbb{E}_{t}\Bigl[\int_{t}^{T}e^{\beta r}|\Delta Y_{r}^{i}|\cdot\bigl|l^{i}(r,y_{r}^{1},z_{r}^{1})-l^{i}(r,y_{r}^{2},z_{r}^{2})\bigr|dr\Bigr] \nonumber \\
&\leq \mathbb{E}_{t}\Bigl[\int_{t}^{T}\beta e^{\beta r}(\Delta Y_{r}^{i})^{2}dr\Bigr] + \tfrac{1}{\beta}\mathbb{E}_{t}\Bigl[\int_{t}^{T}e^{\beta r}(l^{i}(r,y_{r}^{1},z_{r}^{1})-l^{i}(r,y_{r}^{2},z_{r}^{2}))^{2}dr\Bigr] \nonumber \\
&\leq \mathbb{E}_{t}\Bigl[\int_{t}^{T}\beta e^{\beta r}(\Delta Y_{r}^{i})^{2}dr\Bigr] + \tfrac{2L}{\beta}\mathbb{E}_{t}\Bigl[\int_{t}^{T}e^{\beta r}\sum_{j=1}^{N}L_{i}^{Y,j}(\Delta y_{r}^{j})^{2}dr\Bigr] + \tfrac{2L}{\beta}\mathbb{E}_{t}\Bigl[\int_{t}^{T}e^{\beta r}\sum_{j=1}^{N}\sum_{k=1}^{d}L_{i}^{Z,jk}(\Delta z_{r}^{jk})^{2}dr\Bigr] \nonumber \\
&\leq \mathbb{E}_{t}\Bigl[\int_{t}^{T}\beta e^{\beta r}(\Delta Y_{r}^{i})^{2}dr\Bigr] + \tfrac{2L^{2}(1+T)}{\beta}\max_{i \in \{1,...,N\}}\Bigl\{||e^{\tfrac{1}{2}\beta\cdot}\Delta y^{i}||_{\mathcal{S}^{\infty}}^{2} \ , \ ||e^{\tfrac{1}{2}\beta\cdot}\Delta z^{i}||_{\mathcal{Z}_{BMO}^{2}}^{2}\Bigr\}. \label{equ:miscL2}
\end{align}
Now, putting (\ref{equ:misc8}), (\ref{equ:misc9}), and (\ref{equ:miscL2}) into (\ref{equ:misc7}) yields
\begin{align*}
&e^{\beta t}(\Delta Y_{t}^{i})^{2} + \mathbb{E}_{t}\Bigl[\int_{t}^{T}\beta e^{\beta r}(\Delta Y_{r}^{i})^{2}dr\Bigr] + \mathbb{E}_{t}\Bigl[\int_{t}^{T}e^{\beta r}\sum_{j=1}^{d}(\Delta Z_{r}^{ij})^{2}dr\Bigr] \\
&\leq \tfrac{1}{2}||\Delta Y^{i}||_{\mathcal{S}^{\infty}}^{2} + \mathbb{E}_{t}\Bigl[\int_{t}^{T}\beta e^{\beta r}(\Delta Y_{r}^{i})^{2}dr\Bigr] \\
&\ \ \ + \Bigl(64C^{2}\Bigl(1+C^{2}(1+\tfrac{1}{T})^{2}\Bigr)e^{\beta T}R^{2} + \tfrac{2L^{2}(1+T)}{\beta}\Bigr)\max_{i \in \{1,...,N\}}\Bigl\{||e^{\tfrac{1}{2}\beta\cdot}\Delta y^{i}||_{\mathcal{S}^{\infty}}^{2} \ , \ ||e^{\tfrac{1}{2}\beta\cdot}\Delta z^{i}||_{\mathcal{Z}_{BMO}^{2}}^{2}\Bigr\},
\end{align*}
which implies
\begin{equation}
\label{equ:Ndimcontraction}
\trinm{(\Delta Y,\Delta Z)}_\beta^{2} \leq \Bigl(128C^{2}\Bigl(1+C^{2}(1+\tfrac{1}{T})^{2}\Bigr)e^{\beta T}R^{2} + \tfrac{4L^{2}(1+T)}{\beta}\Bigr)\trinm{(\Delta y,\Delta z)}_\beta^{2}.
\end{equation}
By Assumption \ref{assu:well-posedBSDE} (iv), $\Phi$ is a contraction. 

\end{proof}

The following corollary is vital for our stability result (i.e. Theorem \ref{Thm3.9}) and later for the convergence of multi-player equilibria to the mean-field limit. In the following, we omit the index $M$ in parameters of BSDEs for simplicity of notation.

\begin{Corollary}
\label{Cor3.3}
Let the assumptions from Theorem \ref{Thm3.1} hold and define the Picard iteration for BSDE \textnormal{(\ref{equ:3.genBSDE})} by setting $(Y^{(0)},Z^{(0)}):= (0,0)$ and defining $(Y^{(n+1)},Z^{(n+1)})$ as the unique solution to the $N$-dimensional BSDE
\[
Y_{t}^{(n+1)} = \xi + \int_{t}^{T}f_{r}(Y_{r}^{(n)},Z_{r}^{(n)})dr - \int_{t}^{T}Z_{r}^{(n+1)}dW_{r},\ \ \ n \in \mathbb{N}.
\]
Then, there exists a constant $0<M_{*}< 1$ that is independent of the dimensions $N$ and $d$, such that for the unique solution $(Y,Z)$ from Theorem \ref{Thm3.1} and for every $n$, it holds
\begin{equation}
\label{NdimPicard}
\trinm{(Y,Z)-(Y^{(n)},Z^{(n)}) }_\beta  \leq \tfrac{R}{1-M_{*}}e^{\tfrac{1}{2}\beta T}M_{*}^{n},
\end{equation}
where $\beta$ and $R$ are defined as in Theorem \ref{Thm3.1}.
\end{Corollary}
\begin{proof}
    
Let $C$, $L$ be defined as in Theorem \ref{Thm3.1} and set
\[
M_{*}^{2} := 128C^{2}\Bigl(1+C^{2}(1+\tfrac{1}{T})^{2}\Bigr)e^{\beta T}R^{2} + \tfrac{4L^{2}(1+T)}{\beta}.
\]
Then, by Assumption \ref{assu:well-posedBSDE} (iv), we have $0<M_{*}<1$ and as in (\ref{equ:Ndimcontraction}) from the proof of Theorem \ref{Thm3.1}, it follows by our contraction argument that for every $n \geq 1$ it holds
\[
\trinm{(Y^{(n+1)},Z^{(n+1)}) - (Y^{(n)},Z^{(n)})}_\beta   \leq M_{*} \trinm{(Y^{(n)},Z^{(n)}) - (Y^{(n-1)},Z^{(n-1)})}_\beta ,
\]
and therefore in view of the fact that $\trinm{(Y^{(1)},Z^{(1)}) - (Y^{(0)},Z^{(0)})}_\beta = \trinm{ (Y^{(1)},Z^{(1)})}_\beta \leq e^{\tfrac{1}{2}\beta T}R$,
\[
\trinm{ (Y^{(n+1)},Z^{(n+1)}) - (Y^{(n)},Z^{(n)})}_\beta \leq e^{\tfrac{1}{2}\beta T}RM_{*}^{n}.
\]
By the above inequality, a standard calculation leads to 
$$
\trinm{ (Y^{(m)},Z^{(m)})-(Y^{(n)},Z^{(n)})}_\beta \le \tfrac{R}{1-M_{*}}e^{\tfrac{1}{2}\beta T}M_{*}^{n},
$$
which implies our result by taking $m$ to infinity.

\end{proof}

\subsubsection{Well-posedness of the benchmark example}\label{sec:3.1.1}

Now we want to apply this well-posedness result to our benchmark model \eqref{equ:qBSDE}, where 
the driver is defined in (\ref{equ:Def_f}), $\xi := (\xi_{1},...,\xi_{N})$, $Y_{t}$ are $\mathbb{R}^{N}$-valued, $Z_{t}$ is $\mathbb{R}^{N \times (N+1)}$-valued, and $W_{t} := (W^{0}_{t},...,W^{N}_{t})$ is an $(N+1)$-dimensional Brownian motion. Recall that its coefficients are given by \eqref{equ:nu1-5}-\eqref{equ:nu10}, with $(\bar{\sigma}_{t}^{i})^{2} := (\sigma_{t}^{i})^{2} + (\sigma_{t}^{0})^{2}$.

To simplify the multi-player model from Section \ref{sec2.1} and later introduce our mean-field approach (see Section \ref{sec4.1.1}), we assume the following for our model to ensure that the coefficients of the driver from (\ref{equ:Def_f}) are bounded.

\begin{Assumption}\label{assu:pi-model} 
\begin{itemize}

\item[\textnormal{(A1).}] $\eta>0$ is a constant and $\{\xi_{i}\}_{i=1}^N \in \mathcal{F}_{T}$ with each $\xi_i$ uniformly bounded by some $r>0$. 


\item[\textnormal{(A2).}] $ \{\tilde{\mu}^{i}\}_{i \geq 1} $ {and} $ \{\sigma^{i}\}_{i \geq 1}$ are two families of bounded processes, such that for some $\delta_{\mu}>0$, $||\tilde{\mu}^{i}||_{L^{\infty}}\leq \delta_{\mu}$ for every $i$, and $ 0 < \delta_{\sigma} < \bar{\delta}_{\sigma}$ such that  for every $i$ and $t \in [0,T]$, \  $\delta_{\sigma}^{2} \leq (\sigma_{t}^{i})^{2}+(\sigma_{t}^{0})^{2} =: (\bar{\sigma}_{t}^{i})^{2} \leq \bar{\delta}_{\sigma}^{2}$.

\item[\textnormal{(A3).}] For every $t \in [0,T]$, 
$
\biggl|\tfrac{1}{\eta}\Bigl(\tfrac{1}{N}\sum_{j=1}^{N}(\bar{\sigma}_{t}^{j})^{-2}\Bigr)-1\biggr|
\geq \delta_{\sigma}. 
$
    
\end{itemize}


\end{Assumption}

To provide a bound for the coefficients $\{\nu^i_{t,j} \}_{j=1}^{10}$ of BSDE \eqref{equ:qBSDE}, we define
\begin{align}
&b_{1} := b_{2} := b_{3} := \tfrac{\eta}{2},\ b_{4} := b_{5} := \delta_{\sigma}^{-2}\bar{\delta}_{\sigma} , \ b_{6} := \tfrac{1}{2\eta}\delta_{\sigma}^{-4}
\label{equ:b123} \\
& b_{7} := b_{8} := \delta_{\mu}\bar{\delta}_{\sigma}\Bigl(\tfrac{1}{\eta}\delta_{\sigma}^{-5} + \delta_{\sigma}^{-2}\Bigr),\ b_{9} := \delta_{\mu}\Bigl(\tfrac{1}{\eta^{2}}\delta_{\sigma}^{-6} + \tfrac{1}{\eta}\delta_{\sigma}^{-3}\Bigr) \label{equ:b78} \\
&b_{10} := \delta_{\mu}^{2}\Bigl(\tfrac{1}{2\eta^{3}}\delta_{\sigma}^{-8}+\tfrac{1}{\eta^{2}}\delta_{\sigma}^{-5} + \tfrac{1}{2\eta}\delta_{\sigma}^{-2}\Bigr), \label{equ:b10}
\end{align}
and then for every $i=1,...,N,$ $|\nu_{t,j}^{i}|$ is bounded by $b_{j}$, respectively. To further simplify notation, define
\begin{equation}\label{equ:Defbounds}
\begin{split}
& b := \max_{1 \leq i \leq 6}\{b_{i}\} ,\  b_{\mu} := \max_{7 \leq i \leq 10}\{b_{i}\}, \ 
 \bL := 2\max\{1,b_{\mu}\}\Bigl(1+\tfrac{\bar{\delta}_{\sigma}}{\delta_{\sigma}^{2}}\Bigr) \\
&\bar \beta := 32\bL^{2}(1+T) , \ 
\bar{R}^{2} := 4 r^{2} + \tfrac{8T}{\bar \beta}b_{\mu}^{2} , \ 
\bC := \max\Bigl\{4\sqrt{b} + \tfrac{6\bar{\delta}_{\sigma}\sqrt{b}}{\delta_{\sigma}^{2}} \ , \ b + \tfrac{\eta}{2}\Bigr\}.    
\end{split}
\end{equation}
Then we have the following well-posedness result of qBSDE \eqref{equ:qBSDE}.

\begin{Corollary}
\label{Cor3.4}
Suppose Assumption \ref{assu:pi-model} holds, and moreover, assume that
\be\label{assu:modelbd}
256 {\bC}^{2}(1+2 \bC^{2}+\tfrac{2 \bC^{2}}{T^{2}})e^{ \bar{\beta} T} {\bR}^{2} < 1.
\ee
Then BSDE \eqref{equ:qBSDE} has a unique solution in $\mathcal{B}_{\bR}$, where
$
\mathcal{B}_{\bR} := \Bigl\{ (y,z) \in \mathcal{S}^{\infty} \times \mathcal{Z}_{BMO}^{2}: \trinm{(y,z)}_{\bar\beta}^{2} \leq e^{\bar\beta T}\bR^{2} \Bigr\}.
$
\end{Corollary}

\begin{proof}
We only need to show that Assumption \ref{assu:well-posedBSDE} (ii), (iii), and (iv) hold in this model.
For condition (ii), we split the driver $f^i$ into $\tilde{g}$, $\tilde{h}$, and $\tilde{l}$ by defining $f = \tilde{g} + \tilde{h} + \tilde{l}$ with, for any $i=1,...N$,
\begin{align*}
\tilde{g}_{t}^{i}(z) &:= \tfrac{\eta}{2}\sum_{j=0}^{N}(z^{ij})^{2} + \nu_{t,1}^{i}(z^{ii})^{2} + \nu_{t,2}^{i}(z^{i0})^{2} \\
\tilde{h}_{t}^{i}(z) &:= \nu_{t,3}^{i}z^{ii}z^{i0} + \nu_{t,4}^{i}z^{ii}\Bigl(\tfrac{1}{N}\sum_{j=1}^{N}\tfrac{\sigma_{t}^{j}z^{jj}+\sigma_{t}^{0}z^{j0}}{(\bar{\sigma}_{t}^{j})^{2}}\Bigr) + \nu_{t,5}^{i}z^{i0}\Bigl(\tfrac{1}{N}\sum_{j=1}^{N}\tfrac{\sigma_{t}^{j}z^{jj}+\sigma_{t}^{0}z^{j0}}{(\bar{\sigma}_{t}^{j})^{2}}\Bigr) + \nu_{t,6}^{i}\Bigl(\tfrac{1}{N}\sum_{j=1}^{N}\tfrac{\sigma_{t}^{j}z^{jj}+\sigma_{t}^{0}z^{j0}}{(\bar{\sigma}_{t}^{j})^{2}}\Bigr)^{2} \\
\tilde{l}_{t}^{i}(z) &:= \nu_{t,7}^{i}z^{ii} + \nu_{t,8}^{i}z^{i0} + \nu_{t,9}^{i}\Bigl(\tfrac{1}{N}\sum_{j=1}^{N}\tfrac{\sigma_{t}^{j}z^{jj}+\sigma_{t}^{0}z^{j0}}{(\bar{\sigma}_{t}^{j})^{2}}\Bigr) + \nu_{t,10}^{i}.
\end{align*}
To check condition (iii), for $z,\bar{z} \in \mathbb{R}^{N \times (N+1)}$, we have
\begin{align*}
\big|\tilde{g}_{t}^{i}(z)-\tilde{g}_{t}^{i}(\bar{z})\big| &= \bigg| \tfrac{\eta}{2}\sum_{j=0}^{N}\Delta z^{ij}(z^{ij}+\bar{z}^{ij}) + \nu_{t,1}^{i}\Delta z^{ii}(z^{ii} + \bar{z}^{ii}) + \nu_{t,2}\Delta z^{i0}(z^{i0}+\bar{z}^{i0})\bigg|,
\end{align*}
so we can choose
$C_{i}^{Z,jk} := \mathds{1}_{\{i=j\}}\bigl((b+\tfrac{\eta}{2})\mathds{1}_{\{k=0,k=j\}} + \tfrac{\eta}{2}\mathds{1}_{\{k \neq 0,k \neq j\}}\bigr).$
For the ``cross-term'' part $\tilde h$, we have
\begin{align*}
\big|\tilde{h}_{t}^{i}(z)-\tilde{h}_{t}^{i}(\bar{z})\big| &\leq \bigg|\nu_{t,3}^{i}z^{ii}\Delta z^{i0} + \nu_{t,3}^{i}\bar{z}^{i0}\Delta z^{ii} + \nu_{t,4}^{i}z^{ii}\Bigl(\tfrac{1}{N}\sum_{j=1}^{N}\tfrac{\sigma_{t}^{j}\Delta z^{jj}+\sigma_{t}^{0}\Delta z^{j0}}{(\bar{\sigma}_{t}^{j})^{2}}\Bigr) + \nu_{t,4}^{i}\Delta z^{ii} \\
&\ \ \ \cdot\Bigl(\tfrac{1}{N}\sum_{j=1}^{N}\tfrac{\sigma_{t}^{j}\bar{z}^{jj}+\sigma_{t}^{0}\bar{z}^{j0}}{(\bar{\sigma}_{t}^{j})^{2}}\Bigr) + \nu_{t,5}^{i}z^{i0}\Bigl(\tfrac{1}{N}\sum_{j=1}^{N}\tfrac{\sigma_{t}^{j}\Delta z^{jj}+\sigma_{t}^{0}\Delta z^{j0}}{(\bar{\sigma}_{t}^{j})^{2}}\Bigr) + \nu_{t,5}^{i}\Delta z^{i0}\Bigl(\tfrac{1}{N}\sum_{j=1}^{N}\tfrac{\sigma_{t}^{j}\bar{z}^{jj}+\sigma_{t}^{0}\bar{z}^{j0}}{(\bar{\sigma}_{t}^{j})^{2}}\Bigr) \\
&\ \ \ + \nu_{t,6}^{i}\Bigl(\tfrac{1}{N}\sum_{j=1}^{N}\tfrac{\sigma_{t}^{j}\Delta z^{jj}+\sigma_{t}^{0}\Delta z^{j0}}{(\bar{\sigma}_{t}^{j})^{2}}\Bigr)\Bigl(\tfrac{1}{N}\sum_{j=1}^{N}\tfrac{\sigma_{t}^{j}(z^{jj}+\bar{z}^{jj})+\sigma_{t}^{0}(z^{j0}+\bar{z}^{j0})}{(\bar{\sigma}_{t}^{j})^{2}}\Bigr)\bigg| \\
&\leq \sqrt{\max\{|\nu_{t,3}^{i}|,|\nu_{t,4}^{i}|,|\nu_{t,5}^{i}|,|\nu_{t,6}^{i}|\}}\Bigl(|\Delta z^{ii}| + |\Delta z^{i0}| + \tfrac{1}{N}\sum_{j=1}^{N}\tfrac{|\sigma_{t}^{j}| \cdot |\Delta z^{jj}|+|\sigma_{t}^{0}| \cdot |\Delta z^{j0}|}{(\bar{\sigma}_{t}^{j})^{2}}\Bigr) \\
&\ \ \ \cdot \biggl(\Bigl(\sqrt{|\nu_{t,3}^{i}|}+\sqrt{|\nu_{t,4}^{i}|}\Bigr)|z^{ii}| + \sqrt{|\nu_{t,5}^{i}|} \cdot |z^{i0}| + \sqrt{|\nu_{t,3}^{i}|} \cdot |\bar{z}^{i0}| + \Bigl(\sqrt{|\nu_{t,4}^{i}|}+\sqrt{|\nu_{t,5}^{i}|}\Bigr)\\
&\ \ \ \cdot\Bigl(\tfrac{1}{N}\sum_{j=1}^{N}\tfrac{|\sigma_{t}^{j}| \cdot |\bar{z}^{jj}|+|\sigma_{t}^{0}| \cdot |\bar{z}^{j0}|}{(\bar{\sigma}_{t}^{j})^{2}}\Bigr) + \sqrt{|\nu_{t,6}^{i}|} \cdot \tfrac{1}{N}\sum_{j=1}^{N}\tfrac{|\sigma_{t}^{j}|(|z^{jj}|+|\bar{z}^{jj}|)+|\sigma_{t}^{0}|(|z^{j0}|+|\bar{z}^{j0}|)}{(\bar{\sigma}_{t}^{j})^{2}}\biggr),
\end{align*}
so it is natural to choose $\tilde{C}_{i}^{Z,jk}:= \mathds{1}_{\{k=0,k=j\}}\bigl(2\sqrt{b}\mathds{1}_{\{i=j\}} + \tfrac{3\bar{\delta}_{\sigma}\sqrt{b}}{N\delta_{\sigma}^{2}}\bigr)$.
For the linear part, note that
$$
\big|\tilde{l}_{t}^{i}(z)-\tilde{l}_{t}^{i}(\bar{z})\big| = \bigg|\nu_{t,7}^{i}\Delta z^{ii} + \nu_{t,8}^{i}\Delta z^{i0} + \nu_{t,9}^{i}\Bigl(\tfrac{1}{N}\sum_{j=1}^{N}\tfrac{\sigma_{t}^{j}\Delta z^{jj}+\sigma_{t}^{0}\Delta z^{j0}}{(\bar{\sigma}_{t}^{j})^{2}}\Bigr)\bigg|,
$$
so it is easy to check
${L}_{i}^{Z,jk} := \mathds{1}_{\{k=0,k=j\}}\bigl(b_{\mu}  \mathds{1}_{\{i=j\}} + \tfrac{\bar{\delta}_{\sigma}}{N\delta_{\sigma}^{2}}b_{\mu}\bigr)$.
Then it is straightforward to calculate that
\begin{align*}
\sum_{j=1}^{N}\sum_{k=0}^{N} \tilde{C}_{i}^{Z,jk} = 4\sqrt{b} + \tfrac{6\bar{\delta}_{\sigma}\sqrt{b}}{\delta_{\sigma}^{2}} , \ \ \  \sum_{j=1}^{N}\max_{k \in \{0,...,N\}}C_{i}^{Z,jk} = b + \tfrac{\eta}{2} , \ \ \ \sum_{j=1}^{N}\sum_{k=0}^{N} L_{i}^{Z,jk} = 2b_{\mu}\Bigl(1+\tfrac{\bar{\delta}_{\sigma}}{\delta_{\sigma}^{2}}\Bigr),
\end{align*}
which implies that (iii) is satisfied if we choose
\[
\bC = \max\Bigl\{4\sqrt{b} + \tfrac{6\bar{\delta}_{\sigma}\sqrt{b}}{\delta_{\sigma}^{2}} \ , \ b + \tfrac{\eta}{2}\Bigr\}, \  \ \ 
\bL =   2\max\{1,b_{\mu}\}\Bigl(1+\tfrac{\bar{\delta}_{\sigma}}{\delta_{\sigma}^{2}}\Bigr).
\]
For the last condition (iv), note that
$
\tilde{R}^{2} := 4 \max_{i} ||\xi_{i}||_{L^{\infty}}^{2} + \tfrac{8}{\bar{\beta} }  \max_{i} \Bigl|\Bigl|\int_{0}^{T}(\nu_{r,10}^{i})^{2}dr\Bigr|\Bigr|_{L^{\infty}} \le \bR^{2},
$
which implies that (iv) follows from our assumption \eqref{assu:modelbd}.

\end{proof}

\begin{Remark}
\label{Rem3.5}
To ensure that \eqref{assu:modelbd} holds, we need $\bR$ or $\bC$ to be small. Indeed, small $\bR$ corresponds to the small terminal condition of BSDEs, and small $\bC$ corresponds to the small coefficients of the quadratic part. Furthermore, note that we can assume without loss of generality that the terminal condition of the BSDEs is a centred random variable since addition of any constant to the terminal condition does not affect the well-posedness.
 
\end{Remark}

\begin{Remark}
\label{Rem:trueeq-wellposed}
By Remark \ref{rem:NvsN-1}, a straightforward computation shows that if one uses $\mu^{N,\Balpha^{-i}}$ from \eqref{equ:defrealmu} in the driver defined in \eqref{eq:fop}, then the growth condition on the coefficients with respect to \(N\), required to satisfy the uniform boundedness condition in Assumption \ref{assu:well-posedBSDE} \textnormal{(iii)}, remains the same as for the driver given in \eqref{equ:Def_f}. Indeed,
$
\mu^{N,\hBalpha^{-i}}
=
\frac{N}{N-1}\mu^{N,\hBalpha}
-
\frac{1}{N-1}\hat{\alpha}^{i},
$
and the processes $\mu^{N,\hBalpha^{-i}}$, $\mu^{N,\hBalpha}$, and $\hat{\alpha}^{i}$ all admit the same $\mathcal{Z}_{BMO}^{2}$-bound.
Consequently, an analogous well-posedness result also holds when using the approach described in Remark \ref{rem:NvsN-1} to characterize the true equilibrium $\lps{*}\Balpha$, after suitably modifying the constants introduced in \eqref{equ:Defbounds}. It is important to note, however, that this alternative formulation requires the condition
\begin{equation}
\label{equ:a3'}
\bigg|
1-\sum_{k=1}^{N}\frac{b^{k}}{N-1+b^{k}}
\bigg|
\geq
\delta_{\sigma},
\end{equation}
where the coefficients \(b^i\) are defined as in Remark \ref{rem:NvsN-1}, instead of Assumption \ref{assu:pi-model} \textnormal{(A3)}, in order to guarantee boundedness of the corresponding coefficients in the driver.
\end{Remark}

\subsection{Stability of quadratic BSDEs with weak interaction}

In this subsection, we derive a quantitative stability result for solutions to \eqref{equ:3.genBSDE}. This result will serve as a key ingredient in proving the backward propagation of chaos and thus the convergence of the multi-player systems to the mean-field system, which is established in the next section. For an equation of the form \eqref{equ:3.genBSDE}, consider the following form of Picard iteration: for any $n\ge 0,$
\begin{align}
&Y_{t}^{(n+1)} = \xi + \int_{t}^{T}f_{r}(Y_{r}^{(n)},Z_{r}^{(n)})dr - \int_{t}^{T}Z_{r}^{(n+1)}dW_{r} \ , \ (Y^{(0)},Z^{(0)}) = (0,0). \label{equ:stab3}
\end{align}
Set $
\trinm{(Y,Z)}_{\beta,*} := \max_{i \in \{1,...,N\}} \left\{||e^{\tfrac{1}{2}\beta\cdot}Y^{i}||_{\mathcal{H}^{2}} \ , \ ||e^{\tfrac{1}{2}\beta\cdot}Z^{i}||_{\mathcal{H}^{2}} \right\}
$, and $(C, R, \beta, L)$ as in Theorem \ref{Thm3.1}.

\begin{Theorem}
\label{Thm3.9}

Consider two $N$-dimensional BSDEs in the form of \eqref{equ:3.genBSDE} driven by $(\xi, f)$ and $(\bar \xi, \bar f)$, respectively. Assume that Assumption \ref{assu:well-posedBSDE} holds for $(\xi,f)$ with $\beta \geq 2$. For the equation driven by $(\bar \xi, \bar f)$, assume that its Picard iteration \eqref{equ:stab3} satisfies $(\bar{Y}^{(n)},\bar{Z}^{(n)}) \in \mathcal{B}_{R}$ for each $n \in \mathbb{N}$. 
Define $\Delta \xi := \xi - \bar{\xi}$, and $  \Delta Y^{(n)}, \Delta Z^{(n)}$ analogously, and moreover, set $\Delta f_{r}^{(n),i} := f_{r}^{i}(\bar{Y}_{r}^{(n)},\bar{Z}_{r}^{(n)})-\bar{f}_{r}^{i}(\bar{Y}_{r}^{(n)},\bar{Z}_{r}^{(n)})$,
\begin{align*}
&D := 2e^{\beta T}\max_{i \in \{1,...,N\}}||\Delta \xi_{i}||_{L^{2}}^{2} + 8e^{\tfrac{3}{2}\beta T}R\sup_{n \in \mathbb{N},i \in \{1,...,N\}}\mathbb{E}\Bigl[\int_{0}^{T}|\Delta f_{r}^{(n),i}|dr\Bigr] \\
&C^{*} := 2Ce^{\tfrac{1}{2}\beta T}R(1+\tfrac{1}{\sqrt{T}}) + C\sqrt{2}(1+\tfrac{1}{T})\Bigl(\tfrac{8C^{2}}{T}e^{\beta T}R^{2} + 8C^{2}e^{\beta T}R^{2}\Bigr)^{\tfrac{1}{2}}.
\end{align*}
Assume further that $16C^{*}e^{\tfrac{1}{2}\beta T}R \leq 1$. Then, for every $n \in \mathbb{N}$, it holds
\be
\trinm{(\Delta Y^{(n)}, \Delta Z^{(n)})}_{\beta,*}  \le \left\{
\begin{array}{ll}
    c_n D^{2^{-n}}, \ \ &D \in (0,1]\\
    c_n D^{\frac12}, \ \ &D \in (1,\infty), 
\end{array}\right.
\ee
with $c_n= 1-(1- \frac{1}{\sqrt 2})^n.$ In particular, if $D \leq 1$ and, moreover, there exists $(\bar Y, \bar Z) \in  \mathcal{B}_{R}$ with
\begin{equation}
\label{equ:stab5}
\trinm{(\bar{Y},\bar{Z})-(\bar{Y}^{(n)},\bar{Z}^{(n)})}_{\beta} \leq \tilde{C}\tilde{M}_{*}^{n},
\end{equation}
for some constant $\tilde{C}>0$ and $\tilde{M}_{*} \in (0,1)$,
then for every $n \in \mathbb{N}$ it holds
\[
\trinm{(\Delta Y,\Delta Z)}_{\beta,*} \leq \max\{1,\sqrt{T}\}(\hat{C}M_{*}^{n} + \tilde{C}\tilde{M}_{*}^{n}) + D^{2^{-n}},
\]
where $\hat{C}>0$ is some constant, $M_{*} \in (0,1)$ is the constant from Corollary \ref{Cor3.3}, and $(Y,Z)$ is the unique solution of the BSDE driven by $(\xi,f)$ obtained from Theorem \ref{Thm3.1}.

\end{Theorem}

\begin{Remark}
   In the above theorem, $(\bar \xi, \bar f)$ does not need to satisfy Assumption \ref{assu:well-posedBSDE}. Later we see that the equation driven by $(\bar \xi, \bar f)$ is taken as the mean-field quadratic BSDE.
\end{Remark}

\begin{proof}
Taking the difference of the Picard scheme \eqref{equ:stab3}, we get for any $i \in \{1,...,N\}$
\begin{equation}
\label{equ:stab6}
\Delta Y_{t}^{(n+1),i} = \Delta \xi_{i} + \int_{t}^{T}\Bigl(f_{r}^{i}(Y_{r}^{(n)},Z_{r}^{(n)})-f_{r}^{i}(\bar{Y}_{r}^{(n)},\bar{Z}_{r}^{(n)}) + \Delta f_{r}^{(n),i}\Bigr)dr - \int_{t}^{T}\sum_{j=1}^{d}\Delta Z_{r}^{(n+1),ij}dW_{r}^{j}.
\end{equation}
Now we can apply the Itô formula to $e^{\beta t}(\Delta Y_{t}^{(n+1),i})^{2}$ and take the expectation on both sides to get
\begin{align}
&(\Delta Y_{0}^{(n+1),i})^{2} + \mathbb{E}\Bigl[\int_{0}^{T}\beta e^{\beta r}(\Delta Y_{r}^{(n+1),i})^{2}dr\Bigr] + \mathbb{E}\Bigl[\int_{0}^{T}e^{\beta r} \sum_{j=1}^{d}(\Delta Z_{r}^{(n+1),ij})^{2}dr\Bigr] \nonumber \\
&\leq e^{\beta T}||\Delta \xi_{i}||_{L^{2}}^{2} +4e^{\tfrac{3}{2}\beta T}R\mathbb{E}\Bigl[\int_{0}^{T}|\Delta f_{r}^{(n),i}|dr\Bigr] \nonumber \\
&\ \ \ + \mathbb{E}\Bigl[\int_{0}^{T}2e^{\beta r}\Delta Y_{r}^{(n+1),i}\bigl(f_{r}^{i}(Y_{r}^{(n)},Z_{r}^{(n)})-f_{r}^{i}(\bar{Y}_{r}^{(n)},\bar{Z}_{r}^{(n)})\bigr)dr\Bigr] \nonumber \\
&\leq \tfrac{1}{2}D + 4e^{\tfrac{1}{2}\beta T}R\mathbb{E}\Bigl[\int_{0}^{T}e^{\beta r}\Bigl|g_{r}^{i}(Y_{r}^{(n)},Z_{r}^{(n)})-g_{r}^{i}(\bar{Y}_{r}^{(n)},\bar{Z}_{r}^{(n)}) + h_{r}^{i}(Y_{r}^{(n)},Z_{r}^{(n)})-h_{r}^{i}(\bar{Y}_{r}^{(n)},\bar{Z}_{r}^{(n)})\Bigr|dr\Bigr] \nonumber \\
&\ \ \ + \mathbb{E}\Bigl[\int_{0}^{T}\tfrac{\beta}{2}e^{\beta r}(\Delta Y_{r}^{(n+1),i})^{2}dr\Bigr] + \tfrac{2}{\beta}\mathbb{E}\Bigl[\int_{0}^{T}e^{\beta r}\bigl(l_{r}^{i}(Y_{r}^{(n)},Z_{r}^{(n)})-l_{r}^{i}(\bar{Y}_{r}^{(n)},\bar{Z}_{r}^{(n)})\bigr)^{2}dr\Bigr], \label{equ:stab7}
\end{align}
where we recall the form $f = g+h+l$. 
For the second term on the right hand side of the above inequality, in view of Assumption \ref{assu:well-posedBSDE}, we have
\begin{align}
&\mathbb{E}\Bigl[\int_{0}^{T}e^{\beta r}\Bigl|g_{r}^{i}(Y_{r}^{(n)},Z_{r}^{(n)})-g_{r}^{i}(\bar{Y}_{r}^{(n)},\bar{Z}_{r}^{(n)}) + h_{r}^{i}(Y_{r}^{(n)},Z_{r}^{(n)})-h_{r}^{i}(\bar{Y}_{r}^{(n)},\bar{Z}_{r}^{(n)})\Bigr|dr\Bigr] \nonumber \\
&\leq \mathbb{E}\biggl[\int_{0}^{T}e^{\beta r}\biggl(\sum_{j=1}^{N}C_{i}^{Y,j}|\Delta Y_{r}^{(n),j}|\bigl(|Y_{r}^{(n),j}|+|\bar{Y}_{r}^{(n),j}|\bigr) + \sum_{j=1}^{N}\sum_{k=1}^{N}C_{i}^{Z,jk}|\Delta Z_{r}^{(n),jk}|\bigl(|Z_{r}^{(n),jk}|+|Z_{r}^{(n),jk}|\bigr) \nonumber \\
&\ \ \ + \biggl(\sum_{j=1}^{N}\tilde{C}_{i}^{Y,j}|\Delta Y_{r}^{(n),j}| + \sum_{j=1}^{N}\sum_{k=1}^{d}\tilde{C}_{i}^{Z,jk}|\Delta Z_{r}^{(n),jk}|\biggr) \nonumber \\
&\ \ \ \cdot\biggl(\sum_{j=1}^{N}\tilde{C}_{i}^{Y,j}\bigl(|Y_{r}^{(n),j}|+|\bar{Y}_{r}^{(n),j}|\bigr) + \sum_{j=1}^{N}\sum_{k=1}^{d}\tilde{C}_{i}^{Z,jk}\bigl(|Z_{r}^{(n),jk}|+|\bar{Z}_{r}^{(n),jk}|\bigr)\biggr)\biggr)dr\biggr] \nonumber \\
&\leq \sum_{j=1}^{N}C_{i}^{Y,j}\biggl(\mathbb{E}\Bigl[\int_{0}^{T}e^{\beta r}(\Delta Y_{r}^{(n),j})^{2}dr\Bigr]\biggr)^{\tfrac{1}{2}}\cdot\biggl(\mathbb{E}\Bigl[\int_{0}^{T}e^{\beta r}\bigl(|Y_{r}^{(n),j}|+|\bar{Y}_{r}^{(n),j}|\bigr)^{2}dr\Bigr]\biggr)^{\tfrac{1}{2}} \nonumber \\
&\ \ \ + \sum_{j=1}^{N}\max_{k \in \{1,...,d\}}C_{i}^{Z,jk}\biggl(\mathbb{E}\Bigl[\int_{0}^{T}e^{\beta r}\sum_{k=1}^{d}(\Delta Z_{r}^{(n),jk})^{2}dr\Bigr]\biggr)^{\tfrac{1}{2}}\cdot\biggl(\mathbb{E}\Bigl[\int_{0}^{T}e^{\beta r}\sum_{k=1}^{d}\bigl(|Z_{r}^{(n),jk}|+|Z_{r}^{(n),jk}|\bigr)^{2}dr\Bigr]\biggr)^{\tfrac{1}{2}} \nonumber \\
&\ \ \ + \biggl(\mathbb{E}\biggl[\int_{0}^{T}\biggl(2e^{\beta r}\Bigl(\sum_{j=1}^{N}\tilde{C}_{i}^{Y,j}|\Delta Y_{r}^{(n),j}|\Bigr)^{2} + 2e^{\beta r}\Bigl(\sum_{j=1}^{N}\sum_{k=1}^{d}\tilde{C}_{i}^{Z,jk}|\Delta Z_{r}^{(n),jk}|\Bigr)^{2}\biggr)dr\biggr]\biggr)^{\tfrac{1}{2}} \nonumber \\
&\ \ \ \cdot \biggl(\mathbb{E}\biggl[\int_{0}^{T}\biggl(2e^{\beta r}\Bigl(\sum_{j=1}^{N}\tilde{C}_{i}^{Y,j}\bigl(|Y_{r}^{(n),j}|+|\bar{Y}_{r}^{(n),j}|\bigr)\Bigr)^{2} + 2e^{\beta r}\Bigl(\sum_{j=1}^{N}\sum_{k=1}^{d}\tilde{C}_{i}^{Z,jk}\bigl(|Z_{r}^{(n),jk}|+|\bar{Z}_{r}^{(n),jk}|\bigr)\Bigr)^{2}\biggr)dr\biggr]\biggr)^{\tfrac{1}{2}} \nonumber \\
&\leq C^{*}\trinm{(\Delta Y^{(n)},\Delta Z^{(n)})}_{\beta,*}, \label{equ:stab8}
\end{align}
where in the last inequality we apply a similar argument as in \eqref{equ:misc3}.
For the fourth term, we have
\begin{align}
&\mathbb{E}\Bigl[\int_{0}^{T}e^{\beta r}\bigl(l_{r}^{i}(Y_{r}^{(n)},Z_{r}^{(n)})-l_{r}^{i}(\bar{Y}_{r}^{(n)},\bar{Z}_{r}^{(n)})\bigr)^{2}dr\Bigr] \nonumber \\
&\leq \mathbb{E}\Bigl[\int_{0}^{T}e^{\beta r}\Bigl(\sum_{j=1}^{N}L_{i}^{Y,j}|\Delta Y_{r}^{(n),j}| + \sum_{j=1}^{N}\sum_{k=1}^{d}L_{i}^{Z,jk}|\Delta Z_{r}^{(n),jk}|\Bigr)^{2}dr\Bigr] \nonumber \\
&\leq 4L^{2}\trinm{(\Delta Y^{(n)},\Delta Z^{(n)})}_{\beta,*}^{2}. \label{equ:stab8.5}
\end{align}
Now, putting (\ref{equ:stab8}) and (\ref{equ:stab8.5}) into (\ref{equ:stab7}) yields
\begin{align}
\trinm{(\Delta Y^{(n+1)},\Delta Z^{(n+1)})}_{\beta,*}^{2} &\leq \tfrac{1}{2}D + 4C^{*}e^{\tfrac{1}{2}\beta T}R\trinm{(\Delta Y^{(n)},\Delta Z^{(n)})}_{\beta,*} + \tfrac{8L^{2}}{\beta}\trinm{ (\Delta Y^{(n)},\Delta Z^{(n)})}_{\beta,*}^{2} \nonumber \\
&\leq \tfrac{1}{2}D +\tfrac{1}{4}\trinm{(\Delta Y^{(n)},\Delta Z^{(n)})}_{\beta,*} + \tfrac{1}{4}\trinm{(\Delta Y^{(n)},\Delta Z^{(n)})}_{\beta,*}^{2}, \label{equ:stab9}
\end{align}
where we used $\beta = 32L^{2}(1+T)$ and $\beta \geq 2$.
For simplicity of presentation, we restrict our attention to the case \(D\leq 1\). The case \(D>1\) is straightforward and is therefore omitted. We want to prove by induction over $n$ that
\begin{equation}
\label{equ:stab10}
\trinm{(\Delta Y^{(n)},\Delta Z^{(n)})}_{\beta,*} \leq c_n D^{2^{-n}}.
\end{equation}
The induction start for $n=0$ is trivial, since the left hand side is equal to zero. For the induction step, assume that (\ref{equ:stab10}) holds for some $n \in \mathbb{N}$. Then we get from (\ref{equ:stab9}) that
\begin{align*}
\trinm{(\Delta Y^{(n+1)},\Delta Z^{(n+1)})}_{\beta,*}^{2} &\leq \tfrac{1}{2}D + \tfrac{1}{4} c_n D^{2^{-n}} + \tfrac{1}{4} c_n^2 D^{2^{-n+1}}. 
\end{align*}
Note that 
$
\frac12 + \frac14 c_n + \frac14 c_n^2 \le (\frac{1}{\sqrt{2}} + (1- \frac{1}{\sqrt{2}})c_n)^2 = c_{n+1}^2.
$
It follows by the above two inequalities that
\[
\trinm{(\Delta Y^{(n+1)},\Delta Z^{(n+1)})}_{\beta,*} \leq c_{n+1} D^{2^{-(n+1)}},
\]
which completes the induction statement. Finally, because $\trinm{\cdot}_{\beta,*} \leq \max\{1,\sqrt{T}\}\trinm{\cdot}_{\beta}$, we get from Corollary \ref{Cor3.3} that
\begin{align*}
\trinm{(\Delta Y,\Delta Z)}_{\beta,*}  &\leq \trinm{(Y,Z)-(Y^{(n)},Z^{(n)}) }_\beta + \trinm{(\bar{Y},\bar{Z})-(\bar{Y}^{(n)},\bar{Z}^{(n)})  }_{\beta}  + \trinm{(\Delta Y^{(n)},\Delta Z^{(n)})}_{\beta,*} \\
&\leq \max\{1,\sqrt{T}\}(\hat{C}M_{*}^{n} + \tilde{C}\tilde{M}_{*}^{n}) + D^{2^{-n}}.
\end{align*}
\end{proof}

\begin{Remark}
This result gives stability in the sense that for every $\varepsilon > 0$, if $D$ gets sufficiently small, we can choose $n$ sufficiently large such that $\max\bigl\{\max\{1,\sqrt{T}\}(\hat{C}M_{*}^{n}+\tilde{C}\tilde{M}_{*}^{n}) , \ D^{2^{-n}}\bigr\} < \tfrac{\varepsilon}{2}$ and therefore $\trinm{(\Delta Y,\Delta Z)}_{\beta,*} < \varepsilon$. This is the precise argument we apply to show the convergence in Section \ref{sec4}. 
\end{Remark}



\section{Applications in Mean Field Games of Controls}
\label{sec4}

In this section, we apply our results of Section \ref{sec3} to our benchmark model introduced in Section \ref{sec2} and a (mathematically) closely related, yet economically very different model where asset prices are determined by a market-clearing condition. 

\subsection{Nash equilibria of the utility maximization model with price impact}
\label{sec4.1}

We start with the benchmark model introduced in Section \ref{sec2}. Our main result shows that the equilibrium of the $N$-player model (as well as the $\vep_N$-equilibrium) converges to the mean-field equilibrium, which will be introduced in the following. 

\subsubsection{Well-posedness of mean-field quadratic BSDEs with weak interaction}
\label{sec4.1.1}

First, we temporarily step away from the finite-player game and establish a general well-posedness result for one-dimensional mean-field type qBSDEs (mf-qBSDEs) with common noise, which will serve as the limiting equation for the multidimensional qBSDE system driven by \eqref{equ:Def_f}. Our framework extends the specific example studied in \cite[Theorem 4.1]{RefN10} and the more general form considered in \cite{haowenxiong22} without the common noise. In particular, our equations additionally allow for a linear component in the generator. Moreover, the construction of the solution through Picard iterations is tailored to fit the stability result established in Theorem \ref{Thm3.9}.
The multidimensional case can be proved analogously with the only difference being that the respective constants would also depend on the dimension of the Brownian motions.

Recall from Section \ref{sec1.1} that $\{W^j\}_{j=0}^N$ is an $(N+1)$-dimensional standard Brownian motion, and that we write $\E_t[\cdot]:= \E[\cdot |\MF_t]$. Moreover, for any random variable $\xi \in \MF,$ we write $\E^0[\xi]:= \E[\xi| \MF_T^0]$. Note that if $u_t \in \MF_t,$ we have $\E^0[u_t]= \E[u_t| \MF_T^0]= \E[u_t | \MF_t^0].$ Consider the following mf-qBSDEs
\begin{equation}
\label{equ:gMF.1}
\bar{Y}_{t}^{i} = \xi_{i} + \int_{t}^{T}\bar{f}_{r}^{i}(\bar{Y}_{r}^{i},\bar{Z}_{r}^{i},\bar{Z}_{r}^{i0})dr - \int_{t}^{T}(\bar{Z}_{r}^{i}dW_{r}^{i} + \bar{Z}_{r}^{i0}dW_{r}^{0}), \ \  \ i=1,...,N.
\end{equation}
Recall that for any $(y,(z,z^0)) \in \MS^\infty \times (\MZ_{BMO}^2)^{ \otimes 2}$,
\[
\trinm{(y,z,z^{0})}_{\beta}^{2} := \max\big\{||e^{\tfrac{1}{2}\beta\cdot}y||_{\mathcal{S}^{\infty}}^{2},||e^{\tfrac{1}{2}\beta\cdot}(z,z^{0})||_{\mathcal{Z}_{BMO}^{2}}^{2}\big\}.
\]
Note that the above multidimensional mean-field quadratic system is nothing but $N$ scalar-valued mean-field quadratic BSDEs.
In view of Assumption \ref{assu:well-posedBSDE}, here we suppose the following.
\begin{Assumption}\label{assu:mfbsde}

\begin{itemize} 
     
\item[\textnormal{(i)}] $\xi_{i} \in L^{\infty}(\MF^W_T)$ for every $i \in \{1,...,N\}$.

\item[\textnormal{(ii)}] ${\bar{f}}^i = {\bar{g}}^i + {\bar{l}}^i$ such that for every $(y,z,z^{0}),(\bar{y},\bar{z},\bar{z}^{0}) \in \mathcal{H}^{2} \times \mathcal{H}^{2} \times \mathcal{H}^{2}:$
\begin{align*}
&|\bar{g}_{t}^{i}(y,z_{1},z_{0})-\bar{g}_{t}^{i}(\bar{y},\bar{z},\bar{z}^{0})| \\
&\ \ \ \leq C\Bigl(|y|+|\bar{y}|+|z|+|\bar{z}|+|z^{0}|+|\bar{z}^{0}|+|\mathbb{E}^{0}[y]|+|\mathbb{E}^{0}[\bar{y}]|+|\mathbb{E}^{0}[z]|+|\mathbb{E}^{0}[\bar{z}]|+|\mathbb{E}^{0}[z^{0}]| \\
&\ \ \ \ \ \ +|\mathbb{E}^{0}[\bar{z}^{0}]|\Bigr)_t \Bigl(|y-\bar{y}|+|z-\bar{z}|+|z^{0}-\bar{z}^{0}|+|\mathbb{E}^{0}[y-\bar{y}]|+|\mathbb{E}^{0}[z-\bar{z}]|+|\mathbb{E}^{0}[z^{0}-\bar{z}^{0}]|\Bigr)_t \\
&|\bar{l}_{t}^{i}(y,z_{1},z_{0})-\bar{l}_{t}^{i}(\bar{y},\bar{z},\bar{z}^{0})| \\
&\ \ \ \leq L\Bigl(|y-\bar{y}|+|z-\bar{z}|+|z^{0}-\bar{z}^{0}|+|\mathbb{E}^{0}[y-\bar{y}]|+|\mathbb{E}^{0}[z-\bar{z}]|+|\mathbb{E}^{0}[z^{0}-\bar{z}^{0}]|\Bigr)_t,
\end{align*}
where $C,L>0$ are constants and all processes on the right-hand side take values at time $t$.

\item[\textnormal{(iii)}] With a slight abuse of notation and in the same spirit, let
$
\beta := 192L^{2}(1+T) $ and \\
$
R^{2} := 4\sup_i||\xi_{i}||_{L^{\infty}}^{2} + \tfrac{8}{\beta} \sup_i \Big|\Big|\int_{0}^{T}(\bar{f}_{r}^{i}(0))^{2}dr\Big|\Big|_{L^{\infty}}
$
and assume that it holds
\be
3072C^{2}(1+T)^{2}e^{\beta T}R^{2} \leq 1.
\ee


\end{itemize}

\end{Assumption}

\begin{Theorem}
\label{Thm4.1}

Suppose Assumption \ref{assu:mfbsde} holds. 
Then, mf-qBSDE \textnormal{(\ref{equ:gMF.1})} has a unique solution $(\bar{Y}^{i},\bar{Z}^{i},\bar{Z}^{i0})$ in $\mathcal{B}_{R}$, where
\[
\mathcal{B}_{R} := \Bigl\{(y,z,z^{0}) \in \mathcal{S}^{\infty} \times \mathcal{Z}_{BMO}^{2} : \trinm{(y,z,z^{0})}_{\beta}^{2} \leq e^{\beta T}R^2\Bigr\}.
\]
\end{Theorem}

\begin{proof}

We omit the index $i$ for simplicity of notation. We first define the function $\Psi: (\mathcal{B}_{R}, \trinm{\cdot}_{\beta} ) \rightarrow (\mathcal{B}_{R},\trinm{\cdot}_{\beta})$ with $\Psi(y,z,z^{0})$ being defined as the unique solution to the classical BSDE
\begin{equation}
\label{equ:gMF.2}
\bar{Y}_{t} = \xi + \int_{t}^{T}\bar{f}_{r}(y_{r},z_{r},z_{r}^{0})dr - \int_{t}^{T}(\bar{Z}_{r}dW_{r}+\bar{Z}_{r}^{0}dW_{r}^{0}),
\end{equation}
where $(y,z,z^{0}) \in \mathcal{B}_{R}$. The proof is to show $\Psi $ is a contraction under our assumption, similar to that of Theorem \ref{Thm3.1}, so for simplicity of context, we only 
show the invariance of $\Psi.$ Then we get from applying the It\^o-formula to $e^{\beta t}\bar{Y}_{t}^{2}$, and then taking the conditional expectation $\E_t$, that 
\begin{align}
&e^{\beta t}\bar{Y}_{t}^{2} + \mathbb{E}_{t}\Bigl[\int_{t}^{T}e^{\beta r}(\bar{Z}_{r}^{2}+(\bar{Z}_{r}^{0})^{2})dr\Bigr] + \mathbb{E}_{t}\Bigl[\int_{t}^{T}\beta e^{\beta r}\bar{Y}_{r}^{2}dr\Bigr]\nonumber \\
\leq & \ e^{\beta T}||\xi||_{L^{\infty}}^{2} + \tfrac{1}{2}||\bar{Y}||_{\mathcal{S}^{\infty}}^{2} + 2\Bigl(\mathbb{E}_{t}\Bigl[\int_{t}^{T}e^{\beta r}|\bar{g}_{r}(y_{r},z_{r},z_{r}^{0})-\bar{g}_{r}(0)|dr\Bigr]\Bigr)^{2} + \mathbb{E}_{t}\Bigl[\int_{t}^{T}2e^{\beta r}|\bar{Y}_{r}|\cdot |\bar{f}_{r}(0)|dr\Bigr] \nonumber \\
& + \mathbb{E}_{t}\Bigl[\int_{t}^{T}2e^{\beta r}|\bar{Y}_{r}|\cdot|\bar{l}_{r}(y_{r},z_{r},z_{r}^{0})-\bar{l}_{r}(0)|dr\Bigr] \nonumber \\
\leq & \ \tfrac{1}{4}e^{\beta T}R^{2} + \tfrac{1}{2}||\bar{Y}||_{\mathcal{S}^{\infty}}^{2} + 72C^{2}\Bigl(2T||e^{\tfrac{1}{2}\beta\cdot}y||_{\mathcal{S}^{\infty}}^{2} + ||e^{\tfrac{1}{2}\beta\cdot}(z,z^{0})||_{\mathcal{Z}_{BMO}^{2}}^{2} + ||e^{\tfrac{1}{2}\beta\cdot}(\mathbb{E}^{0}[z],\mathbb{E}^{0}[z^{0}])||_{\mathcal{Z}_{BMO}^{2}}^{2}\Bigr)^{2} \nonumber \\
& +\mathbb{E}_{t}\Bigl[\int_{t}^{T}\beta e^{\beta r}\bar{Y}_{r}^{2}dr\Bigr] + \tfrac{12L^{2}}{\beta}\Bigl(2T||e^{\tfrac{1}{2}\beta\cdot}y||_{\mathcal{S}^{\infty}}^{2} + ||e^{\tfrac{1}{2}\beta\cdot}(z,z^{0})||_{\mathcal{Z}_{BMO}^{2}}^{2} + ||e^{\tfrac{1}{2}\beta\cdot}(\mathbb{E}^{0}[z],\mathbb{E}^{0}[z^{0}])||_{\mathcal{Z}_{BMO}^{2}}^{2}\Bigr). \label{equ:gMF.3}
\end{align}
\noindent
Note that
$||e^{\tfrac{1}{2}\beta\cdot}(\mathbb{E}^{0}[z],\mathbb{E}^{0}[z^{0}])||_{\mathcal{Z}_{BMO}^{2}} \leq ||e^{\tfrac{1}{2}\beta\cdot}(z,z^{0})||_{\mathcal{Z}_{BMO}^{2}}.$
It follows from (\ref{equ:gMF.3}) that
\begin{align}
&e^{\beta t}\bar{Y}_{t}^{2} + \mathbb{E}_{t}\Bigl[\int_{t}^{T}e^{\beta r}(\bar{Z}_{r}^{2}+(\bar{Z}_{r}^{0})^{2})dr\Bigr] \nonumber \\
&\leq \tfrac{1}{4}e^{\beta T}R^{2} + \tfrac{1}{2}||\bar{Y}||_{\mathcal{S}^{\infty}}^{2} + 288C^{2}(1+T)^{2}\trinm{(y,z,z^{0})}_\beta^{4} + \tfrac{24L^{2}(1+T)}{\beta}\trinm{(y,z,z^{0})}_\beta^{2} \nonumber \\
&\leq \tfrac{1}{4}e^{\beta T}R^{2} + \tfrac{1}{2}||e^{\tfrac{1}{2}\beta\cdot}\bar{Y}||_{\mathcal{S}^{\infty}}^{2} + 288C^{2}(1+T)^{2}e^{2\beta T}R^{4} + \tfrac{1}{8}e^{\beta T}R^{2} \label{equ:gMF.4},
\end{align}
where the last inequality follows from the definition of $\beta$ and because $(y,z,z^{0}) \in \mathcal{B}_{R}$. Note that $||\bar{Y}||_{\mathcal{S}^{\infty}} < \infty$, and thus by \eqref{equ:gMF.4} 
\begin{equation}
\label{gMF.5}
||e^{\tfrac{1}{2}\beta\cdot}\bar{Y}||_{\mathcal{S}^{\infty}}^{2} \vee ||e^{\tfrac{1}{2}\beta\cdot}(\bar{Z},\bar{Z}^{0})||_{\mathcal{Z}_{BMO}^{2}}^{2} \leq \tfrac{3}{4}e^{\beta T}R^{2} + 576C^{2}(1+T)^{2}e^{2\beta T}R^{4} \leq e^{\beta T}R^{2},
\end{equation}
which implies the invariance of $\Psi.$

\end{proof}

The following corollary is needed when we show the backward propagation of chaos. Its proof is analogous to that of Corollary \ref{Cor3.3}, and thus omitted.

\begin{Corollary}
\label{Cor4.2}
Let the assumptions from Theorem \ref{Thm4.1} hold and define the Picard iteration to mf-qBSDE \textnormal{(\ref{equ:gMF.1})} by setting $(\bar{Y}^{i,(0)},\bar{Z}^{ii,(0)},\bar{Z}^{i0,(0)}):= (0,0,0)$ and defining $(\bar{Y}^{i,(n+1)},\bar{Z}^{ii,(n+1)},\bar{Z}^{i0,(n+1)})$ as the unique solution to the classical BSDE
\[
\bar{Y}_{t}^{i,(n+1)} = \xi_{i} + \int_{t}^{T}\bar{f}_{r}^{i}(\bar{Y}_{t}^{i,(n)},\bar{Z}_{r}^{ii,(n)},\bar{Z}_{r}^{i0,(n)})dr - \int_{t}^{T}(\bar{Z}_{r}^{ii,(n+1)}dW_{r}^{i}+\bar{Z}_{r}^{i0,(n+1)}dW_{r}^{0}), \ \ n \in \mathbb{N}.
\]
Then, with
$
\tilde{M}_{*}^{2} := 2304C^{2}(1+T)^{2}e^{\beta T}R^{2} + \tfrac{1}{8},
$
we have $0<\tilde{M}_{*}<1$ and for the unique solution $(\bar{Y}^{i},\bar{Z}^{ii},\bar{Z}^{i0})$ from Theorem \ref{Thm4.1} and every $n$, it holds
\begin{equation}
\label{NdimPicard2}
\trinm{(\bar{Y}^{i},\bar{Z}^{ii},\bar{Z}^{i0})-(\bar{Y}^{i,(n)},\bar{Z}^{ii,(n)},\bar{Z}^{i0,(n)})}_\beta \leq \tfrac{R}{1-\tilde{M}_{*}}e^{\tfrac{1}{2}\beta T}\tilde{M}_{*}^{n}.
\end{equation}
\end{Corollary}

Now recall our multi-player model from Section 2. In view of the qBSDE for multi-player games \eqref{equ:qBSDE} driven by $f$ in the form of \eqref{equ:Def_f}, we consider the following mf-qBSDE
\begin{equation}
\label{equ:4.1}
\bar{Y}_{t}^{i} = \xi_{i} + \int_{t}^{T}\bar{f}_{r}^{i}(\bar{Z}_{r}^{ii},\bar{Z}_{r}^{i0},\bar{Z}_{r}^{*,ii},\bar{Z}_{r}^{*,i0})dr - \int_{t}^{T}(\bar{Z}_{r}^{ii}dW_{r}^{i}+\bar{Z}_{r}^{i0}dW_{r}^{0}), \ \ i=1,...,N,
\end{equation}
where we define, motivated by (\ref{equ:DefZtilde}), that
\begin{equation}
\label{equ:DefZ*}
\bar{Z}_{t}^{*,ii} := \mathds{E}^{0}\bigl[\tfrac{\sigma_{t}^{i}}{(\bar{\sigma}_{t}^{i})^{2}}\bar{Z}_{t}^{ii}\bigr] \ \ \text{and} \ \ \bar{Z}_{t}^{*,i0} := \mathds{E}^{0}\bigl[\tfrac{\sigma_{t}^{0}}{(\bar{\sigma}_{t}^{i})^{2}}\bar{Z}_{t}^{i0}\bigr],
\end{equation}
and $\bar{f}_{t}^{i}:\mathds{R}^{4} \rightarrow \mathds{R}$ satisfies
$
\bar{f}_{t}^{i}(z_{1},z_{2},z_{3},z_{4}) = \bar{g}_{t}^{i}(z_{1},z_{2},z_{3},z_{4}) + \bar{l}_{t}^{i}(z_{1},z_{2},z_{3},z_{4}) ,
$
with
\begin{align}\nonumber
\bar{g}_{t}^{i}(z_{1},z_{2},z_{3},z_{4}) &= (\tfrac{\eta}{2}+\bar{\nu}_{t,1}^{i})(z_{1})^{2} + (\tfrac{\eta}{2}+\bar{\nu}_{t,2}^{i})(z_{2})^{2} + \bar{\nu}_{t,3}^{i}z_{1}z_{2}\\
&\ \ \ + \bar{\nu}_{t,4}^{i}z_{1}(z_{3}+z_{4}) + \bar{\nu}_{t,5}^{i}z_{2}(z_{3}+z_{4}) + \bar{\nu}_{t,6}^{i}(z_{3}+z_{4})^{2} \label{equ:Def_gbar}\\
\bar{l}_{t}^{i}(z_{1},z_{2},z_{3},z_{4}) &= \bar{\nu}_{t,7}^{i}z_{1} + \bar{\nu}_{t,8}^{i}z_{2} + \bar{\nu}_{t,9}^{i}(z_{3}+z_{4}) + \bar{\nu}_{t,10}^{i}, \label{equ:Def_lbar}
\end{align}
where $\tilde{c}_{t}^i := \bigl(1-\mathds{E}^{0}\bigl[\tfrac{1}{\eta(\bar{\sigma}_{t}^{i})^{2}}\bigr]\bigr)^{-1}$ and
\begin{align}
& \bar{\nu}_{t,1}^{i} := -\tfrac{\eta (\sigma_{t}^{i})^{2}}{2(\bar{\sigma}_{t}^{i})^{2}}  , \  
\bar{\nu}_{t,2}^{i} := -\tfrac{\eta (\sigma_{t}^{0})^{2}}{2(\bar{\sigma}_{t}^{i})^{2}}  , \  
\bar{\nu}_{t,3}^{i} := -\tfrac{\eta\sigma_{t}^{i}\sigma_{t}^{0}}{(\bar{\sigma}_{t}^{i})^{2}}, \  \bar{\nu}_{t,4}^{i} := -\tfrac{\tilde{c}_{t}^i \sigma_{t}^{i}}{(\bar{\sigma}_{t}^{i})^{2}},  \ 
\bar{\nu}_{t,5}^{i} := -\tfrac{\tilde{c}_{t}^i \sigma_{t}^{0}}{(\bar{\sigma}_{t}^{i})^{2}},  \label{equ:barnu3} \\
&  \bar{\nu}_{t,6}^{i} := -\tfrac{(\tilde{c}_{t}^i)^{2}}{2\eta(\bar{\sigma}_{t}^{i})^{2}}, \ 
 \bar{\nu}_{t,7}^{i} := -\tfrac{\tilde{c}_{t}^i \sigma_{t}^{i}}{\eta(\bar{\sigma}_{t}^{i})^{2}}\mathds{E}^{0}\bigl[\tfrac{\tilde{\mu}_{t}^{i}}{(\bar{\sigma}_{t}^{i})^{2}}\bigr] - \tfrac{\sigma_{t}^{i}\tilde{\mu}_{t}^{i}}{(\bar{\sigma}_{t}^{i})^{2}}, \  
\bar{\nu}_{t,8}^{i} := -\tfrac{\tilde{c}_{t}^i \sigma_{t}^{0}}{\eta(\bar{\sigma}_{t}^{i})^{2}}\mathds{E}^{0}\bigl[\tfrac{\tilde{\mu}_{t}^{i}}{(\bar{\sigma}_{t}^{i})^{2}}\bigr] - \tfrac{\sigma_{t}^{0}\tilde{\mu}_{t}^{i}}{(\bar{\sigma}_{t}^{i})^{2}} \label{equ:barnu6}, \\
&\bar{\nu}_{t,9}^{i} := -\tfrac{(\tilde{c}_{t}^i)^{2}}{\eta^{2}(\bar{\sigma}_{t}^{i})^{2}}\mathds{E}^{0}\bigl[\tfrac{\tilde{\mu}_{t}^{i}}{(\bar{\sigma}_{t}^{i})^{2}}\bigr] - \tfrac{\tilde{c}_{t}^i \tilde{\mu}_{t}^{i}}{\eta(\bar{\sigma}_{t}^{i})^{2}},\  
\bar{\nu}_{t,10}^{i} := -\tfrac{(\tilde{c}_{t}^i)^{2}}{2\eta^{3}(\bar{\sigma}_{t}^{i})^{2}}\Bigl(\mathds{E}^{0}\bigl[\tfrac{\tilde{\mu}_{t}^{i}}{(\bar{\sigma}_{t}^{i})^{2}}\bigr]\Bigr)^{2} - \tfrac{\tilde{c}_{t}^i \tilde{\mu}_{t}^{i}}{\eta^{2}(\bar{\sigma}_{t}^{i})^{2}}\mathds{E}^{0}\bigl[\tfrac{\tilde{\mu}_{t}^{i}}{(\bar{\sigma}_{t}^{i})^{2}}\bigr] - \tfrac{(\tilde{\mu}_{t}^{i})^{2}}{2\eta(\bar{\sigma}_{t}^{i})^{2}}\label{equ:barnu10}.
\end{align}
To simplify the model, as suggested by Section \ref{sec:3.1.1} , we assume the following.
\begin{Assumption}\label{assu:mfmodel}
    Suppose Assumption \ref{assu:pi-model} holds, only with \textnormal{(A3)} replaced by the following: 
$$
\min_{i\in \{1,...,N\} } \Bigl|\tfrac{1}{\eta}\mathbb{E}^{0}\bigl[(\bar{\sigma}_{t}^{i})^{-2}\bigr]-1\Bigr| \geq \delta_{\sigma}, \ \ \ \forall t \in [0,T].
$$
\end{Assumption}

Recall also the bounds $b,b_{\mu}, \bR^2$ defined in \eqref{equ:Defbounds}, and it is easy to check that $b$ is an upper bound of $\{|\bar{\nu}^i_{t,j}|\}_{j=1}^6$, and $b_\mu$ is an upper bound of $\{|\bar{\nu}^i_{t,j}|\}_{j=7}^{10}$. Moreover, with a bit of abuse of notation, define
\begin{align*}
\bL := \max\{1,b_{\mu}\}\max\Bigl\{1  ,  \tfrac{\bar{\delta}_{\sigma}^{2}}{\delta_{\sigma}^{4}}\Bigr\}, \ 
\bar \beta := 192 \bar{L}^{2}(1+T), \ 
\bC := (b+\tfrac{\eta}{2})\max\Bigl\{1 ,  \tfrac{\bar{\delta}_{\sigma}^{2}}{\delta_{\sigma}^{4}}\Bigr\}. 
\end{align*}
\noindent
Now we want to prove that mf-qBSDE (\ref{equ:4.1}) satisfies the form from Theorem \ref{Thm4.1} and thus is well-posed.

\begin{Corollary}
\label{Cor4.5}
Suppose that Assumption \ref{assu:mfmodel} holds and, moreover,
\be\label{assu:modelbd1}
3072 \bC^{2}(1+T)^{2}e^{\bar{\beta} T} \bR^{2} \leq 1.
\ee
Then, for each $i \in \{1,...,N\}$, mf-qBSDE \textnormal{(\ref{equ:4.1})} has a unique solution $(\bar{Y}^{i},\bar{Z}^{ii},\bar{Z}^{i0})$ in $\mathcal{B}_{\bR}$, where $\mathcal{B}_{\bR}$ is defined as in Theorem \ref{Thm4.1}.
\end{Corollary}
 
\begin{proof} 

Note that for every $(z_{1},...,z_{8}) \in \mathbb{R}^{8}$, 
\begin{align*}
&\big|\bar{g}_{t}^{i}(z_{1},z_{2},z_{3},z_{4})-\bar{g}_{t}^{i}(z_{5},z_{6},z_{7},z_{8})\big| \leq \Bigl(b+\frac{\eta}{2}\Bigr)\Bigl(\sum_{j=1}^{8}|z_{j}|\Bigr)\big(|\Delta z_{1}|+|\Delta z_{2}|+|\Delta z_{3}|+|\Delta z_{4}|\big),\\
&\big|\bar{l}_{t}^{i}(z_{1},z_{2},z_{3},z_{4})-\bar{l}_{t}^{i}(z_{5},z_{6},z_{7},z_{8})\big| \leq b_{\mu}\Bigl(|\Delta z_{1}|+|\Delta z_{2}|+|\Delta z_{3}|+|\Delta z_{4}| \Bigr),
\end{align*}
where we define
$
\Delta z_{1}:=z_{1}-z_{5} \ , \ \Delta z_{2}:=z_{2}-z_{6} \ , \ \Delta z_{3}:=z_{3}-z_{7} \ , \ \Delta z_{4}:=z_{4}-z_{8}.
$
Then our result follows from Theorem \ref{Thm4.1}. 

\end{proof}

\begin{Remark}
Similarly to Remark \ref{Rem3.5}, we can always satisfy the assumptions from Corollary \ref{Cor4.5} if we choose $||\xi_{i}||_{L^{\infty}}$ and $\delta_{\mu}$, or $\eta$ small enough.
\end{Remark}

\subsubsection[Nash Equilibria of the N-player game]{Nash Equilibria of the $N$-player game}
\label{sec4.1.2}

Returning to the \(N\)-player game introduced in Section \ref{sec2.1} and the corresponding optimality approach developed in Section \ref{sec2.2}, we show that the controls represented through the solutions of the associated BSDEs indeed constitute equilibria. Indeed, we only need to verify that the sufficient conditions \eqref{(i)}--\eqref{(iii)} from Section \ref{sec2.2} are satisfied. 
First, we need a technical Lemma to ensure $L^{p}$-boundedness of the utility function.

\begin{Lemma}
\label{Lemma6.2}
Let Assumption \ref{assu:pi-model} hold. 
For any $\Balpha=(\alpha^i)_{i=1}^N \in \mathcal{A}^{N}$, $\mu \in \mathcal{A}$, define
\[
X_{t}^{i} := x + \int_{0}^{t}\alpha_{r}^{i}(\mu_{r}+\tilde{\mu}_{r}^{i})dr + \int_{0}^{t}\alpha_{r}^{i}(\sigma_{r}^{i}dW_{r}^{i}+\sigma_{r}^{0}dW_{r}^{0}),\  i=1,...,N,
\] 
where $x$ is a constant. Moreover, for the constant $M>0$ from the definition \eqref{equ:2.12} of $\mathcal{A}$, there exists $q^{*}>0$, such that 
$
M  < \min\Big\{(2q^{*}\eta)^{-\tfrac{1}{2}} - \delta_{\mu}\sqrt{T} , \ (4\sqrt{2}q^{*}\eta\bar{\delta}_{\sigma})^{-1}\Big\}.
$
Then for every $q \in [-q^{*},q^{*}]$, it holds
\[
\mathbb{E}[\exp(-q\eta(X_{T}^{i}-\xi_{i}))] \leq \tilde{M}_{r,q},
\]
where we recall $r$ from Assumption \ref{assu:pi-model} $\mathrm{(A1)}$, and 
$$
\tilde{M}_{r,q} := e^{|q|\eta(r+|x|)}\big(1-2|q|\eta M^{2}\big)^{-\tfrac{1}{4}}\big(1-2|q|\eta (M+\delta_{\mu}\sqrt{T})^{2}\big)^{-\tfrac{1}{4}}\big(1-32q^{2}\eta^{2}(\bar{\delta}_{\sigma})^{2}M^{2}\big)^{-\tfrac{1}{4}} >0.
$$
\end{Lemma}

\begin{proof}
 
Since $||\xi_{i}||_{L^{\infty}} \leq r$, we get
\be\label{equ:6.3}
\mathbb{E}[\exp(-q\eta(X_{T}^{i}-\xi_{i}))] = \mathbb{E}[e^{q\eta\xi_{i}}\exp(-{q\eta}X_{T}^{i})]   \leq e^{|q|\eta r}\mathbb{E}[\exp(-{q\eta}X_{T}^{i})].
\ee
Furthermore, we have
\begin{align}\nonumber
&\mathbb{E}[\exp(-{q\eta}X_{T}^{i})]\\
&\ \ \ = \mathbb{E}\Bigl[\exp\Bigl(-q\eta\Bigl(x + \int_{0}^{T}\alpha_{r}^{i}(\mu_{r}+\tilde{\mu}_{r}^{i})dr + \int_{0}^{T} \alpha_{r}^{i}(\sigma_{r}^{i}dW_{r}^{i}+\sigma_{r}^{0}dW_{r}^{0})\Bigr)\Bigr)\Bigr] \nonumber \\
&\ \ \ \leq e^{|qx|{\eta}}\biggl(\mathbb{E}\Bigl[\exp\Bigl(-2q\eta\int_{0}^{T}\alpha_{r}^{i}(\mu_{r}+\tilde{\mu}_{r}^{i})dr\Bigr)\Bigr]\biggr)^{\tfrac{1}{2}}\biggl(\mathbb{E}\Bigl[\exp\Bigl(-2q\eta\int_{0}^{T}\alpha_{r}^{i}(\sigma_{r}^{i}dW_{r}^{i}+\sigma_{r}^{0}dW_{r}^{0})\Bigr)\Bigr]\biggr)^{\tfrac{1}{2}}. \label{equ:6.4}
\end{align}
For the first expectation on the right hand side of the above inequality, note that by our assumption on $M,$ $\big|\big|\sqrt{2|q|\eta}(\alpha^{i})\big|\big|_{\mathcal{Z}_{BMO}^{2}} < 1$ and $\big|\big|\sqrt{2|q|\eta}(\mu+\tilde{\mu}^{i})\big|\big|_{\mathcal{Z}_{BMO}^{2}} < 1$. Thus, by the Cauchy-Schwarz inequality and Lemma \ref{Lemma2.3}, we have
\begin{align*}
\mathbb{E}\Bigl[\exp\Bigl(-2q\eta\int_{0}^{T}\alpha_{r}^{i}(\mu_{r}+\tilde{\mu}_{r}^{i})dr\Bigr)\Bigr] &\leq \mathbb{E}\Bigl[\exp\Bigl(|q|\eta\int_{0}^{T}(\alpha_{r}^{i})^{2}dr + |q|\eta\int_{0}^{T}(\mu_{r}+\tilde{\mu}_{r}^{i})^{2}dr\Bigr)\Bigr]\\
&\leq \biggl(\mathbb{E}\Bigl[\exp\Bigl(2|q|\eta\int_{0}^{T}(\alpha_{r}^{i})^{2}dr\Bigr)\Bigr]\biggr)^{\tfrac{1}{2}}\biggl(\mathbb{E}\Bigl[2|q|\eta\int_{0}^{T}(\mu_{r}+\tilde{\mu}_{r}^{i})^{2}dr\Bigr]\biggr)^{\tfrac{1}{2}}\\
&\leq \Bigl(1-\big|\big|\sqrt{2|q|\eta}(\alpha^{i})\big|\big|_{\mathcal{Z}_{BMO}^{2}}^{2}\Bigr)^{-\tfrac{1}{2}}\Bigl(1-\big|\big|\sqrt{2|q|\eta}(\mu+\tilde{\mu}^{i})\big|\big|_{\mathcal{Z}_{BMO}^{2}}^{2}\Bigr)^{-\tfrac{1}{2}}.
\end{align*}
For the second expectation in (\ref{equ:6.4}), we get, by Corollary \ref{Cor2.4}, that
\begin{align*}
&\mathbb{E}\Bigl[\exp\Bigl(-2q\eta\int_{0}^{T}\alpha_{r}^{i}(\sigma_{r}^{i}dW_{r}^{i}+\sigma_{r}^{0}dW_{r}^{0})\Bigr)\Bigr] \\
&\leq \biggl(\mathbb{E}\Bigl[\exp\Bigl(-4q\eta\int_{0}^{T}\alpha_{r}^{i}\sigma_{r}^{i}dW_{r}^{i}\Bigr)\Bigr]\biggr)^{\tfrac{1}{2}}\biggl(\mathbb{E}\Bigl[\exp\Bigl(-4q\eta\int_{0}^{T}\alpha_{r}^{i}\sigma_{r}^{0}dW_{r}^{0}\Bigr)\Bigr]\biggr)^{\tfrac{1}{2}} \\
&\leq \big(1-2\big|\big|4|q|\eta\alpha^{i}\sigma^{i}\big|\big|_{\mathcal{Z}_{BMO}^{2}}^{2}\big)^{-\tfrac{1}{4}}\big(1-2\big|\big|4|q|\eta\alpha^{i}\sigma^{0}\big|\big|_{\mathcal{Z}_{BMO}^{2}}^{2}\big)^{-\tfrac{1}{4}}.
\end{align*}
Taking everything together, 
we get from (\ref{equ:6.3}) and (\ref{equ:6.4})
\begin{align*}
&e^{|q|\eta r}\mathbb{E}[\exp(-{q\eta}X_{T}^{i})] \\
&\leq e^{|q|\eta(r+|x|)}\big(1-\big|\big|\sqrt{2|q|\eta}(\alpha^{i})\big|\big|_{\mathcal{Z}_{BMO}^{2}}^{2}\big)^{-\tfrac{1}{4}}\big(1-\big|\big|\sqrt{2|q|\eta}(\mu+\tilde{\mu}^{i})\big|\big|_{\mathcal{Z}_{BMO}^{2}}^{2}\big)^{-\tfrac{1}{4}}\\
&\ \ \ \cdot \big(1-2\big|\big|4|q|\eta\alpha^{i}\sigma^{i}\big|\big|_{\mathcal{Z}_{BMO}^{2}}^{2}\big)^{-\tfrac{1}{8}}\big(1-2\big|\big|4|q|\eta\alpha^{i}\sigma^{0}\big|\big|_{\mathcal{Z}_{BMO}^{2}}^{2}\big)^{-\tfrac{1}{8}} \\
&\leq e^{|q|\eta(r+|x|)}\big(1-2|q|\eta M^{2}\big)^{-\tfrac{1}{4}}\big(1-2|q|\eta (M+\delta_{\mu}\sqrt{T})^{2}\big)^{-\tfrac{1}{4}}\big(1-32q^{2}\eta^{2}(\bar{\delta}_{\sigma})^{2}M^{2}\big)^{-\tfrac{1}{4}}.
\end{align*}

\end{proof}

\begin{Theorem}
\label{Thm4.8}
Let $\mu^{N,\Balpha}$ be defined as in Remark \ref{rem:NvsN-1} and assume that Assumptions \ref{assu:pi-model} $\mathrm{(A1),(A2)}$ hold. Moreover, assume that the driver from \eqref{eq:fop}, with $\mu^{N,\hBalpha}$ replaced by $\mu^{N,\Balpha}$, satisfies Assumption \ref{assu:well-posedBSDE}. Then there exists a $\delta>0$ such that
$\lps{*}\Balpha$ obtained in \eqref{eq:trueeq} is a Nash equilibrium of the $N$-player game whenever
\begin{equation}
\label{equ:trueeqcond}
C^{2}(1+2C^{2}+\tfrac{C^{2}}{T^{2}})e^{\beta T}R^{2} < \delta,
\end{equation}
where the constants $C,\beta,R$ are obtained from Assumption \ref{assu:well-posedBSDE}. Moreover, 
suppose assumptions \ref{assu:pi-model} and \eqref{assu:modelbd} hold, and additionally, there exists a $p>1$ such that 
\begin{equation}
\label{equ:condM}
0<M < \min \Big\{(2p\eta)^{-\tfrac{1}{2}} - \delta_{\mu}\sqrt{T}  , \ (4\sqrt{2}p\eta\bar{\delta}_{\sigma})^{-1} \Big\}.
\end{equation}
Then
$\hBalpha$ given by \eqref{eq:fop} is an $\vep_N$-Nash equilibrium, with $|\vep_N| \leq C_{p}N^{-1}$ for some constant $C_{p}>0$.
\end{Theorem}


\begin{proof}
Note that Assumption \ref{assu:well-posedBSDE} ensures that the respective qBSDEs are well-posed in the sense of Theorem \ref{Thm3.1}.
Since $\delta$ controls how large the $\mathcal{Z}_{BMO}^{2}$-norm of the solutions can get, the specific controls that are represented by solutions to the respective qBSDEs will always be admissible for a small enough $\delta$, even if the additional condition (\ref{equ:condM}) is imposed. For the first part of the statement, it is sufficient to prove that conditions (\ref{(i)}) $-$ (\ref{(iii)}) from Section \ref{sec2.2} hold, since this proves the necessary inequality (\ref{equ:martopt}). Conditions (\ref{(i)}) and (\ref{(ii)}) are satisfied in view of the construction in Section \ref{sec2.2}.
For condition (\ref{(iii)}) note that $M^{\Balpha,i}$ defined in (\ref{equ:DefM}) is the stochastic exponential of a BMO-martingale (because the $\alpha^{i}$ and $Z^{ij}$ are in $\mathcal{Z}_{BMO}^{2}$ and $\sigma^{0}$ and $\sigma^{i}$ are bounded), which implies, by \cite[Theorem 2.3]{Ref3}, that $M^{\Balpha,i}$ itself is a true martingale. Furthermore, $C^{\Balpha,i}$ defined in (\ref{equ:DefC}) satisfies $C^{\lps{*}\Balpha,i} = -1$, which implies that $J^{\lps{*}\Balpha,i}$ is a true martingale. Finally, because $C^{\Balpha,i}$ is non-positive and non-increasing, it either holds $\mathbb{E}[J_{T}^{(\lps{*}\Balpha^{-i},\alpha),i}] = -\infty$, which directly implies the desired inequality (\ref{equ:martopt}), or $\mathbb{E}[J_{T}^{(\lps{*}\Balpha^{-i},\alpha),i}] > -\infty$, which implies that $J^{(\lps{*}\Balpha^{-i},\alpha),i}$ is a supermartingale and therefore, that condition (\ref{(iii)}) is satisfied. Hence, inequality (\ref{equ:martopt}) is satisfied, proving that $\lps{*}\Balpha$ is a Nash equilibrium. 


For the second part, note that $\mu^{N,\Balpha}$ defined in (\ref{equ:DefMu}) depends on the $i$-th component of $\Balpha$, which implies that condition (\ref{(ii)}) is not satisfied. Let $i \in \{1,...,N\}$ and $\alpha \in \mathcal{A}$. Then we have
\[
X_{T}^{x,(\hBalpha^{-i},\alpha),i} = x + \int_{0}^{T}\alpha_{r}\Bigl(\Bigl(\tfrac{1}{N}\sum_{j=1}^{N}\hat{\alpha}_{r}^{j} + \tilde{\mu}_{r}^{i}\Bigr)dr + \sigma_{r}^{i}dW_{r}^{i} + \sigma_{r}^{0}dW_{r}^{0}\Bigr) + \int_{0}^{T}\alpha_{r}\tfrac{\alpha_{r}-\hat{\alpha}_{r}^{i}}{N}dr =: X_{T} + \rho_{N},
\]
with
$
X_{T} := x + \int_{0}^{T}\alpha_{r}\Bigl(\Bigl(\tfrac{1}{N}\sum_{j=1}^{N}\hat{\alpha}_{r}^{j} + \tilde{\mu}_{r}^{i}\Bigr)dr + \sigma_{r}^{i}dW_{r}^{i} + \sigma_{r}^{0}dW_{r}^{0}\Bigr)
$
and $\rho_{N} := \int_{0}^{T}\alpha_{r}\tfrac{\alpha_{r}-\hat{\alpha}_{r}^{i}}{N}dr$. If we plug this into our cost function, we get
\begin{align}
I_{i}^{N}(\hBalpha^{-i},\alpha) &= \mathbb{E}[-\exp(-\eta(X_{T}^{x,(\hat{\alpha}^{-i},\alpha),i}-\xi_{i}))] \nonumber \\
&= \mathbb{E}[-\exp(-\eta(X_{T} + \rho_{N}-\xi_{i}))] \nonumber \\ 
&= \mathbb{E}[-e^{-\eta\rho_{N}} \cdot \exp(-\eta(X_{T}-\xi_{i}))]. \label{equ:6.6.5}
\end{align}
Now, $\hat{\alpha}^{i}$ maximizes
\begin{equation}
\label{equ:6.max}
I(\alpha) := \mathbb{E}\Bigl[-\exp\Bigl(-\eta\Bigl(x + \int_{0}^{T}\alpha_{r}\bigl((\mu_{r}^{N,\hBalpha}+\tilde{\mu}_{r}^{i})dr + \sigma_{r}^{i}dW_{r}^{i} + \sigma_{r}^{0}dW_{r}^{0}\bigr) - \xi_i\Bigr)\Bigr)\Bigr] \ , \ \alpha \in \mathcal{A},
\end{equation}
because we fix the drift to be $\mu^{N,\hBalpha}$, which does not depend on $\alpha$ and therefore, by analogous arguments as in the first part of the proof, conditions (\ref{(i)}) $-$ (\ref{(iii)}) are satisfied. This implies
\[
I(\alpha) = \mathbb{E}[-\exp(-\eta(X_{T}-\xi_{i}))] \leq \mathbb{E}[-\exp(-\eta(X_{T}^{x,\hBalpha,i}-\xi_{i}))] = I(\hat{\alpha}^{i}) = I_{i}^{N}(\hBalpha).
\]
Now, we use this, together with (\ref{equ:6.6.5}), to get
\begin{align}
I_{i}^{N}(\hBalpha) - I_{i}^{N}(\hBalpha^{-i},\alpha) &\geq \mathbb{E}[-\exp(-\eta(X_{T}-\xi_{i}))] - \mathbb{E}[-e^{-\eta\rho_{N}} \cdot \exp(-\eta(X_{T}-\xi_{i}))] \nonumber \\
&= \mathbb{E}[(1-e^{-\eta\rho_{N}}) \cdot (-\exp(-\eta(X_{T}-\xi_{i})))] \nonumber \\
&\geq -\mathbb{E}[\eta|\rho_{N}|\exp(-\eta(X_{T}-\xi_{i}))] \label{equ:6.7},
\end{align}
where we used the inequality $e^{-x}-1 \geq -|x|$. Since $p>1$, we can estimate with the Hölder-inequality
\begin{align}
\mathbb{E}[\eta|\rho_{N}|\exp(-\eta(X_{T}-\xi_{i}))] &\leq \eta\Bigl(\mathbb{E}[\exp(-p\eta(X_{T}-\xi_{i}))]\Bigr)^{\frac1p}\Bigl(\mathbb{E}\bigl[|\rho_{N}|^{\tfrac{p}{p-1}}\bigr]\Bigr)^{\tfrac{p-1}{p}} \nonumber \\
&\leq \eta\tilde{M}_{r,p}^{\frac1p}\Biggl(\mathbb{E}\biggl[\Bigl(\int_{0}^{T}(\alpha_{r}^{2}-\alpha_{r}\hat{\alpha}_{r}^{i})dr\Bigr)^{\tfrac{p}{p-1}}\biggr]\Biggr)^{\tfrac{p-1}{p}}N^{-1} \label{equ:4.22} \\
&\leq C_{p}N^{-1}, \label{equ:4.23}
\end{align}
for some constant $C_{p}>0$ independent of $N$, where we used Lemma \ref{Lemma6.2} to estimate (\ref{equ:4.22}) and where the constant $C_{p}<\infty$ exists in view of the energy inequalities (\ref{equ:BMO.1}) because $(\alpha_{r}^{2}-\alpha_{r}\hat{\alpha}_{r}^{i})$ have bounded $\mathcal{Z}_{BMO}^{2}$-norm. Plugging this into (\ref{equ:6.7}) yields
\begin{equation}
I_{i}^{N}(\hBalpha) - I_{i}^{N}(\hBalpha^{-i},\alpha) \geq -C_{p}N^{-1},
\end{equation}
which implies the desired statement by choosing $\varepsilon_{N}:=C_{p}N^{-1}$.


\end{proof}

   

\begin{Remark}
\label{rem:vepNE}
For the first result above, it is, in fact, not difficult to derive the exact condition, in the spirit of \eqref{equ:trueeqcond}, under which $\lps{*}\Balpha$ constitutes a true Nash equilibrium. We formulate the statement as above in order to keep the presentation concise and avoid lengthy computations. 
On the other hand, in view of the boundedness assumption \ref{assu:pi-model} \textnormal{(A2)}, the difference between the corresponding coefficients of the two drivers --- one associated with $\lps{*}\Balpha$ and the other with $\hBalpha$ --- is of order \(N^{-1}\). Combined with the stability result from Theorem \ref{Thm3.9}, this implies that the difference between $\hBalpha$ and $\lps{*}\Balpha$ converges to zero as \(N \to \infty\). This directly corresponds to the conclusion of Theorem \ref{Thm4.8}, namely that $\hBalpha$ is an \(\varepsilon_N\)-Nash equilibrium of the \(N\)-player game.
\end{Remark}

\subsubsection[Convergence of N-player Nash equilibrium to the mean-field equilibrium]{Convergence of $N$-player Nash equilibrium to the mean-field equilibrium}
\label{sec4.1.3}

In view of Remark \ref{rem:vepNE}, we only need to study the limit of $\hBalpha$ in the multi-player game.
As expected, as $N$ goes to infinity, the equilibrium converges to 
\begin{equation}
\label{equ:DefAlphaBar}
\bar{\alpha}_{t}^{i} := \tfrac{\sigma_{t}^{i}}{(\bar{\sigma}_{t}^{i})^{2}}\bar{Z}_{t}^{ii} + \tfrac{\sigma_{t}^{0}}{(\bar{\sigma}_{t}^{i})^{2}}\bar{Z}_{t}^{i0} + \tfrac{\mathds{E}^{0}[\bar{\alpha}_{t}^{i}]+\tilde{\mu}_{t}^{i}}{\eta(\bar{\sigma}_{t}^{i})^{2}} \ , \ i \in \{1,...,N\},
\end{equation}
where we recall that $\bar Z$ is the solution to \eqref{equ:4.1}. Note that by taking the conditional expectation $\E^0$ on both sides of the above identity, we have
\[
\mathds{E}^{0}[\bar{\alpha}_{t}^{i}] = \mathds{E}^{0}[\tfrac{\sigma_{t}^{i}}{(\bar{\sigma}_{t}^{i})^{2}}\bar{Z}_{t}^{ii}] + \mathds{E}^{0}[\tfrac{\sigma_{t}^{0}}{(\bar{\sigma}_{t}^{i})^{2}}\bar{Z}_{t}^{i0}] + \mathds{E}^{0}[\bar{\alpha}_{t}^{i}]\mathds{E}^{0}[\tfrac{1}{\eta(\bar{\sigma}_{t}^{i})^{2}}] + \mathds{E}^{0}[\tfrac{\tilde{\mu}_{t}^{i}}{\eta(\bar{\sigma}_{t}^{i})^{2}}],
\]
which is equivalent to
\begin{equation}
\label{equ:2.41}
\mathds{E}^{0}[\bar{\alpha}_{t}^{i}] = \tilde{c}_{t}\mathds{E}^{0}\bigl[\tfrac{\sigma_{t}^{i}}{(\bar{\sigma}_{t}^{i})^{2}}\bar{Z}_{t}^{ii}\bigr] + \tilde{c}_{t}\mathds{E}^{0}\bigl[\tfrac{\sigma_{t}^{0}}{(\bar{\sigma}_{t}^{i})^{2}}\bar{Z}_{t}^{i0}\bigr] + \tilde{c}_{t}\mathds{E}^{0}\bigl[\tfrac{\tilde{\mu}_{t}^{i}}{\eta(\bar{\sigma}_{t}^{i})^{2}}\bigr].
\end{equation}
It follows by substituting the above representation of $\mathds{E}^{0}[\bar{\alpha}_{t}^{i}]$ into \eqref{equ:DefAlphaBar} that 
\be\label{eq:fulalbar}
\bar{\alpha}_{t}^{i} := \tfrac{\sigma_{t}^{i}}{(\bar{\sigma}_{t}^{i})^{2}}\bar{Z}_{t}^{ii} + \tfrac{\sigma_{t}^{0}}{(\bar{\sigma}_{t}^{i})^{2}}\bar{Z}_{t}^{i0}+\tfrac{1}{\eta (\bar{\sigma}_{t}^{i})^{2}} \Bigl(\tilde{c}_t \E^{0}\Bigl[\tfrac{\sigma_{t}^{i}}{(\bar{\sigma}_{t}^{i})^{2}}\bar{Z}_{t}^{ii} + \tfrac{\sigma_{t}^{0}}{(\bar{\sigma}_{t}^{i})^{2}}\bar{Z}_{t}^{i0} + \tfrac{\tilde{\mu}_{t}^{i}}{\eta(\bar{\sigma}_{t}^{i})^{2} } \Bigr]  + \tilde{\mu}^i_t \Bigr).
\ee

In view of the above representation, we will show the convergence of equilibrium by proving the convergence of the solution to qBSDE \eqref{equ:qBSDE} driven by \eqref{eq:fop} to that of mf-qBSDE \eqref{equ:4.1} by applying our stability result Theorem \ref{Thm3.9}.
For convenience of notation, we define $\bar{Z}^{ij} := 0$ for $i \in \{1,...,N\}, j \in \{0,...,N\}, j \notin \{0,i\}$, and thus we get the $\mathbb{R}^{N \times (N+1)}$-valued process $\bar{Z}_{t}$.

\begin{Remark}
\label{Rem5.1}

Although the unique solutions to BSDEs \((\ref{equ:qBSDE})\) and \((\ref{equ:4.1})\), as established in Corollaries \ref{Cor3.4} and \ref{Cor4.5}, may \emph{a priori} lie in different balls \(\mathcal{B}_{\bar{R}}\), this issue can be resolved by imposing a stronger set of assumptions. More precisely, without loss of generality, we may assume that both solutions belong to the same ball \(\mathcal{B}_{\bar{R}}\). Indeed, this can be achieved by selecting the constants \(\bar{C}\), \(\bar{L}\), and \(\bar{\beta}\) appearing in Corollaries \ref{Cor3.4} and \ref{Cor4.5} to be the maximum of the corresponding constants in the two results.

\end{Remark}

\begin{Definition}
    We call a sequence of progressively measurable processes $(X^1,...,X^k)$ with values in $(\R^d)^{\otimes k}$ conditionally independent given $\MF_T^0$, if for any bounded Borel measurable functions $(\psi_i)_{i=1}^k$, we have for a.e. $t\in [0,T],$
\be\label{eq:cind}
\E[ \Pi_{i=1}^k \psi_i(X^i_t) |\MF_T^0] = \Pi_{i=1}^k \E[\psi_i(X^i_t) |\MF_T^0 ],\ \ \ a.s.
\ee
In a similar spirit, we can define the terminology of conditionally identically distributed processes. We call the sequence of processes conditionally i.i.d. if they are both conditionally independent and conditionally identically distributed. 
    
\end{Definition}

To prove the convergence, we need the following lemma on the average of conditionally i.i.d. processes.

\begin{Lemma}\label{Lemma5.2}
Assume that $  Z^{1},...,Z^{N} $ are predictable, 
%
conditionally independent given $\mathcal{F}_{T}^{0}$, and that 
$
\sup_{i \leq N} ||Z^{i}||_{\MH^{2}} \leq C,
$ for some constant $C>0$. Define
$
(LLN)_{t} := \tfrac{1}{N}\sum_{j=1}^{N}(Z_{t}^{j}-\mathbb{E}^{0}[Z_{t}^{j}]).
$
Then it holds
\[
||(LLN)||_{\mathcal{H}^{2}}^{2} \leq \tfrac{4C^{2}}{N}.
\]
\end{Lemma}

\begin{proof}
    
Note that by conditional independence, we have for any  $i \neq j$,
$$
\mathbb{E}\bigl[(Z_{t}^{j}-\mathbb{E}^{0}[Z_{t}^{j}])\big|\mathcal{F}_{T}^{0},Z_{t}^{i}\bigr]= \mathbb{E}\bigl[(Z_{t}^{j}-\mathbb{E}^{0}[Z_{t}^{j}])\big|\mathcal{F}_{T}^{0}\bigr].
$$
It follows from the fact that $\mathbb{E}\bigl[(Z_{t}^{j}-\mathbb{E}^{0}[Z_{t}^{j}])\big|\mathcal{F}_{T}^{0}\bigr] = 0$,
\begin{align}
\mathbb{E}^{0}\Bigl[(Z_{t}^{i}-\mathbb{E}^{0}[Z_{t}^{i}])(Z_{t}^{j}-\mathbb{E}^{0}[Z_{t}^{j}])\Bigr] &= \mathbb{E}\Bigl[\mathbb{E}\bigl[(Z_{t}^{i}-\mathbb{E}^{0}[Z_{t}^{i}])(Z_{t}^{j}-\mathbb{E}^{0}[Z_{t}^{j}])\big|\mathcal{F}_{T}^{0},Z_{t}^{i}\bigr]\Big|\mathcal{F}_{T}^{0}\Bigr] \nonumber \\
&= \mathbb{E}\Bigl[(Z_{t}^{i}-\mathbb{E}^{0}[Z_{t}^{i}])\mathbb{E}\bigl[(Z_{t}^{j}-\mathbb{E}^{0}[Z_{t}^{j}])\big|\mathcal{F}_{T}^{0}\bigr]\Big|\mathcal{F}_{T}^{0}\Bigr] 
= 0. \label{equ:crossvanish}
\end{align}
Therefore, we can compute
\begin{align*}
||(LLN)||_{\mathcal{H}^{2}}^{2} 
&= \int_{0}^{T}\mathbb{E}\Bigl[\tfrac{1}{N^{2}}\sum_{j=1}^{N}(Z_{r}^{j}-\mathbb{E}^{0}[Z_{r}^{j}])^{2}\Bigr]dr \\
&\leq \tfrac{1}{N^{2}}\sum_{j=1}^{N}\mathbb{E}\Bigl[\int_{0}^{T}\bigl(2(Z_{r}^{j})^{2}+2(\mathbb{E}^{0}[Z_{r}^{j}])^{2}dr\Bigr] 
\leq \tfrac{4C^{2}}{N}.
\end{align*}

\end{proof}

Recall the notation $(\MF^i_t)_t$ of the augmented filtration generated by $W^i$ from Section \ref{sec1.1}. Let $(\MF^{i,0}_t)_t$ be the filtration augmented from $(\MF^i_t \vee \MF^0_t)_t$. Set the constants $\bR$, $\bC$, $\bL$, and $\bar\beta$ as defined in Remark \ref{Rem5.1}. To obtain the convergence, we assume the following.

\begin{Assumption}\label{assu:convergmf}

Suppose assumptions \ref{assu:pi-model} and \ref{assu:mfmodel} hold. Moreover, assume that for any $i=1,...,N$, 
(a). $\tilde \mu^i$ and $\sigma^i$ ($\xi_i $ resp.) are predictable (resp. measurable) with respect to $\MF^{i,0}$ ($\MF^{i,0}_T$ resp.); $\sigma^0$ is predictable with respect to $\MF^{0}$; \\
(b). $(\xi_i, \tilde \mu^i, \sigma^i, \sigma^0, W^i, W^0)$, $i=1,...,N$, are conditionally identically distributed given $\mathcal{F}_{T}^{0}$, and thus conditionally i.i.d. in view of (a).

\end{Assumption}

\begin{Lemma}\label{lem:ciid}
 Suppose that Assumption \ref{assu:convergmf} and the condition $
3072 \bC^{2}(1+T)^{2}e^{\bar{\beta} T} \bR^{2} \leq 1
$ from  Corollary \ref{Cor4.5} 
hold. Denote by $(\bar{Y}^{i,(n)},\bar{Z}^{i,(n)})$ the solutions of the respective Picard scheme of mf-qBSDE \eqref{equ:4.1} from Corollary \ref{Cor4.2}. Then we have for every $n \in \mathbb{N}$ that $(\tilde \mu^i, \sigma^i, \sigma^0, {W}^i,W^0, \bar{Y}^{i,(n)},\bar{Z}^{i,(n)})$, $i=1,...,N$, are conditionally i.i.d. As a consequence, for the coefficients $(\bar{Y}^{i},\bar{Z}^{i})$ of the solution of \eqref{equ:4.1} holds that $(\tilde \mu^i, \sigma^i, \sigma^0, {W}^i,W^0, \bar{Y}^{i},\bar{Z}^{i})$, $i=1,...,N$, are conditionally i.i.d.
\end{Lemma}

\begin{proof}

Since $(\bar{Y}^{i,(n)}, \bar{Z}^{ii,(n)}, \bar{Z}^{i0,(n)})$ is the solution of the respective $n$-th Picard iteration of mf-qBSDE (\ref{equ:4.1}) that we get from Corollary \ref{Cor4.2}, it is predictable with respect to $\MF^{i,0}$. According to \cite[Proposition 10]{CTT16}, we see that there exist predictable functions\footnote{Here we equip the canonical space $C([0,T], \R^{2})$ with the canonical predictable filtration} $(\phi_y^{i,(n)}, \phi_z^{i,(n)}  ): [0,T] \times C([0,T], \R^{2}) \rightarrow \R \times \R^2$, such that 
\be
\begin{split}
&\bar{Y}_{t}^{i,(n)}(\omega)=\phi_y^{i,(n)}(t,(W^i,W^0)(\omega)_{t \wedge . })=:\phi^{i,(n)}_y(t,\bar{W}^i),\\
& ( \bar{Z}^{ii,(n)}_t, \bar{Z}^{i0,(n)}_t )(\omega)=\phi_z^{i,(n)} (t,(W^i,W^0)(\omega)_{t \wedge . })=:\phi^{i,(n)}_z(t,\bar{W}^i),    
\end{split}
\ee
where we write $\bar W^i=(W^i, W^0)$.
Similarly, there exist $ \phi_\mu^i, \phi_{\sigma}^i, \phi_{\sigma}^0$, and $\phi_\xi^i,$ such that 
$\xi_{i} = \phi_{\xi}^i(\bar{W}^i )$, and 
$$
 \tilde{\mu}^{i}_t = \phi_{\mu}^i(t,\bar{W}^i ) ,\ \sigma^{i}_t = \phi_{\sigma}^i (t,\bar{W}^i), \ \sigma^{0} = \phi_{\sigma}^0(t,W^0 ). 
$$
Thus, the Picard iteration for equation \eqref{equ:4.1} can be written as 
\be
\begin{split}
\phi_{y}^{i,(n+1)}(t,\bar W^i) &= \phi_{\xi}^{i}(\bar{W}^{i}) + \int_{t}^{T}\hat{f}(\phi_{\mu}^i (r,\bar W^i) ,\phi_{\sigma}^i(r,\bar W^i) ,\phi_{\sigma}^{0}(r, W^0),\psi_{n}^{0}(r,W^{0}),\phi_{z}^{i,(n)}(r,\bar W^i) ) dr  \\
&\ \ \ - \int_{t}^{T}\phi_{z}^{i,(n+1)}(r,\bar W^i) \cdot d \bar{W}^i_r, \label{equ:newiid1}    
\end{split}
\ee
where $\psi_{n}^{0}(t,W^{0})$ is the suitable function that represents the term from the conditional expectations $\mathbb{E}^{0}[...]$ (and is thus $\mathcal{F}^{0}$-predictable) in the $n$-th iteration, and $\hat f$ is deterministic.

Now we show by induction that for every $n \in \mathbb{N}$,
$(\xi_{i},\tilde \mu^i, \sigma^i, \sigma^0, \bar{W}^i,   \bar{Y}^{i,(n)},\bar{Z}^{i,(n)})$, $i=1,...,N$, are conditionally i.i.d. For $n=0$, we have $(\bar{Y},  \bar{Z})^{i,(0)} = (0,0,0)$ for every $i$, which implies that
the conditional expectations in the driver come from coefficient $\bar{\nu}_{10}^{i}$ defined in (\ref{equ:barnu10}). Since $(\tilde{\mu}^{i},\sigma^{i}, \sigma^0)$ are conditionally i.i.d. given $\MF^0$, the aforementioned conditional expectations are identical for all $i$ and $\mathcal{F}^{0}$-measurable, which implies that they do not depend on $i$ and that we can take them into the suitable $\psi_{n}^{0}$. Furthermore, by the fact that $(\xi_i, \tilde \mu^i, \sigma^i, \sigma^0, \bar{W}^i)$, $i=1,...,N$, have the same law, (\ref{equ:newiid1}) implies
\be
\begin{split}
\phi_{y}^{i,(n+1)}(t,\bar W^j) &= \phi_\xi^j(\bar{W}^{j}) + \int_{t}^{T}\hat{f}(\phi_{\mu}^j (r,\bar W^j) ,\phi_{\sigma}^j(r,\bar W^j) ,\phi_{\sigma}^{0}(r, W^0),\psi_{n}^{0}(r,W^{0}),\phi_{z}^{i,(n)}(r,\bar W^j) ) dr  \\
&\ \ \ - \int_{t}^{T}\phi_{z}^{i,(n+1)}(r,\bar W^j) \cdot d \bar{W}^j_r, \label{equ:newiid2}   
\end{split}
\ee
for any $j \in \{1,...,N\}$. Therefore, $(\phi_{y}^{i,(n+1)}(\cdot,\bar{W}^j),\phi_{z}^{i,(n+1)}(\cdot,\bar{W}^j))$ is the unique solution of the $j$-th component of the $(n+1)$-th Picard iteration of mf-qBSDE (\ref{equ:4.1}), which implies by uniqueness that
\[
\trinm{(\phi_{y}^{i,(n+1)}(\cdot,\bar{W}^j),\phi_{z}^{i,(n+1)}(\cdot,\bar{W}^j))-(\phi_{y}^{j,(n+1)}(\cdot,\bar{W}^j),\phi_{z}^{j,(n+1)}(\cdot,\bar{W}^j))}_{\bar\beta} = 0.
\]
Hence, because $\phi_{y}$ and $\phi_{z}$ are deterministic, we can assume without loss of generality that $\phi_{y}^{i,(n+1)}$ and $\phi_{z}^{i,(n+1)}$ do not depend on the index $i$. Thus, we have $(\bar{Y}_{t}^{i,(1)},\bar{Z}_{t}^{i,(1)}) = (\phi_{y}^{(1)}(t,\bar{W}^{i}),\phi_{z}^{(1)}(t,\bar{W}^{i}))$, which implies that $(\xi_{i},\tilde \mu^i, \sigma^i, \sigma^0,  \bar W^i,   \bar{Y}_{t}^{i,(1)},\bar{Z}_{t}^{i,(1)}),\ i=1,...,N,$ are conditionally i.i.d.

Now assume that the desired statement holds for some $n$. Then, to complete the induction, note that the 
conditional expectation terms of $\hat{f}$ come from coefficients $\bar{\nu}_{4}^{i}-\bar{\nu}_{10}^{i}$ defined in (\ref{equ:barnu3}) $-$ (\ref{equ:barnu10}) and from the two terms $\mathbb{E}^{0}[\tfrac{{\sigma}^{i}}{(\bar{\sigma}^{i})^{2}}\bar{Z}^{ii,(n)}]$ and $\mathbb{E}^{0}[\tfrac{{\sigma}^{0}}{(\bar{\sigma}^{i})^{2}}\bar{Z}^{i0,(n)}]$. Similar to the argument of the case $n=0$, the conditional expectations in $\bar{\nu}_{4}^{i}-\bar{\nu}_{10}^{i}$ are identical for every $i$ and $\mathcal{F}^{0}$-measurable. For the remaining two conditional expectations, note that from the induction assumption we have that $(\sigma^i, \sigma^0, \bar{W}^i, \bar{Z}^{i,(n)})$, $i =1,...,N$, are conditionally i.i.d., which implies that those expectations are also identical for every $i$, and $\mathcal{F}^{0}$-measurable. It follows that we can take these terms into $\psi_{n}^{0}$. Therefore, we can proceed with (\ref{equ:newiid2}) analogously to the proof of the induction start to show that $(\bar{Y}_{t}^{i,(n+1)},\bar{Z}_{t}^{i,(n+1)}) = (\phi_{y}^{(n+1)}(t,\bar{W}^{i}),\phi_{z}^{(n+1)}(t,\bar{W}^{i}))$, which completes the induction.

    
\end{proof}





Now we apply the stability result from Theorem \ref{Thm3.9} to show the convergence of the solution of (\ref{equ:qBSDE}) to the solution of (\ref{equ:4.1}). Recall the contraction constants $M_{*}$ and $\tilde{M}_{*}$ from Corollaries \ref{Cor3.3} and \ref{Cor4.2}, respectively. Moreover, define
\begin{align}
R_{*}^{2} &:= e^{\bar{\beta} T} \bR^{2}, \  \ \ 
C^{*} := 2\bC R_{*}(1+\tfrac{1}{\sqrt{T}}) + \bC \sqrt{2}(1+\tfrac{1}{T})\Bigl(\tfrac{8 \bC^{2}}{T}R_{*}^{2} + 8 \bC^{2}R_{*}^{2}\Bigr)^{\tfrac{1}{2}}. \nonumber
\end{align}
\begin{Theorem}
\label{Thm5.3}
Set the constants $\bR$, $\bC$, $\bar{\beta}$ as in Remark \ref{Rem5.1}, and $R_{*} , C^{*}$ as above. Suppose that Assumption \ref{assu:convergmf} holds and $(\bR, \bC, \bar{\beta})$ satisfies \eqref{assu:modelbd} and \eqref{assu:modelbd1}. Moreover, assume that $16C^{*}R_{*} \leq 1$.
Denote by $(Y,Z)$ and $(\bar{Y},\bar{Z})$ the unique solutions of BSDEs \textnormal{(\ref{equ:qBSDE})} and \textnormal{(\ref{equ:4.1})} from Corollaries \ref{Cor3.4} and \ref{Cor4.5}, respectively, and define $\Delta Y^{i} := Y^{i}-\bar{Y}^{i}$ and $\Delta Z^{ij} := Z^{ij} - \bar{Z}^{ij}$ for $i \in \{1,...,N\}$, $j \in \{0,...,N\}$. Then there exist constants $\hC >0$, $\alpha>0,$ independent of $N$ such that
\be
\trinm{(\Delta Y,\Delta Z)}_{\bar{\beta},*} \leq \tfrac{\hC}{(\ln{N})^{\alpha}}.
\ee
\end{Theorem}

\begin{proof}

First, note that we get from Corollaries \ref{Cor3.4}, \ref{Cor4.2}, and \ref{Cor4.5} that BSDEs (\ref{equ:qBSDE}) and (\ref{equ:4.1}) satisfy the conditions from Theorem \ref{Thm3.9}. Therefore, as in \eqref{equ:stab9}, we have
\[
\trinm{(\Delta Y^{(n+1)},\Delta Z^{(n+1)})}_{\bar{\beta},*}^{2} \leq \tfrac{1}{2}D + \tfrac{1}{4} \trinm{(\Delta Y^{(n)},\Delta Z^{(n)})}_{\bar{\beta},*} + \tfrac{1}{4}\trinm{(\Delta Y^{(n)},\Delta Z^{(n)})}_{\bar{\beta},*}^{2},
\]
where $D$, $\Delta Y^{(n)}$, and $\Delta Z^{(n)}$ are defined as in Theorem \ref{Thm3.9},
so that we get
\begin{align}
D &= 8e^{\tfrac{3}{2} \bar{\beta} T} \bR \sup_{n \in \mathbb{N},i \in \{1,...,N\}}\mathbb{E}\Bigl[\int_{0}^{T}|\Delta f_{r}^{(n),i}|dr\Bigr], \nonumber
\end{align}
where
$
\Delta f_{r}^{(n),i} = f_{r}^{i}(\bar{Z}_{r}^{(n)})-\bar{f}_{r}^{i}(\bar{Z}_{r}^{ii,(n)},\bar{Z}_{r}^{i0,(n)},\bar{Z}_{r}^{*,ii,(n)},\bar{Z}_{r}^{*,i0,(n)}).
$
To compute $\mathbb{E}\Bigl[\int_{0}^{T}|\Delta f_{r}^{(n),i}|dr\Bigr]$, we see that
\begin{align}
\Delta f_{r}^{(n),i} &= \bar{Z}_{r}^{ii,(n)}\biggl(\nu_{r,4}^{i}\tfrac{1}{N}\sum_{j=1}^{N}\Bigl(\tfrac{\sigma_{r}^{j}}{(\bar{\sigma}_{r}^{j})^{2}}\bar{Z}_{r}^{jj,(n)} + \tfrac{\sigma_{r}^{0}}{(\bar{\sigma}_{r}^{j})^{2}}\bar{Z}_{r}^{j0,(n)} - \mathbb{E}^{0}\bigl[\tfrac{\sigma_{r}^{j}}{(\bar{\sigma}_{r}^{j})^{2}}\bar{Z}_{r}^{jj,(n)}+\tfrac{\sigma_{r}^{0}}{(\bar{\sigma}_{r}^{j})^{2}}\bar{Z}_{r}^{j0,(n)}\bigr]\Bigr) \nonumber \\
&\ \ \ + \mathbb{E}^{0}\bigl[\tfrac{\sigma_{r}^{i}}{(\bar{\sigma}_{r}^{i})^{2}}\bar{Z}_{r}^{ii,(n)}+\tfrac{\sigma_{r}^{0}}{(\bar{\sigma}_{r}^{i})^{2}}\bar{Z}_{r}^{i0,(n)}\bigr](\nu_{r,4}^{i}-\bar{\nu}_{r,4}^{i})\biggr) + \bar{Z}_{r}^{i0,(n)}\biggl(\nu_{r,5}^{i}\tfrac{1}{N}\sum_{j=1}^{N}\Bigl(\tfrac{\sigma_{r}^{j}}{(\bar{\sigma}_{r}^{j})^{2}}\bar{Z}_{r}^{jj,(n)} + \tfrac{\sigma_{r}^{0}}{(\bar{\sigma}_{r}^{j})^{2}}\bar{Z}_{r}^{j0,(n)} \nonumber \\
&\ \ \ - \mathbb{E}^{0}\bigl[\tfrac{\sigma_{r}^{j}}{(\bar{\sigma}_{r}^{j})^{2}}\bar{Z}_{r}^{jj,(n)}+\tfrac{\sigma_{r}^{0}}{(\bar{\sigma}_{r}^{j})^{2}}\bar{Z}_{r}^{j0,(n)}\bigr]\Bigr) + \mathbb{E}^{0}\bigl[\tfrac{\sigma_{r}^{i}}{(\bar{\sigma}_{r}^{i})^{2}}\bar{Z}_{r}^{ii,(n)}+\tfrac{\sigma_{r}^{0}}{(\bar{\sigma}_{r}^{i})^{2}}\bar{Z}_{r}^{i0,(n)}\bigr](\nu_{r,5}^{i}-\bar{\nu}_{r,5}^{i})\biggr) \nonumber \\
&\ \ \ + \nu_{r,6}^{i}\biggl(\tfrac{1}{N}\sum_{j=1}^{N}\tfrac{\sigma_{r}^{j}}{(\bar{\sigma}_{r}^{j})^{2}}\bar{Z}_{r}^{jj,(n)} + \tfrac{\sigma_{r}^{0}}{(\bar{\sigma}_{r}^{j})^{2}}\bar{Z}_{r}^{j0,(n)} + \mathbb{E}^{0}\bigl[\tfrac{\sigma_{r}^{j}}{(\bar{\sigma}_{r}^{j})^{2}}\bar{Z}_{r}^{jj,(n)}+\tfrac{\sigma_{r}^{0}}{(\bar{\sigma}_{r}^{j})^{2}}\bar{Z}_{r}^{j0,(n)}\bigr]\biggr) \nonumber \\
&\ \ \ \cdot\biggl(\tfrac{1}{N}\sum_{j=1}^{N}\tfrac{\sigma_{r}^{j}}{(\bar{\sigma}_{r}^{j})^{2}}\bar{Z}_{r}^{jj,(n)} + \tfrac{\sigma_{r}^{0}}{(\bar{\sigma}_{r}^{j})^{2}}\bar{Z}_{r}^{j0,(n)} - \mathbb{E}^{0}\bigl[\tfrac{\sigma_{r}^{j}}{(\bar{\sigma}_{r}^{j})^{2}}\bar{Z}_{r}^{jj,(n)}+\tfrac{\sigma_{r}^{0}}{(\bar{\sigma}_{r}^{j})^{2}}\bar{Z}_{r}^{j0,(n)}\bigr]\biggr) \nonumber \\
&\ \ \ + \biggl(\mathbb{E}^{0}\bigl[\tfrac{\sigma_{r}^{j}}{(\bar{\sigma}_{r}^{j})^{2}}\bar{Z}_{r}^{jj,(n)}+\tfrac{\sigma_{r}^{0}}{(\bar{\sigma}_{r}^{j})^{2}}\bar{Z}_{r}^{j0,(n)}\bigr]\biggr)^{2}(\nu_{r,6}^{i}-\bar{\nu}_{r,6}^{i}) + \bar{Z}_{r}^{ii,(n)}(\nu_{r,7}^{i}-\bar{\nu}_{r,7}^{i}) + \bar{Z}_{r}^{i0,(n)}(\nu_{r,8}^{i}-\bar{\nu}_{r,8}^{i}) \nonumber \\
&\ \ \ + \nu_{r,9}^{i}\tfrac{1}{N}\sum_{j=1}^{N}\Bigl(\tfrac{\sigma_{r}^{j}}{(\bar{\sigma}_{r}^{j})^{2}}\bar{Z}_{r}^{jj,(n)} + \tfrac{\sigma_{r}^{0}}{(\bar{\sigma}_{r}^{j})^{2}}\bar{Z}_{r}^{j0,(n)} - \mathbb{E}^{0}\bigl[\tfrac{\sigma_{r}^{j}}{(\bar{\sigma}_{r}^{j})^{2}}\bar{Z}_{r}^{jj,(n)}+\tfrac{\sigma_{r}^{0}}{(\bar{\sigma}_{r}^{j})^{2}}\bar{Z}_{r}^{j0,(n)}\bigr]\Bigr) \nonumber \\
&\ \ \ + \mathbb{E}^{0}\bigl[\tfrac{\sigma_{r}^{j}}{(\bar{\sigma}_{r}^{j})^{2}}\bar{Z}_{r}^{jj,(n)}+\tfrac{\sigma_{r}^{0}}{(\bar{\sigma}_{r}^{j})^{2}}\bar{Z}_{r}^{j0,(n)}\bigr](\nu_{r,9}^{i}-\bar{\nu}_{r,9}^{i}) + (\nu_{r,10}^{i}-\bar{\nu}_{r,10}^{i}). \nonumber 
\end{align}
To simplify notation, we define
\begin{align}
&\lambda_{r,4}^{i} := -\tfrac{\sigma_{r}^{i}c_{r}\tilde{c}_{r}}{\eta(\bar{\sigma}_{r}^{i})^{2}}, \ 
\lambda_{r,5}^{i} := -\tfrac{\sigma_{r}^{0}c_{r}\tilde{c}_{r}}{\eta(\bar{\sigma}_{r}^{i})^{2}}, \ 
\lambda_{r,6}^{i} := -\tfrac{(c_{r}+\tilde{c}_{r})c_{r}\tilde{c}_{r}}{2\eta^{2}(\bar{\sigma}_{r}^{i})^{2}}, \  
\lambda_{r,7}^{i} := -\tfrac{\sigma_{r}^{i}c_{r}\tilde{c}_{r}}{\eta^{2}(\bar{\sigma}_{r}^{i})^{2}}\mathbb{E}^{0}\bigl[\tfrac{\tilde{\mu}_{r}^{i}}{(\bar{\sigma}_{r}^{i})^{2}}\bigr],
\label{equ:def_lambda4} \\
&\lambda_{r,8}^{i} := -\tfrac{\sigma_{r}^{i}c_{r}}{\eta(\bar{\sigma}_{r}^{i})^{2}}, \  
\lambda_{r,9}^{i} := -\tfrac{\sigma_{r}^{0}c_{r}\tilde{c}_{r}}{\eta^{2}(\bar{\sigma}_{r}^{i})^{2}}\mathbb{E}^{0}\bigl[\tfrac{\tilde{\mu}_{r}^{i}}{(\bar{\sigma}_{r}^{i})^{2}}\bigr], \ 
\lambda_{r,10}^{i} := -\tfrac{\sigma_{r}^{0}c_{r}}{\eta(\bar{\sigma}_{r}^{i})^{2}}, 
\label{equ:def_lambda8} \\
&\lambda_{r,11}^{i} := -\tfrac{(c_{r}+\tilde{c}_{r})c_{r}\tilde{c}_{r}}{\eta^{3}(\bar{\sigma}_{r}^{i})^{2}}\mathbb{E}^{0}\bigl[\tfrac{\tilde{\mu}_{r}^{i}}{(\bar{\sigma}_{r}^{i})^{2}}\bigr] - \tfrac{\tilde{\mu}_{r}^{i}c_{r}\tilde{c}_{r}}{\eta^{2}(\bar{\sigma}_{r}^{i})^{2}}, \ 
\lambda_{r,12}^{i} := -\tfrac{c_{r}^{2}}{\eta^{2}(\bar{\sigma}_{r}^{i})^{2}},
\label{equ:def_lambda11} \\
&\lambda_{r,13}^{i} := -\tfrac{(c_{r}+\tilde{c}_{r})c_{r}\tilde{c}_{r}}{2\eta^{4}(\bar{\sigma}_{r}^{i})^{2}}\Bigl(\mathbb{E}^{0}\bigl[\tfrac{\tilde{\mu}_{r}^{i}}{(\bar{\sigma}_{r}^{i})^{2}}\bigr]\Bigr)^{2} - \tfrac{\tilde{\mu}_{r}^{i}c_{r}\tilde{c}_{r}}{\eta^{3}(\bar{\sigma}_{r}^{i})^{2}}\mathbb{E}^{0}\bigl[\tfrac{\tilde{\mu}_{r}^{i}}{(\bar{\sigma}_{r}^{i})^{2}}\bigr], \label{equ:def_lambda13} \\
&\lambda_{r,14}^{i} := -\tfrac{c_{r}^{2}}{2\eta^{3}(\bar{\sigma}_{r}^{i})^{2}}\biggl(\Bigl(\tfrac{1}{N}\sum_{j=1}^{N}\tfrac{\tilde{\mu}_{r}^{j}}{(\bar{\sigma}_{r}^{j})^{2}}\Bigr)+\mathbb{E}^{0}\bigl[\tfrac{\tilde{\mu}_{r}^{i}}{(\bar{\sigma}_{r}^{i})^{2}}\bigr]\biggr) - \tfrac{\tilde{\mu}_{r}^{i}c_{r}}{\eta^{2}(\bar{\sigma}_{r}^{i})^{2}},
\end{align}
which are all uniformly bounded. Moreover, we define
\begin{align}
(LLN)_{r,1} &:= \tfrac{1}{N}\sum_{j=1}^{N}\Bigl((\bar{\sigma}_{r}^{j})^{-2}-\mathbb{E}^{0}[(\bar{\sigma}_{r}^{j})^{-2}]\Bigr), \ (LLN)_{r,2} := \tfrac{1}{N}\sum_{j=1}^{N}\Bigl(\tfrac{\tilde{\mu}_{r}^{j}}{(\bar{\sigma}_{r}^{j})^{2}}-\mathbb{E}^{0}[\tfrac{\tilde{\mu}_{r}^{j}}{(\bar{\sigma}_{r}^{j})^{2}}]\Bigr),  \label{equ:def_LLN1} \\
(LLN)_{r,3} &:= \tfrac{1}{N}\sum_{j=1}^{N}\Bigl(\tfrac{\sigma_{r}^{j}}{(\bar{\sigma}_{r}^{j})^{2}}\bar{Z}_{r}^{jj,(n)}-\mathbb{E}^{0}[\tfrac{\sigma_{r}^{j}}{(\bar{\sigma}_{r}^{j})^{2}}\bar{Z}_{r}^{jj,(n)}]\Bigr), 
\label{equ:def_LLN3} \\
(LLN)_{r,4} &:= \tfrac{1}{N}\sum_{j=1}^{N}\Bigl(\tfrac{\sigma_{r}^{0}}{(\bar{\sigma}_{r}^{j})^{2}}\bar{Z}_{r}^{j0,(n)}-\mathbb{E}^{0}[\tfrac{\sigma_{r}^{0}}{(\bar{\sigma}_{r}^{j})^{2}}\bar{Z}_{r}^{j0,(n)}]\Bigr), \label{equ:def_LLN4}
\end{align}
which are of the same form as in Lemma \ref{Lemma5.2} in view of Lemma \ref{lem:ciid}, and
\begin{align}
\lambda_{r,15}^{i,n} &:= \lambda_{r,11}^{i}\mathbb{E}^{0}\bigl[\tfrac{\sigma_{r}^{i}}{(\bar{\sigma}_{r}^{i})^{2}}\bar{Z}_{r}^{ii,(n)}+\tfrac{\sigma_{r}^{0}}{(\bar{\sigma}_{r}^{i})^{2}}\bar{Z}_{r}^{i0,(n)}\bigr] + \lambda_{r,7}^{i}\bar{Z}_{r}^{ii,(n)} + \lambda_{r,9}^{i}\bar{Z}_{r}^{i0,(n)} + \lambda_{r,13}^{i} \label{equ:def_lambda15}, \\
\lambda_{r,16}^{i,n} &:= \lambda_{r,12}^{i}\mathbb{E}^{0}\bigl[\tfrac{\sigma_{r}^{i}}{(\bar{\sigma}_{r}^{i})^{2}}\bar{Z}_{r}^{ii,(n)}+\tfrac{\sigma_{r}^{0}}{(\bar{\sigma}_{r}^{i})^{2}}\bar{Z}_{r}^{i0,(n)}\bigr] + \lambda_{r,8}^{i}\bar{Z}_{r}^{ii,(n)} + \lambda_{r,10}^{i}\bar{Z}_{r}^{i0,(n)} + \lambda_{r,14}^{i} \label{equ:def_lambda16}, \\
\lambda_{r,17}^{i,n} &:= \tfrac{\nu_{r,6}^{i}}{N}\sum_{j=1}^{N}\Bigl(\tfrac{\sigma_{r}^{j}}{(\bar{\sigma}_{r}^{j})^{2}}\bar{Z}_{r}^{jj,(n)} + \tfrac{\sigma_{r}^{0}}{(\bar{\sigma}_{r}^{j})^{2}}\bar{Z}_{r}^{j0,(n)} + \mathbb{E}^{0}\bigl[\tfrac{\sigma_{r}^{j}}{(\bar{\sigma}_{r}^{j})^{2}}\bar{Z}_{r}^{jj,(n)}+\tfrac{\sigma_{r}^{0}}{(\bar{\sigma}_{r}^{j})^{2}}\bar{Z}_{r}^{j0,(n)}\bigr]\Bigr) \\
&\ \ \ \ + \nu_{r,4}^{i}\bar{Z}_{r}^{ii,(n)} + \nu_{r,5}^{i}\bar{Z}_{r}^{i0,(n)} + \nu_{r,9}^{i}, \nonumber
\end{align}
which all have bounded $\mathcal{Z}_{BMO}^{2}$-norm.
Now, direct computation yields
\begin{align}
\Delta f_{r}^{(n),i} 
&= \biggl(\bigl(\lambda_{r,4}^{i}\bar{Z}_{r}^{ii,(n)} + \lambda_{r,5}^{i}\bar{Z}_{r}^{i0,(n)}\bigr)\mathbb{E}^{0}\bigl[\tfrac{\sigma_{r}^{i}}{(\bar{\sigma}_{r}^{i})^{2}}\bar{Z}_{r}^{ii,(n)}+\tfrac{\sigma_{r}^{0}}{(\bar{\sigma}_{r}^{i})^{2}}\bar{Z}_{r}^{i0,(n)}\bigr] \nonumber \\ &\ \ \ +\lambda_{r,6}^{i}\Bigl(\mathbb{E}^{0}\bigl[\tfrac{\sigma_{r}^{i}}{(\bar{\sigma}_{r}^{i})^{2}}\bar{Z}_{r}^{ii,(n)}+\tfrac{\sigma_{r}^{0}}{(\bar{\sigma}_{r}^{i})^{2}}\bar{Z}_{r}^{i0,(n)}\bigr]\Bigr)^{2}+\lambda_{r,15}^{i,n}\biggr)(LLN)_{r,1} \nonumber \\
&\ \ \ + \lambda_{r,16}^{i,n}(LLN)_{r,2} + \lambda_{r,17}^{i,n}\bigl((LLN)_{r,3} + (LLN)_{r,4}\bigr) \label{equ:diff_f2},
\end{align}
where we apply the fact that $( \sigma^i, \sigma^0,  \bar{Z}^{ii, (n)}, \bar{Z}^{i0, (n)} )$, $i=1,...,N$, are conditionally i.i.d. given $\MF^0_T$.
It follows from (\ref{equ:diff_f2}) that
\begin{align}
&\mathbb{E}\Bigl[\int_{0}^{T}|\Delta f_{r}^{(n),i}|dr\Bigr] \nonumber \\
&\leq \mathbb{E}\Bigl[\int_{0}^{T}\Big|\Bigl(\lambda_{r,4}^{i}\bar{Z}_{r}^{ii,(n)} + \lambda_{r,5}^{i}\bar{Z}_{r}^{i0,(n)}\Bigr)\mathbb{E}^{0}\bigl[\tfrac{\sigma_{r}^{i}}{(\bar{\sigma}_{r}^{i})^{2}}\bar{Z}_{r}^{ii,(n)}+\tfrac{\sigma_{r}^{0}}{(\bar{\sigma}_{r}^{i})^{2}}\bar{Z}_{r}^{i0,(n)}\bigr](LLN)_{r,1}\Big|dr\Bigr] \nonumber \\
&\ \ \ +\mathbb{E}\Bigl[\int_{0}^{T}\Big|\lambda_{r,6}^{i}\Bigl(\mathbb{E}^{0}\bigl[\tfrac{\sigma_{r}^{i}}{(\bar{\sigma}_{r}^{i})^{2}}\bar{Z}_{r}^{ii,(n)}+\tfrac{\sigma_{r}^{0}}{(\bar{\sigma}_{r}^{i})^{2}}\bar{Z}_{r}^{i0,(n)}\bigr]\Bigr)^{2}(LLN)_{r,1}\Big|dr\Bigr] + \mathbb{E}\Bigl[\int_{0}^{T}\big|\lambda_{r,15}^{i,n}(LLN)_{r,1}\big|dr\Bigr] \nonumber \\
&\ \ \ + \mathbb{E}\Bigl[\int_{0}^{T}\big|\lambda_{r,16}^{i,n}(LLN)_{r,2}\big|dr\Bigr] +\mathbb{E}\Bigl[\int_{0}^{T}\big|\lambda_{r,17}^{i,n}(LLN)_{r,3}\big|dr\Bigr] + \mathbb{E}\Bigl[\int_{0}^{T}\big|\lambda_{r,17}^{i,n}(LLN)_{r,4}\big|dr\Bigr] \nonumber \\
&\leq \tfrac{\bar{\delta}_{\sigma}^{2}}{\eta\delta_{\sigma}^{6}}\biggl(\mathbb{E}\Bigl[\int_{0}^{T}\Bigl(|\bar{Z}_{r}^{ii,(n)}| + |\bar{Z}_{r}^{i0,(n)}|\Bigr)^{2}dr\Bigr]\biggr)^{\tfrac{1}{2}}\biggl(\mathbb{E}\Bigl[\int_{0}^{T}\Bigl(\mathbb{E}^{0}\bigl[|\bar{Z}_{r}^{ii,(n)}|+|\bar{Z}_{r}^{i0,(n)}|\bigr]\Bigr)^{2}(LLN)_{r,1}^{2}dr\Bigr]\biggr)^{\tfrac{1}{2}} \nonumber \\
&\ \ \ +\tfrac{\bar{\delta}_{\sigma}^{2}}{\eta^{2}\delta_{\sigma}^{9}}\biggl(\mathbb{E}\Bigl[\int_{0}^{T}\Bigl(\mathbb{E}^{0}\bigl[|\bar{Z}_{r}^{ii,(n)}|+|\bar{Z}_{r}^{i0,(n)}|\bigr]\Bigr)^{2}dr\Bigr]\biggr)^{\tfrac{1}{2}}\biggl(\mathbb{E}\Bigl[\int_{0}^{T}\Bigl(\mathbb{E}^{0}\bigl[|\bar{Z}_{r}^{ii,(n)}|+|\bar{Z}_{r}^{i0,(n)}|\bigr]\Bigr)^{2}(LLN)_{r,1}^{2}dr\Bigr]\biggr)^{\tfrac{1}{2}} \nonumber \\
&\ \ \ +\biggl(\mathbb{E}\Bigl[\int_{0}^{T}(\lambda_{r,15}^{i,n})^{2}dr\Bigr]\biggr)^{\tfrac{1}{2}}\biggl(\mathbb{E}\Bigl[\int_{0}^{T}(LLN)_{r,1}^{2}\Bigr]\biggr)^{\tfrac{1}{2}} + \biggl(\mathbb{E}\Bigl[\int_{0}^{T}(\lambda_{r,16}^{i,n})^{2}dr\Bigr]\biggr)^{\tfrac{1}{2}}\biggl(\mathbb{E}\Bigl[\int_{0}^{T}(LLN)_{r,2}^{2}\Bigr]\biggr)^{\tfrac{1}{2}} \nonumber \\
&\ \ \ + \biggl(\mathbb{E}\Bigl[\int_{0}^{T}(\lambda_{r,17}^{i,n})^{2}dr\Bigr]\biggr)^{\tfrac{1}{2}}\biggl(\mathbb{E}\Bigl[\int_{0}^{T}(LLN)_{r,3}^{2}\Bigr]\biggr)^{\tfrac{1}{2}} + \biggl(\mathbb{E}\Bigl[\int_{0}^{T}(\lambda_{r,17}^{i,n})^{2}dr\Bigr]\biggr)^{\tfrac{1}{2}}\biggl(\mathbb{E}\Bigl[\int_{0}^{T}(LLN)_{r,4}^{2}\Bigr]\biggr)^{\tfrac{1}{2}}, \label{equ:5.13}
\end{align}
where we used $\max\{|\lambda_{r,4}^{i}|,|\lambda_{r,5}^{i}|\} \leq \tfrac{\bar{\delta}_{\sigma}}{\eta\delta_{\sigma}^{4}}$ and $|\lambda_{r,6}^{i}| \leq \tfrac{1}{\eta^{2}\delta_{\sigma}^{5}}$ for the last inequality.
Note further that
\begin{align}
\max\Biggl\{\biggl(\mathbb{E}\Bigl[\int_{0}^{T}\Bigl(|\bar{Z}_{r}^{ii,(n)}| + |\bar{Z}_{r}^{i0,(n)}|\Bigr)^{2}dr\Bigr]\biggr)^{\tfrac{1}{2}} \ , \ \biggl(\mathbb{E}\Bigl[\int_{0}^{T}\Bigl(\mathbb{E}^{0}\bigl[\bar{Z}_{r}^{ii,(n)}+\bar{Z}_{r}^{i0,(n)}\bigr]\Bigr)^{2}dr\Bigr]\biggr)^{\tfrac{1}{2}}\Biggr\} \leq 2R_{*}. \label{equ:5.est1}
\end{align}
Similar to the proof of Lemma \ref{Lemma5.2}, since $(LLN)_{1}$ is bounded and $\bar{\sigma}^{j}$, $j=1,...,N,$ are conditionally i.i.d. given $\mathcal{F}_{T}^{0}$, we get
\begin{align}
&\mathbb{E}\Bigl[\int_{0}^{T}\Bigl(\mathbb{E}^{0}\bigl[|\bar{Z}_{r}^{ii,(n)}|+|\bar{Z}_{r}^{i0,(n)}|\bigr]\Bigr)^{2}(LLN)_{r,1}^{2}dr\Bigr] \nonumber \\
&\leq \tfrac{1}{N^{2}}\sum_{j=1}^{N}\mathbb{E}\Bigl[\int_{0}^{T}\Bigl(\mathbb{E}^{0}\bigl[|\bar{Z}_{r}^{ii,(n)}|+|\bar{Z}_{r}^{i0,(n)}|\bigr]\Bigr)^{2}\Bigl((\bar{\sigma}_{r}^{j})^{-2}-\mathbb{E}^{0}[(\bar{\sigma}_{r}^{j})^{-2}]\Bigr)^{2}dr\Bigr] \nonumber \\
&\leq \tfrac{4}{\delta_{\sigma}^{4}N}\mathbb{E}\Bigl[\int_{0}^{T}\Bigl(\mathbb{E}^{0}\bigl[|\bar{Z}_{r}^{ii,(n)}|+|\bar{Z}_{r}^{i0,(n)}|\bigr]\Bigr)^{2}dr\Bigr] 
\leq \tfrac{16R_{*}^{2}}{\delta_{\sigma}^{4}N}. \label{equ:5.est2}
\end{align}
On the other hand, by Assumption \ref{assu:convergmf}, we have the corresponding uniform bounds for processes $\lambda_{r,7}^{i}-\lambda_{r,14}^{i}$ and thus, in view of the $\mathcal{Z}_{BMO}^{2}$-bounds of $\bar{Z}^{ii,(n)}$ and $\bar{Z}^{i0,(n)}$, we have
\begin{align}
\mathbb{E}\Bigl[\int_{0}^{T}(\lambda_{r,15}^{i,n})^{2}dr\Bigr] &\leq 4R_{*}^{2}\biggl(\tfrac{4\bar{\delta}_{\sigma}^{2}}{\delta_{\sigma}^{4}}\Bigl(\tfrac{2\delta_{\mu}}{\eta^{3}\delta_{\sigma}^{7}} + \tfrac{\delta_{\mu}}{\eta^{2}\delta_{\sigma}^{4}}\Bigr)^{2} + \tfrac{2\delta_{\mu}^{2}\bar{\delta}_{\sigma}^{2}}{\eta^{4}\delta_{\sigma}^{12}}\biggr) + 4T\Bigl(\tfrac{\delta_{\mu}^{2}}{\eta^{4}\delta_{\sigma}^{9}} + \tfrac{\delta_{\mu}^{2}}{\eta^{3}\delta_{\sigma}^{6}}\Bigr)^{2} \label{equ:estlambda15} \\
\mathbb{E}\Bigl[\int_{0}^{T}(\lambda_{r,16}^{i,n})^{2}dr\Bigr] &\leq 4R_{*}^{2}\Bigl(\tfrac{4\bar{\delta}_{\sigma}^{2}}{\eta^{4}\delta_{\sigma}^{12}} + \tfrac{2\bar{\delta}_{\sigma}^{2}}{\eta^{2}\delta_{\sigma}^{6}}\Bigr) + 4T\Bigl(\tfrac{\delta_{\mu}}{\eta^{3}\delta_{\sigma}^{6}} + \tfrac{\delta_{\mu}}{\eta^{2}\delta_{\sigma}^{2}}\Bigr)^{2} \label{equ:estlambda16} \\
\mathbb{E}\Bigl[\int_{0}^{T}(\lambda_{r,17}^{i,n})^{2}dr\Bigr] &\leq 4R_{*}^{2}\Bigl(2b^{2}+\tfrac{16b^{2}\bar{\delta}_{\sigma}^{2}}{\delta_{\sigma}^{4}}\Bigr) + 4Tb_{\mu}^{2}. \label{equ:estlambda17}
\end{align}
Furthermore, Lemma \ref{Lemma5.2} yields
\begin{align}
\mathbb{E}\Bigl[\int_{0}^{T}(LLN)_{r,1}^{2}dr\Bigr] &\leq \tfrac{16T^{2}}{\delta_{\sigma}^{4}N}, \ \ \ 
\mathbb{E}\Bigl[\int_{0}^{T}(LLN)_{r,2}^{2}dr\Bigr] \leq \tfrac{16T^{2}\delta_{\mu}^{2}}{\delta_{\sigma}^{4}N}, \label{equ:estLLN2} \\
\mathbb{E}\Bigl[\int_{0}^{T}(LLN)_{r,3}^{2}dr\Bigr] &\leq \tfrac{4\bar{\delta}_{\sigma}^{2}R_{*}^{2}}{\delta_{\sigma}^{4}N}, \ \ \  
\mathbb{E}\Bigl[\int_{0}^{T}(LLN)_{r,4}^{2}dr\Bigr] \leq \tfrac{4\bar{\delta}_{\sigma}^{2}R_{*}^{2}}{\delta_{\sigma}^{4}N} \label{equ:estLLN4}.
\end{align}
Now, we put (\ref{equ:5.est1}) $-$ (\ref{equ:estLLN4}) into (\ref{equ:5.13}) to get
\begin{align}
&\mathbb{E}\Bigl[\int_{0}^{T}|\Delta f_{r}^{(n),i}|dr\Bigr] \nonumber \\
&\leq \Biggl(\tfrac{8R_{*}^{2}\bar{\delta}_{\sigma}^{2}}{\eta \delta_{\sigma}^{8}} + \tfrac{8R_{*}^{2}\bar{\delta}_{\sigma}^{2}}{\eta \delta_{\sigma}^{11}} + \tfrac{4T}{\delta_{\sigma}^{2}}\sqrt{4R_{*}^{2}\biggl(\tfrac{4\bar{\delta}_{\sigma}^{2}}{\delta_{\sigma}^{4}}\Bigl(\tfrac{2\delta_{\mu}}{\eta^{3}\delta_{\sigma}^{7}} + \tfrac{\delta_{\mu}}{\eta^{2}\delta_{\sigma}^{4}}\Bigr)^{2} + \tfrac{2\delta_{\mu}^{2}\bar{\delta}_{\sigma}^{2}}{\eta^{4}\delta_{\sigma}^{12}}\biggr) + 4T\Bigl(\tfrac{\delta_{\mu}^{2}}{\eta^{4}\delta_{\sigma}^{9}} + \tfrac{\delta_{\mu}^{2}}{\eta^{3}\delta_{\sigma}^{6}}\Bigr)^{2}} \nonumber \\
&\ \ \ + \tfrac{4T\delta_{\mu}}{\delta_{\sigma}^{2}}\sqrt{4R_{*}^{2}\Bigl(\tfrac{4\bar{\delta}_{\sigma}^{2}}{\eta^{4}\delta_{\sigma}^{12}} + \tfrac{2\bar{\delta}_{\sigma}^{2}}{\eta^{2}\delta_{\sigma}^{6}}\Bigr) + 4T\Bigl(\tfrac{\delta_{\mu}}{\eta^{3}\delta_{\sigma}^{6}} + \tfrac{\delta_{\mu}}{\eta^{2}\delta_{\sigma}^{2}}\Bigr)^{2}} + \tfrac{4\bar{\delta}_{\sigma}R_{*}}{\delta_{\sigma}^{2}}\sqrt{4R_{*}^{2}\Bigl(2b^{2}+\tfrac{16b^{2}\bar{\delta}_{\sigma}^{2}}{\delta_{\sigma}^{4}}\Bigr) + 4Tb_{\mu}^{2}} \Biggr)N^{-\tfrac{1}{2}} \nonumber \\
&=: C_{1}N^{-\tfrac{1}{2}}, \label{equ:estdiff_2}
\end{align}
which directly implies
\begin{equation}
\label{equ:5.estD}
D \leq 8C_{1}e^{\tfrac{3}{2} \bar{\beta} T} \bR N^{-\tfrac{1}{2}}.
\end{equation}
Therefore, for a sufficiently large $N$, we have $D \leq 1$, which, together with Theorem \ref{Thm3.9}, implies
\begin{align}\nonumber
\trinm{(\Delta Y,\Delta Z)}_{\bar{\beta},*} &\leq \max\{1,\sqrt{T}\}(\hat{C}M_{*}^{n} + \tilde{C}\tilde{M}_{*}^{n}) + D^{2^{-n}}\\
&\leq \max\{1,\sqrt{T}\}(\hat{C}M_{*}^{n} + \tilde{C}\tilde{M}_{*}^{n}) + N^{-\tfrac{1}{2^{n+1}}}, \label{equ:5.36}
\end{align}
where $M_{*}$ and $\tilde{M}_{*}$ are the contraction constants from Corollaries \ref{Cor3.3} and \ref{Cor4.2}, respectively, and $\hat{C}, \tilde C$ can be chosen as
$
\hat{C} := \tfrac{\bR}{1-M_{*}}e^{\tfrac{1}{2} \bar{\beta} T},\   \tilde{C} := \tfrac{ \bR}{1-\tilde{M}_{*}}e^{\tfrac{1}{2} \bar{\beta} T}.
$
Therefore, by setting $C_{3} := \max\{1,\sqrt{T}\}(\hat{C} + \tilde{C})$ and $\bar{M} := \max\{M_{*},\tilde{M}_{*}\}$, (\ref{equ:5.36}) becomes
\begin{equation}
\label{equ:5.37}
\trinm{ (\Delta Y,\Delta Z)}_{\bar{\beta},*} \leq C_{3}\bar{M}^{n} + N^{-\tfrac{1}{2^{n+1}}}.
\end{equation}
Now, if we take any 
$\delta \in (0,1)$ and choose $n = \max\bigl\{0 , \ \lfloor \delta \frac{  \ln \ln N}{ \ln{2}}
\rfloor -1\bigr\}$, we get from (\ref{equ:5.37}) that for a sufficiently large $N$
\begin{align}
\trinm{(\Delta Y,\Delta Z)}_{\bar{\beta},*} &\leq C_{3}\bar{M}^{-2}\bar{M}^{\log_{2}\bigl((\ln{N})^{\delta}\bigr)} + N^{-(\ln{N})^{-\delta}} \nonumber \\
&= C_{3}\bar{M}^{-2}\exp\bigl(\delta(\ln{\bar{M}})\log_{2}(\ln{N})\bigr) + \exp\bigl(-(\ln{N})^{1-\delta}) \nonumber \\
&\leq C_{3}\bar{M}^{-2}\exp\Bigl(-\tfrac{\delta |\ln{\bar{M}}|}{\ln{2}}\ln{\ln{N}}\Bigr) + C_4 \exp\bigl(-\tfrac{\delta |\ln{\bar{M}}|}{\ln{2}}\ln{\ln{N}}\bigr), \label{equ:5.17}
\end{align}
where inequality (\ref{equ:5.17}) follows from the fact that
$(\ln{N})^{1-\delta} > \tfrac{\delta |\ln{\bar{M}}|}{\ln{2}} \ln{\ln{N}}$ for a sufficiently large $N$ and the constant $C_{4}>0$ depends on the chosen $\delta$. Therefore, our conclusion follows with $\alpha = \tfrac{\delta |\ln{\bar{M}}|}{\ln{2}}.$


\end{proof}
\begin{Remark}\label{rem:logspeed}
By refining the final steps of the proof above, one can in fact obtain the convergence rate
$
(
\frac{\ln \ln N}{\ln N}
)^{
\frac{|\ln \bar{M}|}{\ln 2}
}.
$
However, the main obstruction to obtaining a polynomial convergence rate stems from the presence of idiosyncratic noise together with common noise. Indeed, if either $\sigma^{i}$ or $\sigma^{0}$ is equal to zero, the limiting mean-field qBSDE \eqref{equ:4.1} only contains the quadratic terms in form of \(Z\mathbb{E}^{0}[Z]\) and \((\mathbb{E}^{0}[Z])^{2}\). In that case, by applying a classical approach and \eqref{equ:5.estD}, one can recover the convergence rate $N^{-\tfrac{1}{2}}$.
\end{Remark}

\begin{Corollary}
\label{Cor5.4}
Let $\hat{\Balpha}$ and $\bar{\Balpha}$ be defined as in \eqref{eq:fop} and \eqref{equ:DefAlphaBar} and let the assumptions from Theorem \ref{Thm5.3} hold. Then there exists a constant $C^{**}>0 $ 
such that
\[
\sup_{i\in \{1,...,N\}}||\hat{\alpha}^{i}-\bar{\alpha}^{i}||_{\mathcal{H}^{2}} \leq \tfrac{C^{**}}{(\ln{N})^{\alpha}}.
\]
\end{Corollary}
\begin{proof}

We start by defining the control
\begin{equation}
\label{equ:DefAlphaTilde}
\tilde{\alpha}_{t}^{i} := \tfrac{\sigma_{t}^{i}}{(\bar{\sigma}_{t}^{i})^{2}}\bar{Z}_{t}^{ii} + \tfrac{\sigma_{t}^{0}}{(\bar{\sigma}_{t}^{i})^{2}}\bar{Z}_{t}^{i0} + \tfrac{c_{t}}{\eta(\bar{\sigma}_{t}^{i})^{2}}\Bigl(\tfrac{1}{N}\sum_{j=1}^{N}\tfrac{\sigma_{t}^{j}\bar{Z}_{t}^{jj}+\sigma_{t}^{0}\bar{Z}_{t}^{j0}}{(\bar{\sigma}_{t}^{j})^{2}}\Bigr) -\tfrac{\nu_{t,9}^{i}}{c_{t}},
\end{equation}
where we recall that $\nu^i_{t,9}$ is given by \eqref{equ:nu9}. Then it follows from Theorem \ref{Thm5.3} that
\begin{align*}
||\hat{\alpha}^{i}-\tilde{\alpha}^{i}||_{\mathcal{H}^{2}} &= \Bigl|\Bigl|\tfrac{\sigma_{t}^{i}}{(\bar{\sigma}_{t}^{i})^{2}}\Delta{Z}_{t}^{ii} + \tfrac{\sigma_{t}^{0}}{(\bar{\sigma}_{t}^{i})^{2}}\Delta{Z}_{t}^{i0} + \tfrac{c_{t}}{\eta(\bar{\sigma}_{t}^{i})^{2}}\Bigl(\tfrac{1}{N}\sum_{j=1}^{N}\tfrac{\sigma_{t}^{j}\Delta{Z}_{t}^{jj}+\sigma_{t}^{0}\Delta{Z}_{t}^{j0}}{(\bar{\sigma}_{t}^{j})^{2}}\Bigr)\Bigr|\Bigr|_{\mathcal{H}^{2}} \\
&\leq \tfrac{\bar{\delta}_{\sigma}}{\delta_{\sigma}^{2}}||\Delta Z^{ii}||_{\mathcal{H}^{2}} + \tfrac{\bar{\delta}_{\sigma}}{\delta_{\sigma}^{2}}||\Delta Z^{i0}||_{\mathcal{H}^{2}} + \tfrac{\bar{\delta}_{\sigma}}{\eta\delta_{\sigma}^{5}} \Bigl|\Bigl|\tfrac{1}{N}\sum_{j=1}^{N}(|\Delta Z^{jj}|+|\Delta Z^{j0}| )\Bigr|\Bigr|_{\mathcal{H}^{2}} \\
&\leq \bigl(\tfrac{2\bar{\delta}_{\sigma}}{\delta_{\sigma}^{2}}+\tfrac{2\bar{\delta}_{\sigma}}{\eta\delta_{\sigma}^{5}}\bigr)\tfrac{\hC}{(\ln{N})^{\alpha}}.
\end{align*}
Furthermore, we can compute
\begin{align*}
&||\tilde{\alpha}^{i}-\bar{\alpha}^{i}||_{\mathcal{H}^{2}} \\
&= \Big|\Big|\Bigl(\tfrac{c\tilde{c}}{\eta^{2}(\bar{\sigma}^{i})^{2}}\mathbb{E}^{0}\Bigl[\tfrac{\sigma^{i}\bar{Z}^{ii}+\sigma^{0}\bar{Z}^{i0}}{(\bar{\sigma}^{i})^{2}}\Bigr]+\tfrac{\bar{\nu}_{9}^{i}}{\eta}-\tfrac{\lambda_{11}^{i}}{c}\Bigr)(LLN)_{1} - \tfrac{\lambda_{12}^{i}}{c}(LLN)_{2} + \tfrac{c}{\eta(\bar{\sigma}^{i})^{2}}(LLN)_{3} + \tfrac{c}{\eta(\bar{\sigma}^{i})^{2}}(LLN)_{4}\Big|\Big|_{\mathcal{H}^{2}} \\
&\leq \tfrac{\bar{\delta}_{\sigma}}{\eta^{2}\delta_{\sigma}^{6}}\Big|\Big|\mathbb{E}^{0}\Bigl[\bar{Z}^{ii}+\bar{Z}^{i0}\Bigr](LLN)_{1}\Big|\Big|_{\mathcal{H}^{2}} + (\tfrac{b_{\mu}}{\eta}+\tfrac{2\delta_{\mu}}{\eta^{2}\delta_{\sigma}^{8}}+\tfrac{\delta_{\mu}}{\eta^{2}\delta_{\sigma}^{4}})||(LLN)_{1}||_{\mathcal{H}^{2}} + \tfrac{1}{\eta^{2}\delta_{\sigma}^{5}}||(LLN)_{2}||_{\mathcal{H}^{2}}\\
&\ \ \ + \tfrac{1}{\eta\delta_{\sigma}^{3}}||(LLN)_{3}||_{\mathcal{H}^{2}} + \tfrac{1}{\eta\delta_{\sigma}^{3}}||(LLN)_{4}||_{\mathcal{H}^{2}},
\end{align*}
where $\lambda_{11}^i$ and $\lambda_{12}^i$ are defined in (\ref{equ:def_lambda11}), and $(LLN)_{1}-(LLN)_{4}$ in (\ref{equ:def_LLN1}) $-$ (\ref{equ:def_LLN4}), respectively (without the additional index $n$). Now we can estimate as in (\ref{equ:5.est2}) and (\ref{equ:estLLN2}) $-$ (\ref{equ:estLLN4}) to get
\begin{align}
||\tilde{\alpha}^{i}-\bar{\alpha}^{i}||_{\mathcal{H}^{2}} &\leq \Bigl(\tfrac{4R_{*}\bar{\delta}_{\sigma}}{\eta^{2}\delta_{\sigma}^{8}} + \tfrac{4T}{\delta_{\sigma}^{2}}(\tfrac{b_{\mu}}{\eta}+\tfrac{2\delta_{\mu}}{\eta^{2}\delta_{\sigma}^{8}}+\tfrac{\delta_{\mu}}{\eta^{2}\delta_{\sigma}^{4}}) + \tfrac{4T\delta_{\mu}}{\eta^{2}\delta_{\sigma}^{7}} + \tfrac{4R_{*}\bar{\delta}_{\sigma}}{\eta\delta_{\sigma}^{5}}\Bigr)N^{-\tfrac{1}{2}}. \label{equ:convtilde}
\end{align}
Putting everything together, we get
$
||\hat{\alpha}^{i}-\bar{\alpha}^{i}||_{\mathcal{H}^{2}} \leq \tfrac{C^{**}}{(\ln{N})^{\alpha}}.
$

\end{proof}

\subsection{The Radner equilibrium model}
\label{sec:4.2}
In this subsection, we consider a mean-field pricing game and derive the corresponding Radner equilibrium. It turns out that the analysis is largely analogous to that of the price impact model in the last subsection, and thus we only briefly present the model and show that the resulting equilibrium characterization leads to the same class of qBSDEs studied above.

Motivated by \cite{RefN10} and \cite{Kardaras2022}, suppose the market has a dynamically traded stock with price given by
\be
\label{equ:stoch-liq}
dS_{t}^{ \Balpha} = S_{t}^{ \Balpha}( \nu_{t}^{\Balpha} dt+\sigma_{t}^0 dW_{t}^0  ), \ S^{\Balpha}_0=s_0,
\ee
where $s_0$ is a constant, $ \sigma^0$ is a fixed progressively measurable process with respect to $\MF^0$ that satisfies $\sigma^{0} \geq \delta_{\sigma}$ for some $\delta_{\sigma} > 0$, and $\nu_{t}^{\Balpha} $ is the price drift to be determined. Thus, $\frac{\nu^{\Balpha}_t}{\sigma_t^0}$ is the price kernel and $d \mathbb Q:= \frac{\nu^{\Balpha}_T}{\sigma_T^0} d \mathbb{P}$ is the equilibrium pricing measure.

Now suppose there are $N$ agents, and agent $i$ invests $\beta^i$ in $S^{ \Balpha}$ with initial wealth $x^i$, where $\beta^i \in \MA$ with $\MA$ given by \eqref{equ:2.12}. Thus, its wealth is 
\begin{equation}
\label{equ:welth-liq}
X_{t}^{  \Balpha, i}(\beta^i) = x^i + \int_{0}^{t}\beta_{r}^{i}\tfrac{dS_{r}^{ \Balpha}}{S_{r}^{ \Balpha}} = x^i + \int_{0}^{t}{\beta_{r}^{i} \nu_{r}^{ \Balpha} dr} + \int_{0}^{t}{\beta_{r}^{i}\sigma_{r}^0 dW_{r}^0 } , \ t \in [0,T].
\end{equation}
Each agent faces an endowment of \(\xi_i\) at the terminal time \(T\), and has its own idiosyncratic risk, characterized by a Brownian noise $W^i$, such that $\{W^i\}_{i=0}^N$ are independent and $\xi_i \in \MF^{i,0}_T.$ Then each agent seeks to maximize the expected exponential utility
\begin{equation}
\label{eq:uti-liq}
I_{i}^{N}(\Balpha, \beta^i)
:=
\mathbb{E}
\bigl[
U(X_{T}^{\Balpha,i}(\beta^i)-\xi_i)
\bigr],
\qquad
U(x):=-\exp(-\eta x),
\quad
\eta>0,
\end{equation}
subject to the market-clearing condition
\begin{equation}
\label{equ:liqui-cond.}
\sum_{i=1}^{N}\beta^i = 0.
\end{equation}


\begin{Definition}[Radner equilibrium]
A pair $(\nu^{ \hat{\Balpha}},\hat{\Balpha}) \in \mathcal{H}^{2} \times \mathcal{A}^{N}$ with $\hat{\Balpha}=\{\halpha^i \}_{i=1}^N$, is called Radner equilibrium for the $N$-agent game if for every $i \in \{1,...,N\}$, with the given $\nu^{ \hat{\Balpha}}$, $\hat{\alpha}^{i}$ maximizes the utility functional of agent $i$ and the market-clearing conditions hold, i.e. 
$$
I^N_i(\hBalpha, \halpha^i)= \sup_{\alpha \in \MA}I^N_i(\hBalpha, \alpha), \ \ \ \sum_{j=1}^{N}\hat{\alpha}^{j} = 0.
$$

\end{Definition}

We obtain the Radner equilibrium via three steps:\\
1. Fix some $\nu^{ \Balpha}$ in $\mathcal{H}^{2}$; \\
2. Maximize the cost functional of each player and obtain optimal controls $\hat{\alpha}^{i}$ for $i \in \{1,...,N\}$; \\
3. Choose $\nu^{\Balpha}$ from step 1 in such a way that the market-clearing condition holds, i.e., that $\sum_{j=1}^{N}\hat{\alpha}^{j} = 0$.

\begin{Remark}

\textnormal{(i)}. In the above model, for simplicity of calculation and our mean-field consideration, we take the risk aversion parameter $\eta_i$ for each agent to be the same. However, the general case of the multi-dimensional model and its Radner equilibria follow similarly as in \cite{Kardaras2022}. Unlike in \cite{Kardaras2022}, we assume that each agent has its own idiosyncratic risk, which leads to the formulation of systems of qBSDEs with multiple noise sources, in contrast to the two-risk case considered in  \cite{Kardaras2022}.\\
\textnormal{(ii)}. A crucial assumption from our argument is again the non-degeneracy of $\sigma^0,$ in contrast to the discrete model considered in e.g. \cite{CHKP16}.

\end{Remark}


Because of the similarities to the model from Section \ref{sec2.1}, we can proceed analogously to Section \ref{sec2.2} by applying the martingale optimization approach until we reach equation (\ref{eq:fop}) with $(\tilde{\mu}^i,\sigma^i)=0$ and $\mu^{N,\hBalpha}= \nu^{\hBalpha}.$ Thus, in this model, we have 
\be\label{eq:fop2}
f_{t}^{i}(z) = \tfrac{\eta}{2}\sum_{j=0}^{N}(z^{ij})^{2} - \tfrac{\eta}{2( {\sigma}_{t}^{0})^{2}}\bigl(  \sigma_{t}^{0}z^{i0} + \tfrac{\nu_{t}^{ \hBalpha} }{\eta}\bigr)^{2}, 
 \  \text{ and } \ \ 
 \hat{\alpha}_{t}^{i} =  \tfrac{1}{ {\sigma}_{t}^{0} }Z_{t}^{i0} + \tfrac{\nu_{t}^{ \hBalpha} }{\eta( {\sigma}_{t}^{0})^{2}},
\ee
where $f$ represents the driver of the respective BSDE and $Z^{i0}$ is part of the solution. Since we have the condition $\sum_{j=1}^{N}\hat{\alpha}^{j} = 0$, we can take the arithmetic mean on both sides of the second equation of (\ref{eq:fop2}) to get
\[
0 = \tfrac{1}{N}\sum_{j=1}^{N} (\tfrac{1}{ {\sigma}_{t}^{0} }Z_{t}^{j0} ) + \tfrac{\nu_{t}^{ \hBalpha} }{\eta( {\sigma}_{t}^{0})^{2}},
\]
which implies
$
\nu_{t}^{\hat{\Balpha}} = -\eta \, \sigma^0_t \tfrac{1}{N}\sum_{j=1}^{N}  Z_{t}^{j0} .
$
Now, we can plug this representation of $\nu_{t}^{ \hat{\Balpha}}$ into the driver $f$ from (\ref{eq:fop2}), which yields a quadratic BSDE that is similar to the one we get in Section \ref{sec2.2} in the sense that, under similar assumptions, they satisfy the same growth conditions that are needed to obtain well-posedness (see Theorem \ref{Thm3.1}). Therefore, the BSDE that we obtain has a unique solution, which then directly yields the desired pair $(\nu^{\hat{\Balpha}},\hat{\Balpha})$.

Since the aforementioned conditions for well-posedness are dimension-free, we get an equilibrium for every $N$, which means that we can look at the corresponding mean-field game if $N$ goes to infinity by replacing every arithmetic mean by the conditional expectation $\mathbb{E}^{0}[...]$ and then proceeding as in Section \ref{sec4.1}. Because of the similar structure of the BSDEs, we can apply the stability result from Theorem \ref{Thm3.9} and proceed analogously to Corollary \ref{Cor4.5} and Theorem \ref{Thm5.3} to obtain the convergence of equilibria from Corollary \ref{Cor5.4}.

\begin{Remark}
In \cite{RefN10} a similar model is examined with multidimensional price dynamics. The model is only examined in the mean-field sense, meaning that the authors only examine the respective mean-field BSDEs and then show that the resulting optimal controls (in this paper the $\bar{\alpha}^{j}$) satisfy the market-clearing condition in an asymptotic sense, i.e., that $\tfrac{1}{N}\sum_{j=1}^{N}\bar{\alpha}^{j}$ converges to zero in $\mathcal{H}^{2}$ (see \cite[Theorem 5.1]{RefN10}). 
Indeed, their results are consistent with ours in two respects. First, the asymptotic market-clearing condition established in their work follows directly from our convergence of equilibria; see Corollary \ref{Cor5.4}. Second, their convergence rate of order \(N^{-1/2}\) can be recovered from our framework by Remark \ref{rem:logspeed} with \(\sigma^i=0\).

\end{Remark}

\bibliographystyle{abbrv}
\bibliography{citations.bib}

\end{document}